\documentclass[10pt, a4paper, reqno]{amsart} 

\usepackage{mathtools,amsmath,amssymb,amsthm}
\usepackage[a4paper, margin=0.7in]{geometry}
\usepackage[pagewise]{lineno}
\usepackage[doublespacing]{setspace}

\usepackage[compact]{titlesec}         

\titlespacing{\section}{0pt}{0pt}{0pt} 

\AtBeginDocument{
  \setlength\abovedisplayskip{0pt}
  \setlength\belowdisplayskip{0pt}}

\newtheorem{theorem}{Theorem}

\usepackage{hyperref}
\hypersetup{
colorlinks=true,
linkcolor=blue,
citecolor=red,
filecolor=magenta,
urlcolor=cyan,
pdftitle={},
pdfauthor={},
bookmarks=true,
}

\begin{document}

\author{Abhishek Das}
\title[Vanishing viscosity limits]{Explicit structure of the vanishing viscosity limits for the zero-pressure gas dynamics system initiated by the linear combination of a characteristic function and a $\delta$-distribution}

\address{Abhishek Das 
\newline 
National Institute of Science Education and Research Bhubaneswar,
P.O. Jatni, Khurda 752050, Odisha, India
\newline
An OCC of Homi Bhabha National Institute,
Training School Complex,
Anushaktinagar, Mumbai 400094, India
}

\email{abhishek.das@niser.ac.in, ad16.1992@gmail.com}



\maketitle

\begin{abstract}

In this article, we consider the one-dimensional zero-pressure gas dynamics system

\[
u_t + \left( {u^2}/{2} \right)_x = 0,\ \rho_t + (\rho u)_x = 0
\]
in the upper-half plane with a linear combination of a characteristic function and a $\delta$-measure

\[
u|_{t=0} = u_a\ \chi_{ {}_{ \left( -\infty , a \right) } } + u_b\ \delta_{x=b},\ \rho|_{t=0} = \rho_c\ \chi_{ {}_{ \left( -\infty , c \right) } } + \rho_d\ \delta_{x=d}
\]
as initial data, where $a$, $b$, $c$, $d$ are distinct points on the real line ordered as $a < c < b < d$, and provide a detailed analysis of the vanishing viscosity limits for the above system utilizing the corresponding modified adhesion model

\[
u^\epsilon_t + \left({(u^\epsilon)^2}/{2} \right)_x =\frac{\epsilon}{2} u^\epsilon_{xx},\ \rho^\epsilon_t + (\rho^\epsilon u^\epsilon)_x = \frac{\epsilon}{2} \rho^\epsilon_{xx}.
\] 
For this purpose, we use suitable Hopf-Cole transformations and various asymptotic properties of the function erfc$: z \longmapsto \int_{z}^{\infty} e^{-s^2}\ ds$.
\end{abstract}

\section*{1. Introduction}

We consider the initial value problem for the zero-pressure gas dynamics system

\begin{equation}
u_t + \left( {u^2}/{2} \right)_x = 0,\ \rho_t + (\rho u)_x = 0,\ \left( x , t \right) \in \textbf{R}^1 \times \left( 0 , \infty \right),
\label{intro-1}
\end{equation}
under the assumption that the initial data consists of a linear combination of a characteristic function and a $\delta$-measure, more precisely,

\begin{equation}
u|_{t=0} = u_a\ \chi_{ {}_{ \left( -\infty , a \right) } } + u_b\ \delta_{x=b},\ \rho|_{t=0} = \rho_c\ \chi_{ {}_{ \left( -\infty , c \right) } } + \rho_d\ \delta_{x=d}.
\label{intro-2}
\end{equation}
Here $a$, $b$, $c$, $d$ are fixed with $a < c < b < d$ and $u_a$, $u_b$, $\rho_c$, $\rho_d$ are real constants. 

The system \eqref{intro-1} is an important system of partial differential equations finding applications in cosmology and having close relations with the Zeldovich approximation (\cite{z1}). \eqref{intro-1} is used to describe the evolution of matter in the expansion of universe as cold dust moving only under the effect of gravity. Here the objects under study, namely $u$ and $\rho$, respectively denote the velocity and the density of the particles, $x \in {\textbf{R}}^1$ denotes the space variable and $t>0$ denotes time. 

In \eqref{intro-1}, the first equation involving the velocity component $u$ alone is the Burgers equation. In general, we cannot expect to find solutions of this problem in the class of smooth functions even for smooth initial data and so we need the notion of weak solutions satisfying \eqref{intro-1} in the sense of distributions, namely

\begin{footnotesize}
\[
\begin{aligned}
\int_{0}^{\infty} \int_{-\infty}^{\infty} \left( u\ \phi_t + f(u)\ \phi_x \right)\ dx dt + \int_{-\infty}^{\infty} u(x,0)\ \phi (x,0)\ dx = 0,\ \phi \in C_c^\infty \left( \textbf{R}^1 \times [ 0 , \infty ) \right).
\end{aligned} 
\]
\end{footnotesize}
The initial value problem for the Burgers equation was studied by Hopf \cite{h1} using the method of vanishing viscosity, where the author has considered the problem

\[
\begin{aligned}
u_t + \left( \frac{u^2}{2} \right)_x &= \mu u_{xx},
\\
u|_{t=0} &= u_0
\end{aligned}
\]
for $\mu > 0$ with locally integrable initial data $u_0$ under the additional assumption that

\[
\lim_{|x| \rightarrow \infty} \frac{\int_{0}^{x} u_0 \left( \xi \right)\ d\xi}{x^2} = 0.
\] 
This resulted in an explicit formula

\[
\begin{aligned}
u(x,t) = \frac{\int_{-\infty}^{\infty}\ \frac{x-y}{t}\ e^{- \frac{F(x,y,t)}{2 \mu}}}{\int_{-\infty}^{\infty}\ e^{- \frac{F(x,y,t)}{2 \mu}}},
\end{aligned}
\]
where $F : (x,y,t) \longmapsto \frac{(x-y)^2}{2t} + \int_{0}^{y}\ u_0 \left( \xi \right)\ d\xi$. The subsequent aymptotic limit as $\mu \rightarrow 0$ was discussed in detail.

The work of Hopf \cite{h1} was generalized by Lax \cite{lax1} to general conservation laws of the form

\[
u_t + \left[ f(u) \right]_x = 0,
\]
under bounded measurable initial data $u|_{t=0} = u_0$, under the assumption that the flux function $f \in C^2$ is strictly convex $\left( f'' > 0 \right)$ with superlinear growth at infinity $\left( \lim_{y \rightarrow \infty} \frac{f(y)}{|y|} = \infty \right)$. It was observed that for each fixed $t>0$, except possibly for countably many $x \in \textbf{R}^1$, there exists a unique minimizer $y(x,t)$ for

\[
\begin{aligned}
\theta \left( x,y,t \right) := f^* \left( \frac{x-y}{t} \right) + \int_{0}^{y}\ u_0 \left( \xi \right) d\xi,
\end{aligned}
\]
where $f^* : z \longmapsto {\displaystyle{ \max_{p \in \textbf{R}^1} }} \{ pz - f(p) \}$ denotes the convex conjugate of $f$. 

A weak solution of the associated initial value problem could then be given by

\[
\begin{aligned}
\overline{u} (x,t) = \left( f^* \right)^{-1} \left( \frac{x-y(x,t)}{t} \right),
\end{aligned}
\]
which is defined pointwise a.e.

There are many published works For scalar conservation laws with bounded nonnegative measures as initial data (see Bertsch et. al. \cite{ber-sm-tr-tes-20}, Demengel and Serre \cite{dem-serre-91}, Liu and Pierre \cite{liu-pie-84} and the references given there). These works focus on scalar conservation laws of the form

\[
\begin{aligned}
u_{t}+ ( \phi(u) )_x = 0,
\end{aligned}
\]
with the initial data $u|_{t=0} = u_0$ at $t=0$ being a bounded Borel measure. The existence and uniqueness of solutions will then depend on the following factors:

\begin{enumerate}

	\item The initial data $u_0$ and its sign

	\item Whether the flux $\phi$ is odd or convex

\end{enumerate}
It is natural to expect the flux function $\phi$ will have a smoothing effect on the solution by virtue of its strong nonlinearity. 

The vanishing viscosity method has been previously applied by Joseph \cite{ktj-93} for the system \eqref{intro-1} under Riemann-type initial data

\[
\begin{aligned}
\left( u(x,0) , \rho(x,0) \right) = \begin{cases}
\left( u_{{}_{L}} , \rho_{{}_{L}} \right),\ &x<0,
\\
\left( u_{{}_{R}} , \rho_{{}_{R}} \right),\ &x>0.
\end{cases}
\end{aligned} 
\]
The modified adhesion model corresponding to the given problem \eqref{intro-1}, that is

\[
u^\epsilon_t + \left( \frac{(u^\epsilon)^2}{2} \right)_x = \frac{\epsilon}{2} u^\epsilon_{xx},\ \rho^\epsilon_t + ( \rho^\epsilon u^\epsilon )_x = \frac{\epsilon}{2} \rho^\epsilon_{xx}
\]
was subsequently reduced to the new system

\[
\begin{aligned}
U_t^\epsilon + \frac{\left( U_x^\epsilon \right)^2}{2} = \frac{\epsilon}{2} U_{xx}^\epsilon,\ R_t^\epsilon + R_x^\epsilon U_x^\epsilon = \frac{\epsilon}{2} R_{xx}^\epsilon,
\end{aligned}
\] 
with new initial data

\[
\left( U^\epsilon(x,0) , R^\epsilon(x,0) \right) = \begin{cases}
\left( u_{{}_{L}} x , \rho_{{}_{L}} x \right),\ &x<0,
\\
\left( u_{{}_{R}} x , \rho_{{}_{R}} x \right),\ &x>0.
\end{cases}
\] 
It was observed that the distributional derivatives

\[
\left( \overline{u}^\epsilon , \overline{R}^\epsilon \right) := \left( U_x^\epsilon, R_x^\epsilon \right)
\]
with respect to the space variable $x$ solved the system \eqref{intro-1} under the prescribed Riemann-type initial data. 

The first step involved the linearization of the first equation in \eqref{intro-1} by the Hopf-Cole transformation

\[
V^\epsilon = e^{- \frac{U^\epsilon}{\epsilon}},
\]
resulting in the consideration of the new linear problem

\[
\begin{aligned}
V_t^\epsilon &= V_{xx}^\epsilon,
\\
V^\epsilon (x,0) &= \begin{cases}
e^{- \frac{u_{{}_{L}} x}{\epsilon}},\ &x < 0,
\\
e^{- \frac{u_{{}_{R}} x}{\epsilon}},\ &x > 0.
\end{cases}
\end{aligned}
\]
The second equation in \eqref{intro-1} involving $\rho$ was linearised by a modified Hopf-Cole transformation

\[
S^\epsilon = R^\epsilon\ e^{- \frac{U_\epsilon}{\epsilon}}
\]
The new linear problem to be considered was

\[
\begin{aligned}
S_t^\epsilon &= S_{xx}^\epsilon,
\\
S^\epsilon (x,0) &= \begin{cases}
\rho_{{}_{L}} x\ e^{- \frac{u_{{}_{L}} x}{\epsilon}},\ &x < 0,
\\
\rho_{{}_{R}} x\ e^{- \frac{u_{{}_{R}} x}{\epsilon}},\ &x > 0.
\end{cases}
\end{aligned}
\]
These linear problems were explicitly solved and $u^\epsilon$, $\rho^\epsilon$ were recovered through the relation

\[
\left( u^\epsilon , \rho^\epsilon \right) = \left( - \epsilon \cdot \frac{V_x^\epsilon}{V^\epsilon} , \left( \frac{S^\epsilon}{V^\epsilon} \right)_x \right),
\]
where $u^\epsilon$ and $\rho^\epsilon$ are related to the original system \eqref{intro-1} through the modified adhesion model

\[
\begin{aligned}
u^\epsilon_t + \left( \frac{(u^\epsilon)^2}{2} \right)_x &= \frac{\epsilon}{2} u^\epsilon_{xx},\ \rho^\epsilon_t + \left( \rho^\epsilon u^\epsilon \right)_x = \frac{\epsilon}{2} \rho^\epsilon_{xx},
\\
\left( u^\epsilon \left( x,0 \right) , \rho^\epsilon \left( x,0 \right) \right) &= \begin{cases}
\left( u_L , \rho_L \right),\ &{ x < 0 },
\\
\left( u_R , \rho_R \right),\ &{ x > 0 }.
\end{cases}
\end{aligned}
\]
The above strategy has been previously applied in

\begin{itemize}

	\item \cite{das-1-submitted} for the problem

\[
\begin{aligned}
u_t + \left({u^2}/{2} \right)_x &= 0,\ \rho_t + (\rho u)_x = 0,
\\
u|_{t=0} &= u_a\ \delta_{x=a} + u_b\ \delta_{x=b},\ \rho|_{t=0} = \rho_c\ \delta_{x=c} +\rho_d\ \delta_{x=d}.
\end{aligned}
\]
	
	\item \cite{das-ktj-submitted} in the context of the problem discussed in this article for the case $c=a$, $d=b$. 

\end{itemize}

\section*{2. Computation of the vanishing viscosity limits}
	
Let us now generalize the ideas discussed in the previous section for our problem.
\\
For each $\epsilon > 0$, suppose $\left( u^\epsilon , \rho^\epsilon \right)$ are approximate solutions for the problem

\begin{align}
u^\epsilon_t + \left( \frac{(u^\epsilon)^2}{2} \right)_x &= \frac{\epsilon}{2} u^\epsilon_{xx},\ \rho^\epsilon_t + ( \rho^\epsilon u^\epsilon )_x = \frac{\epsilon}{2} \rho^\epsilon_{xx},
\label{intro-3}
\\
u^\epsilon|_{t=0} &= u_a\ \chi_{ {}_{ \left( -\infty , a \right) } } + u_b\ \delta_{x=b},\ \rho^\epsilon|_{t=0} = \rho_c\ \chi_{ {}_{ \left( -\infty , c \right) } } + \rho_d\ \delta_{x=d}.
\label{intro-4}
\end{align}
First let us observe that if $\left( U^\epsilon , R^\epsilon \right)$ is a solution to the system

\begin{equation}
U_t^\epsilon + \frac{\left( U_x^\epsilon \right)^2}{2} = \frac{\epsilon}{2} U_{xx}^\epsilon,\ R_t^\epsilon + R_x^\epsilon U_x^\epsilon = \frac{\epsilon}{2} R_{xx}^\epsilon
\label{intro-5}
\end{equation}
under the initial conditions

\begin{equation}
\begin{aligned}
U^\epsilon (x,0) &= \begin{cases}
u_a (x-a),\ &x<a,
\\
0,\ &a<x<b,
\\
u_b,\ &x>b,
\end{cases}
\\
R^\epsilon (x,0) &= \begin{cases}
\rho_c (x-c),\ &x<c,
\\
0,\ &c<x<d,
\\
\rho_d,\ &x>d,
\end{cases}
\end{aligned}
\label{intro-6}
\end{equation}
then the distributional derivatives $\overline{u}^\epsilon := U_x^\epsilon$ and $\overline{\rho}^\epsilon := R_x^\epsilon$ in the space variable $x$ will solve the problem \eqref{intro-3} under the initial conditions \eqref{intro-4}.

The following theorem gives the detailed analysis of the asymptotic behavior of $\left( u^\epsilon , \rho^\epsilon \right)$ as $\epsilon \rightarrow 0$ in each of the regions $\left( -\infty , a \right)$, $\left( a , c \right)$, $\left( c , b \right)$, $\left( b , d \right)$ and finally $\left( d , \infty \right)$. Within each of these regions, the limit $\lim_{\epsilon \rightarrow 0}\ \left( u^\epsilon , \rho^\epsilon \right)$ is separately evaluated in the subregions $x > a + u_a t$ and $x < a + u_a t$ depending on the sign of $u_a$.

\begin{theorem}

Suppose $a$, $b$, $c$, $d$ are points on the real line ordered according to the inequalities $a < c < b < d$. Given real constants $u_a$, $u_b$, $\rho_c$ and $\rho_d$, let us consider the one-dimensional zero-pressure gas dynamics system

\[
u_t + \left( {u^2}/{2} \right)_x = 0,\ \rho_t + (\rho u)_x = 0,\ \left( x , t \right) \in {\textnormal{\textbf{R}}}^1 \times \left( 0 , \infty \right)
\] 
under the initial data

\[
u|_{t=0} = u_a\ \chi_{ {}_{ \left( -\infty , a \right) } } + u_b\ \delta_{x=b},\ \rho|_{t=0} = \rho_c\ \chi_{ {}_{ \left( -\infty , c \right) } } + \rho_d\ \delta_{x=d}.
\] 
Suppose $u^\epsilon$, $\rho^\epsilon$ are approximate solutions of the system

\[
\begin{aligned}
u^\epsilon_t + \left( \frac{(u^\epsilon)^2}{2} \right)_x &= \frac{\epsilon}{2} u^\epsilon_{xx},\ \rho^\epsilon_t + ( \rho^\epsilon u^\epsilon )_x = \frac{\epsilon}{2} \rho^\epsilon_{xx},
\\
u^\epsilon|_{t=0} &= u_a\ \chi_{ {}_{ \left( -\infty , a \right) } } + u_b\ \delta_{x=b},\ \rho^\epsilon|_{t=0} = \rho_c\ \chi_{ {}_{ \left( -\infty , c \right) } } + \rho_d\ \delta_{x=d}.
\end{aligned}
\]
Then the structure of the vanishing viscosity limit $\left( u , \rho \right) = \lim_{\epsilon \rightarrow 0} \left( u^\epsilon , \rho^\epsilon \right)$ can be explicitly described under various cases depending on the relative signs of $u_a$ and $u_b$ as follows:

\textbf{Case 1.} $u_a < 0$, $u_b > 0$

In this case, we consider the curves

\begin{itemize}

	\item[$(i)$] $r(s) := a + u_a s$,

	\item[$(ii)$] $p(s) := b + \sqrt{2 u_b s}$

\end{itemize}
defined for every $s \geq 0$. The explicit structure of $\lim_{\epsilon \rightarrow 0} \left( u^\epsilon , \rho^\epsilon \right)$ can then be described as follows:

\[
\begin{aligned}
u(x,t) &= \begin{cases}
u_a,\ &{ x < r(t) },
\\
\frac{x-a}{t},\ &{ x \in \Big( r(t) , a \Big) },
\\
0,\ &{ x \in \Big( a , b \Big) \cup \Big( p(t) , \infty \Big) },
\\
\frac{x-b}{t},\ &{ x \in \Big( b , p(t) \Big) },
\end{cases}
\\
\rho &= \rho_c \left( \chi_{ {}_{ {}_{ \left( -\infty , a + u_a t \right) } } } + \frac{4 \left( x-a \right) - 2 u_a t}{u_a t} \chi_{ {}_{ {}_{ \left( a + u_a t , a \right) } } } + \chi_{ {}_{ {}_{ \left( a , c \right) } } } \right) + \rho_d\ \delta_{ x = \gamma_{ {}_{d} } (t) }.
\end{aligned}
\]

\textbf{Case 2.} $u_a > 0$, $u_b > 0$

Consider the curves

\[
\begin{aligned}
l(s) &:= a + \frac{u_b}{u_a} + \frac{u_a}{2} \cdot s,
\\
\tilde{l} (s) &:= a + \frac{u_a}{2} \cdot s,
\\
r(s) &:= a + u_a s,
\\
p(s) &:= b + \sqrt{2 u_b s},
\\
q(s) &:= b+u_a s-\sqrt{2\ u_a \left( b-a \right) s},
\\
\gamma_{ {}_{a} } (s) &:= \begin{cases}
a + \frac{u_a}{2} \cdot s,\ &{ 0 \leq s \leq \frac{2 \left( b-a \right)}{u_a} },
\\
b + u_a s - \sqrt{2 u_a \left( b-a \right) s},\ &{ \frac{2 \left( b-a \right)}{u_a} \leq s \leq t_{ {}_{p,l} } } := \left( \frac{\sqrt{2 u_b} + \sqrt{2 u_a \left( b-a \right)}}{u_a} \right)^2,
\\
a + \frac{u_b}{u_a} + \frac{u_a}{2} \cdot s,\ &{ s \geq t_{ {}_{p,l} } },
\end{cases}
\\
\gamma_{ {}_{b} } (s) &:= \begin{cases}
b + \sqrt{2 u_b s},\ &{ 0 \leq s \leq t_{ {}_{p,l} } },
\\
a + \frac{u_b}{u_a} + \frac{u_a}{2} \cdot s,\ &{ s \geq t_{ {}_{p,l} } },
\end{cases}
\\
\gamma_{ {}_{c} } (s) &:= \begin{cases}
c,\ &{ 0 \leq s \leq \frac{2 \left( c-a \right)}{u_a} },
\\
\gamma_{ {}_{a} } (s),\ &{ s \geq \frac{2 \left( c-a \right)}{u_a} },
\end{cases}
\\
\gamma_{ {}_{d} } (s) &:= \begin{cases}
d,\ &{ 0 \leq s \leq \frac{(d-b)^2}{2 u_b} },
\\
\gamma_{ {}_{b} } (s),\ &{ s \geq \frac{(d-b)^2}{2 u_b} }
\end{cases}
\end{aligned}
\]
defined over $\left[ 0 , \infty \right)$. Then the limit $\lim_{\epsilon \rightarrow 0} \left( u^\epsilon , \rho^\epsilon \right)$ has the following explicit representation:

\[
\begin{aligned}
u(x,t) &= \begin{cases}
u_a,\ &{
\begin{aligned}
&x \in \Big( -\infty , a \Big) \cup \Big( \big( ( a , b ) \setminus \{ c \} \big) \cap \big( -\infty , \tilde{l} (t) \big) \Big) 
\\
&\cup \Big( \big( ( b , \infty ) \setminus \{ d \} \big) \cap \big( ( -\infty , \min{ \{ p(t) , q(t) , r(t) \} } ) \cup ( p(t) , \min{ \{ l(t) , q(t) , r(t) \} } ) \big)  \Big),
\end{aligned} }
\\
\frac{x-b}{t},\ &{ x \in \Big( \big( ( b , \infty ) \setminus \{ d \} \big) \cap \big( \left( r(t) , p(t) \right) \cup \left( q(t) , \min{ \{ p(t) , r(t) \} } \right) \big) \Big) },
\\
0,\ &{ 
\begin{aligned}
x \in &\bigg( \Big( ( a , b ) \setminus \{ c \} \Big) \cap \Big( \big( \tilde{l} (t) , \infty \big) \setminus \{ r(t) \} \Big) \bigg) 
\\
&{
\begin{aligned}
&\cup \Bigg( \Big( ( b , \infty ) \setminus \{ d \} \Big) \cap \bigg( \big( \max{ \{ p(t) , r(t) \} , \infty } \big) 
\\
&\cup \big( \max{ \{ p(t) , q(t) \} } , r(t) \big) \cup \big( \max{ \{ l(t) , p(t) \} } , \min{ \{ q(t) , r(t) \} } \big) \bigg) \Bigg),
\end{aligned} }
\end{aligned} }
\end{cases}
\\
\rho &= \rho_c \left( \chi_{ {}_{ {}_{ \left( -\infty , \gamma_{ {}_{ c } } (t) \right) } } } - \left( \left( x-c \right) \delta_{ {}_{ x = \gamma_{ {}_{c} } (t) } } - u_a t\ \delta_{ {}_{ x = \gamma_{ {}_{a} } (t) } } \right) \right) + \rho_d\ \delta_{ {}_{ x = \gamma_{ {}_{d} } (t) } }.
\end{aligned} 
\]

\textbf{Case 3.} $u_a > 0$, $u_b < 0$

Consider the curves

\[
\begin{aligned}
l(t) &:= a + \frac{u_a}{2} \cdot t,
\\
\tilde{l} (t) &:= a + \frac{u_a}{2} \cdot t,
\\
r(t) &:= a + u_a t,
\\
p(t) &:= b - \sqrt{- 2 u_b t},
\\
q(t) &:= b+u_a t-\sqrt{2 \left( u_a \left( b-a \right) - u_b \right) t}
\end{aligned}
\]
defined over $\left[ 0 , \infty \right)$. As in the previous cases, we can define additional curves $x = \gamma_{ { {}_{a} } } (t)$, $x = \gamma_{ { {}_{c} } } (t)$ and $x = \gamma_{ { {}_{d} } } (t)$ so that the limit $\lim_{\epsilon \rightarrow 0} \left( u^\epsilon , \rho^\epsilon \right)$ has the following explicit representation:

\begin{small}
\[
\begin{aligned}
u(x,t) &= \begin{cases}
u_a,\ &{ 
\begin{aligned}
x \in &\Big( \big( -\infty , a \big) \cap \big( -\infty , q(t) \big) \Big) 
\\
&\cup \bigg( \Big( ( a , b ) \setminus \{ c \} \Big) \cap \Big( -\infty , \min{ \{ q(t) , r(t) \} } \Big) \cap \Big( \big( -\infty , \min{ \{ \tilde{l} (t) , p(t) \} } \big) \cup \big( p(t) , \infty \big) \Big) \bigg)
\\
&\cup \bigg( \Big( ( b , \infty ) \setminus \{ d \} \Big) \cap \Big( -\infty , \min{ \{ l(t) , r(t) \} } \Big) \bigg),
\end{aligned} }
\\
\frac{x-b}{t},\ &{ 
\begin{aligned}
x \in &\Big( \big( -\infty , a \big) \cap \big( q(t) , \infty \big) \Big) 
\\
&\cup \bigg( \Big( ( a , b ) \setminus \{ c \} \Big) \cap \Big( \big( \max{ \{ p(t) , q(t) \} } , r(t) \big) \cup \big( \max{ \{ p(t) , r(t) \} } , \infty \big) \Big) \bigg),
\end{aligned} }
\\
0,\ &{ 
\begin{aligned}
x \in &
\Bigg( 
\Big( ( a , b ) \setminus \{ c \} \Big) 
\cap 
\bigg( 
\Big( r(t) , p(t) \Big) \cup \bigg( \Big( -\infty , \min{ \big\{ p(t) , r(t) \big\} } \Big) \cap \Big( \big( l(t) , q(t) \big) \cup \big( q(t) , \infty \big) \Big) \bigg)
\bigg) 
\Bigg)  
\\
&\cup \bigg( \Big( ( b , \infty ) \setminus \{ d \} \Big) \cap \Big( ( l(t) , r(t) ) \cup ( r(t) , \infty ) \Big) \bigg),
\end{aligned} }
\end{cases}
\\
\rho &= \rho_c \left( \chi_{ {}_{ {}_{ \left( -\infty , \gamma_{ {}_{ c } } (t) \right) } } } - \left( \left( x-c \right) \delta_{ {}_{ x = \gamma_{ {}_{c} } (t) } } - u_a t\ \delta_{ {}_{ x = \gamma_{ {}_{a} } (t) } } \right) \right) + \rho_d\ \delta_{ {}_{ x = \gamma_{ {}_{d} } (t) } }.
\end{aligned} 
\]
\end{small}

\textbf{Case 4.} $u_a < 0$, $u_b < 0$

In this case, introduce the curves

\[
\begin{aligned}
l(s) &:= \frac{a+b}{2} + \frac{u_b}{b-a} \cdot s,
\\
r(s) &:= a + u_a s,
\\
p(s) &:= b - \sqrt{- 2 u_b s},
\\
q(s) &:= b+u_a s-\sqrt{2 \left( u_a \left( b-a \right) - u_b \right) s}\ ({\textnormal{ defined if }} u_a \left( b-a \right) > u_b\ )
\end{aligned}
\]
over $\left[ 0 , \infty \right)$. Under this case, we have to consider two separate cases: $u_a \left( b-a \right) > u_b$ and $u_a \left( b-a \right) \leq u_b$. As done before, we can also define additional curves $x = \gamma_{ { {}_{a,1} } } (t)$, $x = \gamma_{ { {}_{a,2} } } (t)$, $x = \gamma_{ { {}_{c} } } (t)$ and $x = \gamma_{ { {}_{d} } } (t)$ in each case so that the limit $\lim_{\epsilon \rightarrow 0} \left( u^\epsilon , \rho^\epsilon \right)$ has the following explicit representation:

\begin{itemize}

	\item $u_a \left( b-a \right) > u_b$
	
\[
\begin{aligned}
u(x,t) &= \begin{cases}
u_a,\ &{ x \in ( -\infty , a ) \cap \Big( -\infty , \min{ \left\{ q(t) , r(t) \right\} } \Big) },
\\
\frac{x-a}{t},\ &{ x \in ( -\infty , a ) \cap \Big( r(t) , l(t) \Big) },
\\
\frac{x-b}{t},\ &{ 
\begin{aligned}
x \in &\bigg( ( -\infty , a ) \cap \Big( \big( \max{ \{ l(t) , r(t) \} } , \infty \big) \cup \big( q(t) , r(t) \big) \Big) \bigg) 
\\
&\cup \bigg( \Big( ( a , b ) \setminus \{ c \} \Big) \cap \big( p(t) , \infty \big) \bigg)
\end{aligned} },
\\
0,\ &{ x \in \bigg( \Big( ( a , b ) \setminus \{ c \} \Big) \cap \big( -\infty , p(t) \big) \bigg) \cup \Big( ( b , \infty ) \setminus \{ d \} \Big) },
\end{cases}
\\
\rho &= \rho_c \Bigg[ \chi_{ {}_{ \left( -\infty , \gamma_{ {}_{ a,1 } } (t) \right) } } + \left( \frac{4 \left( x-a \right)}{u_a t} - 2 \right) \chi_{ {}_{ \left( \gamma_{ {}_{ a,1 } } (t) , \gamma_{ {}_{ a,2 } } (t) \right) } } + \chi_{ {}_{ \left( \gamma_{ {}_{ a,2 } } (t) , \gamma_{ {}_{ c } } (t) \right) } }
\\
&+ \left( 2 \left( c-a \right) + u_a t + \gamma_{ {}_{ c } } (t) - x - 2 - \frac{2 \left( x-a \right)^2}{u_a t} \right) \delta_{ {}_{ x = \gamma_{ {}_{ a,1 } } (t) } } 
\\
&+ \left( 2 \left( x-a \right) + c - \gamma_{ {}_{ c } } (t) - 2 - \frac{2 \left( x-a \right)^2}{u_a t} \right) \delta_{ {}_{ x = \gamma_{ {}_{ a,2 } } (t) } } \Bigg] + \rho_d\ \delta_{ x = \gamma_{ {d} } (t) }.
\end{aligned} 
\]
	
	\item $u_a \left( b-a \right) \leq u_b$

\[
\begin{aligned}
u(x,t) &= \begin{cases}
u_a,\ &{ x \in ( -\infty , a ) \cap \Big( -\infty , r(t) \Big) },
\\
\frac{x-a}{t},\ &{ x \in ( -\infty , a ) \cap \Big( r(t) , l(t) \Big) },
\\
\frac{x-b}{t},\ &{ 
\begin{aligned}
x \in &\bigg( ( -\infty , a ) \cap \Big( \big( \max{ \{ l(t) , r(t) \} } , \infty \big) \Big) \bigg) 
\\
&\cup \bigg( \Big( ( a , b ) \setminus \{ c \} \Big) \cap \big( p(t) , \infty \big) \bigg)
\end{aligned} },
\\
0,\ &{ x \in \bigg( \Big( ( a , b ) \setminus \{ c \} \Big) \cap \big( -\infty , p(t) \big) \bigg) \cup \Big( ( b , \infty ) \setminus \{ d \} \Big) },
\end{cases}
\\
\rho &= \rho_c \Bigg[ \chi_{ {}_{ \left( -\infty , \gamma_{ {}_{ a,1 } } (t) \right) } } + \left( \frac{4 \left( x-a \right)}{u_a t} - 2 \right) \chi_{ {}_{ \left( \gamma_{ {}_{ a,1 } } (t) , \gamma_{ {}_{ a,2 } } (t) \right) } } + \chi_{ {}_{ \left( \gamma_{ {}_{ a,2 } } (t) , \gamma_{ {}_{ c } } (t) \right) } }
\\
&+ \left( 2 \left( c-a \right) + u_a t + \gamma_{ {}_{ c } } (t) - x - 2 - \frac{2 \left( x-a \right)^2}{u_a t} \right) \delta_{ {}_{ x = \gamma_{ {}_{ a,1 } } (t) } } 
\\
&+ \left( 2 \left( x-a \right) + c - \gamma_{ {}_{ c } } (t) - 2 - \frac{2 \left( x-a \right)^2}{u_a t} \right) \delta_{ {}_{ x = \gamma_{ {}_{ a,2 } } (t) } } \Bigg] + \rho_d\ \delta_{ x = \gamma_{ {d} } (t) }.
\end{aligned} 
\]

\end{itemize}

\end{theorem}

\begin{proof}

The first step is to consider the generalized Hopf-Cole transformations

\begin{equation}
V^\epsilon := e^{- \frac{U^\epsilon}{\epsilon}},\ S^\epsilon := R^\epsilon\ e^{- \frac{U^\epsilon}{\epsilon}},
\label{intro-7}
\end{equation}
leading us to the consideration of the linear problem

\begin{equation}
V_t^\epsilon = \frac{\epsilon}{2} V_{xx}^\epsilon,\ S_t^\epsilon = \frac{\epsilon}{2} S_{xx}^\epsilon
\label{intro-8}
\end{equation}
under the initial conditions

\begin{align}
V^\epsilon (x,0) &= \begin{cases}
e^{- \frac{u_a (x-a)}{\epsilon}},\ &x<a,
\\
1,\ &a<x<b,
\\
e^{- \frac{u_b}{\epsilon}},\ &x>b,
\end{cases}
\label{intro-9}
\\
S^\epsilon (x,0) &= \begin{cases}
\rho_c (x-c)\ e^{- \frac{u_a (x-a)}{\epsilon}},\ &x<a,
\\
\rho_c (x-c),\ &a<x<c,
\\
0,\ &c<x<d,
\\
\rho_d\ e^{- \frac{u_b}{\epsilon}},\ &x>d.
\end{cases}
\label{intro-10}
\end{align}
The system \eqref{intro-8} under the initial conditions \eqref{intro-9}-\eqref{intro-10} can be solved explicitly by

\[
\begin{aligned}
\begin{pmatrix}
V^\epsilon (x,t)
\\
\\
S^\epsilon (x,t)
\end{pmatrix} = \begin{pmatrix}
\frac{1}{\sqrt{\pi}} 
\Bigg[ {\textnormal{erfc}} \left( \frac{x - a - u_a t}{\sqrt{2t \epsilon}} \right) \exp \left( \frac{\left( x - a - u_a t \right)^2 - (x-a)^2}{2t \epsilon} \right)
\\
+ \left( {\textnormal{erfc}} \left( \frac{x-b}{\sqrt{2t \epsilon}} \right) - {\textnormal{erfc}} \left( \frac{x-a}{\sqrt{2t \epsilon}} \right) \right) + {\textnormal{erfc}} \left( - \frac{x-b}{\sqrt{2t \epsilon}} \right) e^{- \frac{u_b}{\epsilon}} \Bigg]
\\
----------------------------------
\\
\frac{1}{\sqrt{2 \pi t \epsilon}} \left[ t \epsilon\ \rho_c\ e^{- \frac{(x-a)^2}{2 t \epsilon}} + \sqrt{2 t \epsilon}\ \rho_c ( x - c - u_a t )\ e^{\frac{( x - a - u_a t )^2 - (x-a)^2}{2 t \epsilon}} \textnormal{erfc} \left( \frac{x - a - u_a t}{\sqrt{2 t \epsilon}} \right) \right.
\\
+ t \epsilon\ \rho_c \left( e^{- \frac{(x-a)^2}{2 t \epsilon}} - e^{- \frac{(x-c)^2}{2 t \epsilon}} \right) + \sqrt{2 t \epsilon}\ \rho_c (x-c) \left( \textnormal{erfc} \left( \frac{x-c}{\sqrt{2 t \epsilon}} \right) - \textnormal{erfc} \left( \frac{x-a}{\sqrt{2 t \epsilon}} \right) \right)
\\
+ \left. \sqrt{2 t \epsilon}\ \rho_d\ e^{- \frac{u_b}{\epsilon}}\ \textnormal{erfc} \left( -\frac{x-d}{\sqrt{2 t \epsilon}} \right) \right]
\end{pmatrix}
\end{aligned}
\]
where $\textnormal{erfc} : z \longmapsto \int_{z}^{\infty}\ e^{-t^2}\ dt$ for every $z \in \textbf{R}^1$. Coming back to the original problem \eqref{intro-3}-\eqref{intro-4}, we  explicitly recover $u^\epsilon$ and $R^\epsilon$ as follows:

\[
\begin{aligned}
u^\epsilon &= -\epsilon \cdot \frac{V_x^\epsilon}{V^\epsilon} 
\\
&= \frac{\epsilon}{\sqrt{2t \epsilon}} \cdot \frac{\begin{aligned}
u_a \cdot \frac{\sqrt{2t \epsilon}}{\epsilon}\ {\textnormal{erfc}} \left( \frac{ x - a - u_a t }{\sqrt{2 t \epsilon}} \right) e^{\frac{( x - a - u_a t )^2 - (x-a)^2}{2 t \epsilon}} 
\\
+ e^{- \frac{(x-b)^2}{2t \epsilon}} \left( 1 - e^{- \frac{u_b}{\epsilon}} \right)
\end{aligned}}{
\begin{aligned}
{\textnormal{erfc}} \left( \frac{x - a - u_a t}{\sqrt{2t \epsilon}} \right) \exp \left( {\frac{\left( x - a - u_a t \right)^2 - (x-a)^2} {2 t \epsilon}} \right)
\\
+ \left( {\textnormal{erfc}} \left( \frac{x-b}{\sqrt{2t \epsilon}} \right) - {\textnormal{erfc}} \left( \frac{x-a}{\sqrt{2t \epsilon}} \right) \right) + {\textnormal{erfc}} \left( - \frac{x-b}{\sqrt{2t \epsilon}} \right) e^{- \frac{u_b}{\epsilon}}
\end{aligned} },
\\
\\
R^\epsilon &= \frac{S^\epsilon}{V^\epsilon} 
\\
&= \frac{
\begin{aligned}
\frac{t \epsilon}{\sqrt{2 t \epsilon}}\ \rho_c\ e^{- \frac{(x-a)^2}{2 t \epsilon}} + \rho_c ( x - c - u_a t )\ e^{\frac{( x - a - u_a t )^2 - (x-a)^2}{2 t \epsilon}} \textnormal{erfc} \left( \frac{x - a - u_a t}{\sqrt{2 t \epsilon}} \right)
\\
+ \frac{t \epsilon}{\sqrt{2 t \epsilon}}\ \rho_c \left( e^{- \frac{(x-a)^2}{2 t \epsilon}} - e^{- \frac{(x-c)^2}{2 t \epsilon}} \right) + \rho_c (x-c) \left( \textnormal{erfc} \left( \frac{x-c}{\sqrt{2 t \epsilon}} \right) - \textnormal{erfc} \left( \frac{x-a}{\sqrt{2 t \epsilon}} \right) \right)
\\
+ \rho_d\ e^{- \frac{u_b}{\epsilon}}\ \textnormal{erfc} \left( -\frac{x-d}{\sqrt{2 t \epsilon}} \right)\
\end{aligned}
}{
\begin{aligned}
{\textnormal{erfc}} \left( \frac{x - a - u_a t}{\sqrt{2t \epsilon}} \right) \exp \left( {\frac{\left( x - a - u_a t \right)^2 - (x-a)^2} {2 t \epsilon}} \right)
\\
+ \left( {\textnormal{erfc}} \left( \frac{x-b}{\sqrt{2t \epsilon}} \right) - {\textnormal{erfc}} \left( \frac{x-a}{\sqrt{2t \epsilon}} \right) \right) + {\textnormal{erfc}} \left( - \frac{x-b}{\sqrt{2t \epsilon}} \right) e^{- \frac{u_b}{\epsilon}}
\end{aligned} }.
\end{aligned}
\]
Throughout this article, we assume that $u_a,u_b \neq 0$. For each $\epsilon > 0$ and $\left( x,t \right) \in \textbf{R}^1 \times \left( 0 , \infty \right)$, define
\begin{itemize}

	\item $A_\epsilon = A_\epsilon (x,t) := \frac{|x-a|}{\sqrt{2t \epsilon}}$
	
	\item $B_\epsilon = B_\epsilon (x,t) := \frac{|x-b|}{\sqrt{2t \epsilon}}$
	
	\item $C_\epsilon = C_\epsilon (x,t) := \frac{|x-c|}{\sqrt{2t \epsilon}}$
	
	\item $D_\epsilon = D_\epsilon (x,t) := \frac{|x-d|}{\sqrt{2t \epsilon}}$
	
	\item $P_\epsilon = P_\epsilon (x,t) := \frac{| x - a - u_a t |}{\sqrt{2 t \epsilon}}$

\end{itemize}
Then
\begin{itemize}

	\item $A_\epsilon (x,t) \xrightarrow{\epsilon \rightarrow 0} \infty$ whenever $x \neq a$
	
	\item $B_\epsilon (x,t) \xrightarrow{\epsilon \rightarrow 0} \infty$ whenever $x \neq b$
	
	\item $C_\epsilon (x,t) \xrightarrow{\epsilon \rightarrow 0} \infty$ whenever $x \neq c$
	
	\item $D_\epsilon (x,t) \xrightarrow{\epsilon \rightarrow 0} \infty$ whenever $x \neq d$
	
	\item $P_\epsilon (x,t) \xrightarrow{\epsilon \rightarrow 0} \infty$ whenever $x \neq a + u_a t$

\end{itemize}

Our next objective is to obtain the corresponding explicit expressions of $u^\epsilon = u^\epsilon (x,t)$ and $R^\epsilon = R^\epsilon (x,t)$ in the regions $x < a$, $a < x < c$, $c < x < b$, $b < x < d$ and $x > d$ in terms of $A_\epsilon$, $B_\epsilon$, $C_\epsilon$, $D_\epsilon$ and $P_\epsilon$. The details of the relevant computations involved here have been provided in the appendix.

\begin{itemize}
	\item $x < a$
\[
\begin{aligned}
u^\epsilon &= \begin{cases}
\frac{\begin{aligned}
u_a\ {\textnormal{erfc}} \left( P_\epsilon \right) e^{P_\epsilon^2 - A_\epsilon^2} + \frac{\epsilon}{\sqrt{2t \epsilon}} \left( 1 - e^{- \frac{u_b}{\epsilon}} \right) e^{- B_\epsilon^2}
\end{aligned}}{
\begin{aligned}
{\textnormal{erfc}} (A_\epsilon) + {\textnormal{erfc}} (B_\epsilon) \left( e^{- \frac{u_b}{\epsilon}} - 1 \right) + {\textnormal{erfc}} \left( P_\epsilon \right) e^{P_\epsilon^2 - A_\epsilon^2}
\end{aligned} },\ &{ x > a + u_a t },
\\
\\
\frac{\begin{aligned}
u_a \left( \sqrt{\pi} - {\textnormal{erfc}} \left( P_\epsilon \right) \right) e^{P_\epsilon^2 - A_\epsilon^2} + \frac{\epsilon}{\sqrt{2t \epsilon}} \left( 1 - e^{- \frac{u_b}{\epsilon}} \right) e^{- B_\epsilon^2}
\end{aligned}}{
\begin{aligned}
{\textnormal{erfc}} (A_\epsilon) + {\textnormal{erfc}} (B_\epsilon) \left( e^{- \frac{u_b}{\epsilon}} - 1 \right) + \left( \sqrt{\pi} - {\textnormal{erfc}} \left( P_\epsilon \right) \right) e^{P_\epsilon^2 - A_\epsilon^2}
\end{aligned} },\ &{ x < a + u_a t },
\end{cases} 
\\
\\
R^\epsilon
&= \begin{cases}
\frac{
\begin{aligned}
\rho_c \bigg[ \sqrt{2 t \epsilon} \left( \frac{1}{e^{A_\epsilon^2}} - \frac{1}{2\ e^{C_\epsilon^2}} \right) + \left( x - c - u_a t \right) \textnormal{erfc} \left( P_\epsilon \right) e^{P_\epsilon^2 - A_\epsilon^2}
\\
+ \left( x-c \right) \left( \textnormal{erfc} \left( A_\epsilon \right) - \textnormal{erfc} \left( C_\epsilon \right) \right) \bigg] + \rho_d\ \textnormal{erfc} \left( D_\epsilon \right) e^{- \frac{u_b}{\epsilon}}
\end{aligned}
}{
\begin{aligned}
{\textnormal{erfc}} \left( A_\epsilon \right) + {\textnormal{erfc}} \left( B_\epsilon \right) \left( e^{- \frac{u_b}{\epsilon}} - 1 \right) + {\textnormal{erfc}} \left( P_\epsilon \right) e^{ P_\epsilon^2 - A_\epsilon^2 }
\end{aligned} },\ &{ x > a + u_a t },
\\
\\
\frac{
\begin{aligned}
\rho_c \bigg[ \sqrt{2 t \epsilon} \left( \frac{1}{e^{A_\epsilon^2}} - \frac{1}{2\ e^{C_\epsilon^2}} \right) + \left( x - c - u_a t \right) \left( \sqrt{\pi} - \textnormal{erfc} \left( P_\epsilon \right) \right) e^{P_\epsilon^2 - A_\epsilon^2}
\\
+ \left( x-c \right) \left( \textnormal{erfc} \left( A_\epsilon \right) - \textnormal{erfc} \left( C_\epsilon \right) \right) \bigg] + \rho_d\ \textnormal{erfc} \left( D_\epsilon \right) e^{- \frac{u_b}{\epsilon}}
\end{aligned}
}{
\begin{aligned}
{\textnormal{erfc}} \left( A_\epsilon \right) + {\textnormal{erfc}} \left( B_\epsilon \right) \left( e^{- \frac{u_b}{\epsilon}} - 1 \right) + \left( \sqrt{\pi} - \textnormal{erfc} \left( P_\epsilon \right) \right) e^{ P_\epsilon^2 - A_\epsilon^2 }
\end{aligned} },\ &{ x < a + u_a t }.
\end{cases} 
\end{aligned}
\]
	\item $a < x < c$

\[
\begin{aligned}
u^\epsilon &= \begin{cases}
\frac{\begin{aligned}
u_a\ {\textnormal{erfc}} \left( P_\epsilon \right) e^{P_\epsilon^2 - A_\epsilon^2} + \frac{\epsilon}{\sqrt{2t \epsilon}} \left( 1 - e^{- \frac{u_b}{\epsilon}} \right) e^{- B_\epsilon^2}
\end{aligned}}{
\begin{aligned}
\sqrt{\pi} - {\textnormal{erfc}}(A_\epsilon) + {\textnormal{erfc}}(B_\epsilon) \left( e^{- \frac{u_b}{\epsilon}} - 1 \right) + {\textnormal{erfc}} \left( P_\epsilon \right) e^{P_\epsilon^2 - A_\epsilon^2} 
\end{aligned} },\ &{ x > a + u_a t },
\\
\\
\frac{\begin{aligned}
u_a \left( \sqrt{\pi} - {\textnormal{erfc}} \left( P_\epsilon \right) \right) e^{P_\epsilon^2 - A_\epsilon^2} + \frac{\epsilon}{\sqrt{2t \epsilon}} \left( 1 - e^{- \frac{u_b}{\epsilon}} \right) e^{- B_\epsilon^2}
\end{aligned}}{
\begin{aligned}
\sqrt{\pi} - {\textnormal{erfc}}(A_\epsilon) + {\textnormal{erfc}}(B_\epsilon) \left( e^{- \frac{u_b}{\epsilon}} - 1 \right) + \left( \sqrt{\pi} - {\textnormal{erfc}} \left( P_\epsilon \right) \right) e^{P_\epsilon^2 - A_\epsilon^2} 
\end{aligned} },\ &{ x < a + u_a t },
\end{cases} 
\\
\\
R^\epsilon 
&= \begin{cases}
\frac{
\begin{aligned}
\rho_c \bigg[ \sqrt{2 t \epsilon} \left( \frac{1}{e^{A_\epsilon^2}} - \frac{1}{2\ e^{C_\epsilon^2}} \right) + \left( x - c - u_a t \right) \textnormal{erfc} \left( P_\epsilon \right) e^{P_\epsilon^2 - A_\epsilon^2}
\\
+ \left( x-c \right) \left( \sqrt{\pi} - \textnormal{erfc} \left( A_\epsilon \right) - \textnormal{erfc} \left( C_\epsilon \right) \right) \bigg] + \rho_d\ \textnormal{erfc} \left( D_\epsilon \right) e^{- \frac{u_b}{\epsilon}}
\end{aligned}
}{
\begin{aligned}
\sqrt{\pi} - {\textnormal{erfc}}(A_\epsilon) + {\textnormal{erfc}}(B_\epsilon) \left( e^{- \frac{u_b}{\epsilon}} - 1 \right) + {\textnormal{erfc}} \left( P_\epsilon \right) e^{P_\epsilon^2 - A_\epsilon^2} 
\end{aligned} },\ &{ x > a + u_a t },
\\
\\
\frac{
\begin{aligned}
\rho_c \bigg[ \sqrt{2 t \epsilon} \left( \frac{1}{e^{A_\epsilon^2}} - \frac{1}{2\ e^{C_\epsilon^2}} \right) + \left( x - c - u_a t \right) \left( \sqrt{\pi} - \textnormal{erfc} \left( P_\epsilon \right) \right) e^{P_\epsilon^2 - A_\epsilon^2}
\\
+ \left( x-c \right) \left( \sqrt{\pi} - \textnormal{erfc} \left( A_\epsilon \right) - \textnormal{erfc} \left( C_\epsilon \right) \right) \bigg] + \rho_d\ \textnormal{erfc} \left( D_\epsilon \right) e^{- \frac{u_b}{\epsilon}}
\end{aligned}
}{
\begin{aligned}
\sqrt{\pi} - {\textnormal{erfc}}(A_\epsilon) + {\textnormal{erfc}}(B_\epsilon) \left( e^{- \frac{u_b}{\epsilon}} - 1 \right) + \left( \sqrt{\pi} - {\textnormal{erfc}} \left( P_\epsilon \right) \right) e^{P_\epsilon^2 - A_\epsilon^2} 
\end{aligned} },\ &{ x < a + u_a t }.
\end{cases} 
\end{aligned}
\]
	\item $c < x < b$

\[
\begin{aligned}
u^\epsilon &= \begin{cases}
\frac{\begin{aligned}
u_a\ {\textnormal{erfc}} \left( P_\epsilon \right) e^{P_\epsilon^2 - A_\epsilon^2} + \frac{\epsilon}{\sqrt{2t \epsilon}} \left( 1 - e^{- \frac{u_b}{\epsilon}} \right) e^{- B_\epsilon^2}
\end{aligned}}{
\begin{aligned}
\sqrt{\pi} - {\textnormal{erfc}}(A_\epsilon) + {\textnormal{erfc}}(B_\epsilon) \left( e^{- \frac{u_b}{\epsilon}} - 1 \right) + {\textnormal{erfc}} \left( P_\epsilon \right) e^{P_\epsilon^2 - A_\epsilon^2}
\end{aligned} },\ &{ x > a + u_a t },
\\
\\
\frac{\begin{aligned}
u_a \left( \sqrt{\pi} - {\textnormal{erfc}} \left( P_\epsilon \right) \right) e^{P_\epsilon^2 - A_\epsilon^2} + \frac{\epsilon}{\sqrt{2t \epsilon}} \left( 1 - e^{- \frac{u_b}{\epsilon}} \right) e^{- B_\epsilon^2}
\end{aligned}}{
\begin{aligned}
\sqrt{\pi} - {\textnormal{erfc}}(A_\epsilon) + {\textnormal{erfc}}(B_\epsilon) \left( e^{- \frac{u_b}{\epsilon}} - 1 \right) + \left( \sqrt{\pi} - {\textnormal{erfc}} \left( P_\epsilon \right) \right) e^{P_\epsilon^2 - A_\epsilon^2} 
\end{aligned} },\ &{ x < a + u_a t },
\end{cases}
\\
\\
R^\epsilon
&= \begin{cases}
\frac{
\begin{aligned}
\rho_c \bigg[ \sqrt{2 t \epsilon} \left( \frac{1}{e^{A_\epsilon^2}} - \frac{1}{2\ e^{C_\epsilon^2}} \right) + \left( x - c - u_a t \right) \textnormal{erfc} \left( P_\epsilon \right) e^{P_\epsilon^2 - A_\epsilon^2}
\\
+ (x-c) \left( \textnormal{erfc} \left( C_\epsilon \right) - \textnormal{erfc} \left( A_\epsilon \right) \right) \bigg] + \rho_d\ \textnormal{erfc} \left( D_\epsilon \right) e^{- \frac{u_b}{\epsilon}}
\end{aligned}
}{
\begin{aligned}
\sqrt{\pi} - {\textnormal{erfc}}(A_\epsilon) + {\textnormal{erfc}}(B_\epsilon) \left( e^{- \frac{u_b}{\epsilon}} - 1 \right) + {\textnormal{erfc}} \left( P_\epsilon \right) e^{P_\epsilon^2 - A_\epsilon^2}
\end{aligned} },\ &{ x > a + u_a t },
\\
\\
\frac{
\begin{aligned}
\rho_c \bigg[ \sqrt{2 t \epsilon} \left( \frac{1}{e^{A_\epsilon^2}} - \frac{1}{2\ e^{C_\epsilon^2}} \right) + \left( x - c - u_a t \right) \left( \sqrt{\pi} - \textnormal{erfc} \left( P_\epsilon \right) \right) e^{P_\epsilon^2 - A_\epsilon^2}
\\
+ (x-c) \left( \textnormal{erfc} \left( C_\epsilon \right) - \textnormal{erfc} \left( A_\epsilon \right) \right) \bigg] + \rho_d\ \textnormal{erfc} \left( D_\epsilon \right) e^{- \frac{u_b}{\epsilon}}
\end{aligned}
}{
\begin{aligned}
\sqrt{\pi} - {\textnormal{erfc}}(A_\epsilon) + {\textnormal{erfc}}(B_\epsilon) \left( e^{- \frac{u_b}{\epsilon}} - 1 \right) + \left( \sqrt{\pi} - {\textnormal{erfc}} \left( P_\epsilon \right) \right) e^{P_\epsilon^2 - A_\epsilon^2}
\end{aligned} },\ &{ x < a + u_a t }.
\end{cases} 
\end{aligned}
\]
\item $b < x < d$

\[
\begin{aligned}
u^\epsilon 
&= \begin{cases}
\frac{\begin{aligned}
u_a\ {\textnormal{erfc}} \left( P_\epsilon \right) e^{P_\epsilon^2 - A_\epsilon^2} + \frac{\epsilon}{\sqrt{2t \epsilon}} \left( 1 - e^{- \frac{u_b}{\epsilon}} \right) e^{- B_\epsilon^2}
\end{aligned}}{
\begin{aligned}
\sqrt{\pi}\ e^{- \frac{u_b}{\epsilon}} - {\textnormal{erfc}}({A_\epsilon}) + {\textnormal{erfc}} (B_\epsilon) \left( 1 - e^{- \frac{u_b}{\epsilon}} \right) + {\textnormal{erfc}} \left( P_\epsilon \right) e^{P_\epsilon^2 - A_\epsilon^2} 
\end{aligned} },\ &{ x > a + u_a t },
\\
\\
\frac{\begin{aligned}
u_a\ \left( \sqrt{\pi} - {\textnormal{erfc}} \left( P_\epsilon \right) \right) e^{P_\epsilon^2 - A_\epsilon^2} + \frac{\epsilon}{\sqrt{2t \epsilon}} \left( 1 - e^{- \frac{u_b}{\epsilon}} \right) e^{- B_\epsilon^2}
\end{aligned}}{
\begin{aligned}
\sqrt{\pi}\ e^{- \frac{u_b}{\epsilon}} - {\textnormal{erfc}}({A_\epsilon}) + {\textnormal{erfc}} (B_\epsilon) \left( 1 - e^{- \frac{u_b}{\epsilon}} \right) + \left( \sqrt{\pi} - {\textnormal{erfc}} \left( P_\epsilon \right) \right) e^{P_\epsilon^2 - A_\epsilon^2} 
\end{aligned} },\ &{ x < a + u_a t },
\end{cases} 
\\
\\
R^\epsilon
&= \begin{cases}
\frac{
\begin{aligned}
\rho_c \bigg[ \sqrt{2 t \epsilon} \left( \frac{1}{e^{A_\epsilon^2}} - \frac{1}{2\ e^{C_\epsilon^2}} \right) + \left( x - c - u_a t \right) \textnormal{erfc} \left( P_\epsilon \right) e^{P_\epsilon^2 - A_\epsilon^2}
\\
+ (x-c) \left( \textnormal{erfc} \left( C_\epsilon \right) - \textnormal{erfc} \left( A_\epsilon \right) \right) \bigg] + \rho_d\ \textnormal{erfc} \left( D_\epsilon \right) e^{- \frac{u_b}{\epsilon}}
\end{aligned}
}{
\begin{aligned}
\sqrt{\pi}\ e^{- \frac{u_b}{\epsilon}} - {\textnormal{erfc}}({A_\epsilon}) + {\textnormal{erfc}} (B_\epsilon) \left( 1 - e^{- \frac{u_b}{\epsilon}} \right) + {\textnormal{erfc}} \left( P_\epsilon \right) e^{P_\epsilon^2 - A_\epsilon^2} 
\end{aligned} } ,\ &{ x > a + u_a t },
\\
\\
\frac{
\begin{aligned}
\rho_c \bigg[ \sqrt{2 t \epsilon} \left( \frac{1}{e^{A_\epsilon^2}} - \frac{1}{2\ e^{C_\epsilon^2}} \right) + \left( x - c - u_a t \right) \left( \sqrt{\pi} - \textnormal{erfc} \left( P_\epsilon \right) \right) e^{P_\epsilon^2 - A_\epsilon^2}
\\
+ (x-c) \left( \textnormal{erfc} \left( C_\epsilon \right) - \textnormal{erfc} \left( A_\epsilon \right) \right) \bigg] + \rho_d\ \textnormal{erfc} \left( D_\epsilon \right) e^{- \frac{u_b}{\epsilon}}
\end{aligned}
}{
\begin{aligned}
\sqrt{\pi}\ e^{- \frac{u_b}{\epsilon}} - {\textnormal{erfc}}({A_\epsilon}) + {\textnormal{erfc}} (B_\epsilon) \left( 1 - e^{- \frac{u_b}{\epsilon}} \right) + \left( \sqrt{\pi} - {\textnormal{erfc}} \left( P_\epsilon \right) \right) e^{P_\epsilon^2 - A_\epsilon^2} 
\end{aligned} },\ &{ x < a + u_a t }.
\end{cases} 
\end{aligned}
\]
\item $x > d$

\[
\begin{aligned}
u^\epsilon 
&= \begin{cases}
\frac{\begin{aligned}
u_a\ {\textnormal{erfc}} \left( P_\epsilon \right) e^{P_\epsilon^2 - A_\epsilon^2} + \frac{\epsilon}{\sqrt{2t \epsilon}} \left( 1 - e^{- \frac{u_b}{\epsilon}} \right) e^{- B_\epsilon^2}
\end{aligned}}{
\begin{aligned}
\sqrt{\pi}\ e^{- \frac{u_b}{\epsilon}} - {\textnormal{erfc}}({A_\epsilon}) + {\textnormal{erfc}} (B_\epsilon) \left( 1 - e^{- \frac{u_b}{\epsilon}} \right) + {\textnormal{erfc}} \left( P_\epsilon \right) e^{P_\epsilon^2 - A_\epsilon^2} 
\end{aligned} },\ &{ x > a + u_a t },
\\
\\
\frac{\begin{aligned}
u_a\ \left( \sqrt{\pi} - {\textnormal{erfc}} \left( P_\epsilon \right) \right) e^{P_\epsilon^2 - A_\epsilon^2} + \frac{\epsilon}{\sqrt{2t \epsilon}} \left( 1 - e^{- \frac{u_b}{\epsilon}} \right) e^{- B_\epsilon^2}
\end{aligned}}{
\begin{aligned}
\sqrt{\pi}\ e^{- \frac{u_b}{\epsilon}} - {\textnormal{erfc}}({A_\epsilon}) + {\textnormal{erfc}} (B_\epsilon) \left( 1 - e^{- \frac{u_b}{\epsilon}} \right) + \left( \sqrt{\pi} - {\textnormal{erfc}} \left( P_\epsilon \right) \right) e^{P_\epsilon^2 - A_\epsilon^2} 
\end{aligned} },\ &{ x < a + u_a t },
\end{cases} 
\\
\\
R^\epsilon 
&= \begin{cases}
\frac{
\begin{aligned}
\rho_c \bigg[ \sqrt{2 t \epsilon} \left( \frac{1}{e^{A_\epsilon^2}} - \frac{1}{2\ e^{C_\epsilon^2}} \right) + \left( x - c - u_a t \right) \textnormal{erfc} \left( P_\epsilon \right) e^{P_\epsilon^2 - A_\epsilon^2}
\\
+ (x-c) \left( \textnormal{erfc} \left( C_\epsilon \right) - \textnormal{erfc} \left( A_\epsilon \right) \right) \bigg] + \rho_d \left( \sqrt{\pi} - \textnormal{erfc} \left( D_\epsilon \right) \right) e^{- \frac{u_b}{\epsilon}}
\end{aligned}
}{
\begin{aligned}
\sqrt{\pi}\ e^{- \frac{u_b}{\epsilon}} - {\textnormal{erfc}}({A_\epsilon}) + {\textnormal{erfc}} (B_\epsilon) \left( 1 - e^{- \frac{u_b}{\epsilon}} \right) + {\textnormal{erfc}} \left( P_\epsilon \right) e^{P_\epsilon^2 - A_\epsilon^2} 
\end{aligned} } ,\ &{ x > a + u_a t },
\\
\\
\frac{
\begin{aligned}
\rho_c \bigg[ \sqrt{2 t \epsilon} \left( \frac{1}{e^{A_\epsilon^2}} - \frac{1}{2\ e^{C_\epsilon^2}} \right) + \left( x - c - u_a t \right) \left( \sqrt{\pi} - \textnormal{erfc} \left( P_\epsilon \right) \right) e^{P_\epsilon^2 - A_\epsilon^2}
\\
+ (x-c) \left( \textnormal{erfc} \left( C_\epsilon \right) - \textnormal{erfc} \left( A_\epsilon \right) \right) \bigg] + \rho_d \left( \sqrt{\pi} - \textnormal{erfc} \left( D_\epsilon \right) \right) e^{- \frac{u_b}{\epsilon}}
\end{aligned}
}{
\begin{aligned}
\sqrt{\pi}\ e^{- \frac{u_b}{\epsilon}} - {\textnormal{erfc}}({A_\epsilon}) + {\textnormal{erfc}} (B_\epsilon) \left( 1 - e^{- \frac{u_b}{\epsilon}} \right) + \left( \sqrt{\pi} - {\textnormal{erfc}} \left( P_\epsilon \right) \right) e^{P_\epsilon^2 - A_\epsilon^2} 
\end{aligned} },\ &{ x < a + u_a t }.
\end{cases} 
\end{aligned}
\]
\end{itemize}
Depending on the relative positions of $u_a$ and $u_b$, we study the asymptotic behavior of $\left( u^\epsilon , R^\epsilon \right)$ as $\epsilon \rightarrow 0$ in each of the regions mentioned above. For discussing the passage to the limit, we will extensively use the following asymptotic properties of the function erfc:

\begin{enumerate}
	
	\item $\lim_{z \rightarrow \infty} \textnormal{erfc} \left( z \right) = 0$
	
	\item $\textnormal{erfc} \left( z \right) = \left( \frac{1}{2 z} - \frac{1}{4 z^3} + o \left( \frac{1}{z^3} \right) \right)\ e^{- z^2} \textnormal{ as } z \rightarrow \infty$
	
	\item $\lim_{z \rightarrow \infty}\ f(z) = \frac{1}{2}$, where $f : z \longmapsto z\ \textnormal{erfc} \left( z \right) e^{z^2}$ for every $z \in \textbf{R}^1$

\end{enumerate}
These properties have been proved in \cite{das-1-submitted}. However, for the sake of completeness, the derivations have again been provided in the appendix.
	
\textbf{Case 1.} $u_a < 0$, $u_b > 0$
	
\begin{enumerate}

	\item $x < a$
	
Within this region, the limit ${\displaystyle{\lim_{\epsilon \rightarrow 0}\ \left( u^\epsilon , R^\epsilon \right)}}$ has to be separately evaluated in the subregions $x > a + u_a t$ and $x < a + u_a t$ as follows:
	
\subsection*{Subregion 1. $x > a + u_a t$}

\[
\begin{aligned}
&\lim_{\epsilon \rightarrow 0}
\begin{pmatrix}
\frac{\begin{aligned}
\frac{\left( b-x \right) f(P_\epsilon)}{x-a-u_a t} \cdot u_a + \frac{(b-x) \left( 1 - e^{- \frac{| u_b |}{\epsilon}} \right)}{2t\ e^{B_\epsilon^2 - A_\epsilon^2}}
\end{aligned}}{
\begin{aligned}
\frac{\left( b-x \right) f(A_\epsilon)}{a-x} + \frac{\left( e^{- \frac{| u_b |}{\epsilon}} - 1 \right) f(B_\epsilon)}{e^{B_\epsilon^2 - A_\epsilon^2}} + \frac{\left( b-x \right) f(P_\epsilon)}{x-a-u_a t}
\end{aligned} }
\\
--------------------------
\\
\frac{
\begin{aligned}
\frac{\rho_c \left( b-x \right)}{2} + \rho_c \left( x-c-u_a t \right) \frac{\left( b-x \right) f(P_\epsilon)}{x-a-u_a t}
\\
\\
+ \rho_c \left( x-c \right) \left( \frac{\left( b-x \right) f(A_\epsilon)}{\left( a-x \right)} - \frac{\left( b-x \right) f(C_\epsilon)}{\left( c-x \right) e^{C_\epsilon^2 - A_\epsilon^2}} \right)
\\
\\
+\ \rho_c\ \frac{\left( b-x \right)}{2} \left( 1 - e^{A_\epsilon^2 - C_\epsilon^2} \right) + \rho_d\ \frac{\left( b-x \right) f(D_\epsilon)}{\left( d-x \right) e^{D_\epsilon^2 - A_\epsilon^2}} \cdot e^{- \frac{| u_b |}{\epsilon}}
\end{aligned}
}{
\begin{aligned}
\frac{\left( b-x \right) f(A_\epsilon)}{a-x} + \frac{\left( e^{- \frac{| u_b |}{\epsilon}} - 1 \right) f(B_\epsilon)}{e^{B_\epsilon^2 - A_\epsilon^2}} + \frac{\left( b-x \right) f(P_\epsilon)}{x-a-u_a t}
\end{aligned} }
\end{pmatrix}^T
\\
= &\left( \frac{x-a}{t} , \rho_c\ \frac{2 \left( x-a-u_a t \right) \left( x-a \right) + \left( a-c \right) u_a t}{u_a t} \right).
\end{aligned}
\]

\subsection*{Subregion 2. $x < a + u_a t$}	
	
\[
\begin{aligned}
&\lim_{\epsilon \rightarrow 0}
\begin{pmatrix}
\frac{\begin{aligned}
u_a \left( \sqrt{\pi} - {\textnormal{erfc}} \left( P_\epsilon \right) \right) + \frac{b-x}{2t} \cdot \frac{1 - e^{- \frac{| u_b |}{\epsilon}}}{B_\epsilon\ e^{B_\epsilon^2 + P_\epsilon^2 - A_\epsilon^2}}
\end{aligned}}{
\begin{aligned}
\frac{\left( b-x \right) f(A_\epsilon)}{\left( a-x \right) B_\epsilon\ e^{P_\epsilon^2}} + \frac{\left( e^{- \frac{| u_b |}{\epsilon}} - 1 \right) f(B_\epsilon)}{B_\epsilon\ e^{B_\epsilon^2 + P_\epsilon^2 - A_\epsilon^2}} + \left( \sqrt{\pi} - {\textnormal{erfc}} \left( P_\epsilon \right) \right)
\end{aligned} }
\\
--------------------------
\\
\frac{\begin{aligned}
\frac{\rho_c \left( b-x \right)}{2\ B_\epsilon\ e^{P_\epsilon^2}} + \rho_c \left( x-c-u_a t \right) \left( \sqrt{\pi} - {\textnormal{erfc}} \left( P_\epsilon \right) \right) 
\\
\\
+\ \rho_c \frac{\left( b-x \right)}{2} \left( \frac{1}{B_\epsilon\ e^{P_\epsilon^2}} - \frac{1}{B_\epsilon\ e^{C_\epsilon^2 + P_\epsilon^2 - A_\epsilon^2}} \right) 
\\
\\
+ \rho_c \left( x-c \right) \left( \frac{\left( b-x \right) f(A_\epsilon)}{\left( a-x \right) B_\epsilon\ e^{P_\epsilon^2}} - \frac{\left( b-x \right) f(C_\epsilon)}{\left( c-x \right) B_\epsilon\ e^{C_\epsilon^2 + P_\epsilon^2 - A_\epsilon^2}} \right)
\\
\\
+ \rho_d \frac{\left( b-x \right) f(D_\epsilon)}{\left( d-x \right) B_\epsilon\ e^{D_\epsilon^2 + P_\epsilon^2 - A_\epsilon^2}} \cdot e^{- \frac{| u_b |}{\epsilon}}
\end{aligned}}{\begin{aligned}
\frac{\left( b-x \right) f(A_\epsilon)}{\left( a-x \right) B_\epsilon\ e^{P_\epsilon^2}} + \frac{\left( e^{- \frac{| u_b |}{\epsilon}} - 1 \right) f(B_\epsilon)}{B_\epsilon\ e^{B_\epsilon^2 + P_\epsilon^2 - A_\epsilon^2}} + \left( \sqrt{\pi} - {\textnormal{erfc}} \left( P_\epsilon \right) \right)
\end{aligned} } 
\end{pmatrix}^T
\\
= &\left( u_a , \rho_c \left( x-c-u_a t \right) \right).
\end{aligned}
\]
Here we have used the strict inequalities

\[
\begin{aligned}
&(b-x)^2 + (a+u_a t-x)^2 - (a-x)^2 = (a+u_a t-x)^2 + 2 \left( b-a \right) \left| \frac{a+b}{2} - x \right| > 0,
\\
&(d-x)^2 + (a+u_a t-x)^2 - (a-x)^2 = (a+u_a t-x)^2 + 2 \left( d-a \right) \left| \frac{a+d}{2} - x \right| > 0.
\end{aligned}
\]
For the remaining regions, the restriction $u_a < 0$ will automatically imply that $x > a + u_a t$. Therefore, the limit ${\displaystyle{\lim_{\epsilon \rightarrow 0} \left( u^\epsilon , R^\epsilon \right)}}$ in each of these regions will be evaluated as follows:

	\item $a < x < c$
\[
\begin{aligned}
&\lim_{\epsilon \rightarrow 0} \begin{pmatrix}
\frac{\begin{aligned}
u_a\ \frac{f(P_\epsilon)}{P_\epsilon\ e^{A_\epsilon^2}} + \frac{\epsilon}{\sqrt{2t \epsilon}} \left( 1 - e^{- \frac{| u_b |}{\epsilon}} \right) e^{- B_\epsilon^2}
\end{aligned}}{
\begin{aligned}
\sqrt{\pi} - {\textnormal{erfc}}(A_\epsilon) + {\textnormal{erfc}}(B_\epsilon) \left( e^{- \frac{| u_b |}{\epsilon}} - 1 \right) + \frac{f(P_\epsilon)}{P_\epsilon\ e^{A_\epsilon^2}}
\end{aligned} }
\\
-------------------------
\\
\frac{
\begin{aligned}
\sqrt{2 t \epsilon}\ \rho_c\ e^{- A_\epsilon^2} + \rho_c \left( x - c - u_a t \right) \frac{f(P_\epsilon)}{P_\epsilon\ e^{A_\epsilon^2}}
\\
- \frac{t \epsilon}{\sqrt{2 t \epsilon}}\ \rho_c\ e^{- C_\epsilon^2} + \rho_c (x-c) \left( \sqrt{\pi} - \textnormal{erfc} \left( A_\epsilon \right) - \textnormal{erfc} \left( C_\epsilon \right) \right)
\\
+ \rho_d\ e^{- \frac{| u_b |}{\epsilon}}\ \textnormal{erfc} \left( D_\epsilon \right)\
\end{aligned}
}{
\begin{aligned}
\sqrt{\pi} - {\textnormal{erfc}}(A_\epsilon) + {\textnormal{erfc}}(B_\epsilon) \left( e^{- \frac{| u_b |}{\epsilon}} - 1 \right) + \frac{f(P_\epsilon)}{P_\epsilon\ e^{A_\epsilon^2}} 
\end{aligned} }
\end{pmatrix}^T
= \left( 0 , \rho_c \left( x-c \right) \right).
\end{aligned}
\]	

	\item $c < x < b$

\[
\begin{aligned}
&\lim_{\epsilon \rightarrow 0} \begin{pmatrix}
\frac{\begin{aligned}
u_a\ \frac{f(P_\epsilon)}{P_\epsilon\ e^{A_\epsilon^2}} + \frac{\epsilon}{\sqrt{2t \epsilon}} \cdot \left( 1 - e^{- \frac{| u_b |}{\epsilon}} \right) e^{- B_\epsilon^2}
\end{aligned}}{
\begin{aligned}
\sqrt{\pi} - {\textnormal{erfc}}(A_\epsilon) + {\textnormal{erfc}}(B_\epsilon) \left( e^{- \frac{| u_b |}{\epsilon}} - 1 \right) + \frac{f(P_\epsilon)}{P_\epsilon\ e^{A_\epsilon^2}}
\end{aligned} }
\\
----------------------
\\
\frac{
\begin{aligned}
\sqrt{2 t \epsilon}\ \rho_c\ e^{- A_\epsilon^2} + \rho_c ( x - c - u_a t )\ \frac{f(P_\epsilon)}{P_\epsilon\ e^{A_\epsilon^2}}
\\
- \frac{t \epsilon}{\sqrt{2 t \epsilon}}\ \rho_c\ e^{- C_\epsilon^2} + \rho_c (x-c) \left( \textnormal{erfc} \left( C_\epsilon \right) - \textnormal{erfc} \left( A_\epsilon \right) \right)
\\
+ \rho_d\ e^{- \frac{| u_b |}{\epsilon}}\ \textnormal{erfc} \left( D_\epsilon \right)\
\end{aligned}
}{
\begin{aligned}
\sqrt{\pi} - {\textnormal{erfc}}(A_\epsilon) + {\textnormal{erfc}}(B_\epsilon) \left( e^{- \frac{| u_b |}{\epsilon}} - 1 \right) + \frac{f(P_\epsilon)}{P_\epsilon\ e^{A_\epsilon^2}}
\end{aligned} }
\end{pmatrix}
= \left( 0 , 0 \right).
\end{aligned} 
\]	

	\item $b < x < d$
	
\[
\begin{aligned}
&\lim_{\epsilon \rightarrow 0} \begin{pmatrix}
\frac{\begin{aligned}
u_a\ \frac{\left( x-b \right) f(P_\epsilon)}{\left( x-a-u_a t \right) e^{A_\epsilon^2 - B_\epsilon^2}} + \frac{x-b}{2t} \left( 1 - e^{- \frac{| u_b |}{\epsilon}} \right)
\end{aligned}}{
\begin{aligned}
\sqrt{\pi}\ B_\epsilon\ e^{B_\epsilon^2 - \frac{u_b}{\epsilon}} - \frac{\left( x-b \right) f(A_\epsilon)}{\left( x-a \right) e^{A_\epsilon^2 - B_\epsilon^2}} + f(B_\epsilon) \left( 1 - e^{- \frac{| u_b |}{\epsilon}} \right) + \frac{\left( x-b \right) f(P_\epsilon)}{\left( x-a-u_a t \right) e^{A_\epsilon^2 - B_\epsilon^2}}
\end{aligned} }
\\
-----------------------------------
\\
\frac{
\begin{aligned}
\rho_c\ \frac{x-b}{e^{A_\epsilon^2 - B_\epsilon^2}} + \rho_c \frac{\left( x-c-u_a t \right) \left( x-b \right) f(P_\epsilon)}{\left( x-a-u_a t \right) e^{A_\epsilon^2 - B_\epsilon^2}} 
\\
- \rho_c\ \frac{x-b}{2\ e^{C_\epsilon^2 - B_\epsilon^2}} + \rho_c \left( x-c \right) \left( \frac{\left( x-b \right) f(C_\epsilon)}{\left( x-c \right) e^{C_\epsilon^2 - B_\epsilon^2}} - \frac{\left( x-b \right) f(A_\epsilon)}{\left( x-a \right) e^{A_\epsilon^2 - B_\epsilon^2}} \right)
\\
+ \rho_d\ \frac{\left( x-b \right) f(D_\epsilon)}{\left( d-x \right) e^{D_\epsilon^2 - B_\epsilon^2}} e^{- \frac{| u_b |}{\epsilon}}
\end{aligned}
}{
\begin{aligned}
\sqrt{\pi}\ B_\epsilon\ e^{B_\epsilon^2 - \frac{u_b}{\epsilon}} - \frac{\left( x-b \right) f(A_\epsilon)}{\left( x-a \right) e^{A_\epsilon^2 - B_\epsilon^2}} + f(B_\epsilon) \left( 1 - e^{- \frac{| u_b |}{\epsilon}} \right) + \frac{\left( x-b \right) f(P_\epsilon)}{\left( x-a-u_a t \right) e^{A_\epsilon^2 - B_\epsilon^2}}
\end{aligned} }
\end{pmatrix}^T
\\
= &\begin{cases}
\left( 0 , 0 \right),\ &{x > b + \sqrt{2 u_b t}},
\\
\left( \frac{x-b}{t} , 0 \right),\ &{x < b + \sqrt{2 u_b t}}.
\end{cases}
\end{aligned}
\]	

	\item $x > d$ 
	
\[
\begin{aligned}
&\lim_{\epsilon \rightarrow 0} \begin{pmatrix}
\frac{\begin{aligned}
u_a\ \frac{\left( x-b \right) f(P_\epsilon)}{\left( x-a-u_a t \right) e^{A_\epsilon^2 - B_\epsilon^2}} + \frac{x-b}{2t} \left( 1 - e^{- \frac{| u_b |}{\epsilon}} \right)
\end{aligned}}{
\begin{aligned}
\sqrt{\pi}\ B_\epsilon\ e^{B_\epsilon^2 - \frac{u_b}{\epsilon}} - \frac{\left( x-b \right) f(A_\epsilon)}{\left( x-a \right) e^{A_\epsilon^2 - B_\epsilon^2}} + f(B_\epsilon) \left( 1 - e^{- \frac{| u_b |}{\epsilon}} \right) + \frac{\left( x-b \right) f(P_\epsilon)}{\left( x-a-u_a t \right) e^{A_\epsilon^2 - B_\epsilon^2}}
\end{aligned} }
\\
-----------------------------------
\\
\frac{
\begin{aligned}
\rho_c\ \frac{x-b}{e^{A_\epsilon^2 - B_\epsilon^2}} + \rho_c \frac{\left( x-c-u_a t \right) \left( x-b \right) f(P_\epsilon)}{\left( x-a-u_a t \right) e^{A_\epsilon^2 - B_\epsilon^2}} 
\\
- \rho_c\ \frac{x-b}{2\ e^{C_\epsilon^2 - B_\epsilon^2}} + \rho_c \left( x-c \right) \left( \frac{\left( x-b \right) f(C_\epsilon)}{\left( x-c \right) e^{C_\epsilon^2 - B_\epsilon^2}} - \frac{\left( x-b \right) f(A_\epsilon)}{\left( x-a \right) e^{A_\epsilon^2 - B_\epsilon^2}} \right)
\\
+ \rho_d\ B_\epsilon\ e^{B_\epsilon^2 - \frac{u_b}{\epsilon}} \left( \sqrt{\pi} - \textnormal{erfc} \left( D_\epsilon \right) \right)
\end{aligned}
}{
\begin{aligned}
\sqrt{\pi}\ B_\epsilon\ e^{B_\epsilon^2 - \frac{u_b}{\epsilon}} - \frac{\left( x-b \right) f(A_\epsilon)}{\left( x-a \right) e^{A_\epsilon^2 - B_\epsilon^2}} + f(B_\epsilon) \left( 1 - e^{- \frac{| u_b |}{\epsilon}} \right) + \frac{\left( x-b \right) f(P_\epsilon)}{\left( x-a-u_a t \right) e^{A_\epsilon^2 - B_\epsilon^2}}
\end{aligned} }
\end{pmatrix}^T
\\
= &\begin{cases}
\left( 0 , \rho_d \right),\ &{x > b + \sqrt{2 u_b t}},
\\
\left( \frac{x-b}{t} , 0 \right),\ &{x < b + \sqrt{2 u_b t}}.
\end{cases}
\end{aligned}
\]	
\end{enumerate}

Now, to recover $\rho$, set $u = \lim_{\epsilon \rightarrow 0} u^\epsilon$ and $R = \lim_{\epsilon \rightarrow 0} R^\epsilon$. For each $s \geq 0$, let us define

\begin{itemize}

	\item[$(i)$] $r(s) := a + u_a s$,

	\item[$(ii)$] $p(s) := b + \sqrt{2 u_b s}$.

\end{itemize}
The explicit structure of $u$ and $R$ under the present case can then be described as follows:

\[
\begin{aligned}
u(x,t) &= \begin{cases}
u_a,\ &{ x < r(t) },
\\
\frac{x-a}{t},\ &{ x \in \Big( r(t) , a \Big) },
\\
0,\ &{ x \in \Big( a , b \Big) \cup \Big( p(t) , \infty \Big) },
\\
\frac{x-b}{t},\ &{ x \in \Big( b , p(t) \Big) },
\end{cases}
\\
R(x,t) &= \begin{cases}
\rho_c \left( x-c-u_a t \right),\ &{ x \in \Big( -\infty , r(t) \Big) },
\\
\rho_c\ \frac{2 \left( x-a-u_a t \right) \left( x-a \right) + \left( a-c \right) u_a t}{u_a t},\ &{ x \in \Big( r(t) , a \Big) },
\\
\rho_c \left( x-c \right),\ &{ x \in \Big( a , c \Big) },
\\
0,\ &x \in { \big( c , d \big) } \cup { \big( d , p(t) \big) },
\\
\rho_d,\ &{ x \in \Big( \max{ \big\{ d , p(t) \big\} } , \infty \Big) }.
\end{cases}
\end{aligned}
\]
The next step is to consider an arbitrary test function $\phi \in C^{\infty}_{c} \left( \textbf{R}^1 \times \left[ 0 , \infty \right) ; \textbf{R}^1 \right)$ and consider the action of the distributional derivative $R_x$ of $R$ with respect to the space variable $x$ on $\phi$. We see that

\[
\begin{aligned}
\left\langle R_x , \phi \right\rangle = &- \left\langle R , \phi_x \right\rangle
\\
= &- \rho_c \int_{0}^{\infty} \Bigg[ \int_{-\infty}^{a+u_a t} \left( x-c-u_a t \right) \phi_x\ dx + \int_{a+u_a t}^{a} \frac{2 \left( x-a-u_a t \right) \left( x-a \right) + \left( a-c \right) u_a t}{u_a t} \phi_x\ dx
\\
&+ \int_{a}^{c} \left( x-c \right) \phi_x\ dx \Bigg] dt - \rho_d \Bigg[ \int_{0}^{\infty} \int_{\gamma_{ {}_{d} } (t)}^{\infty}\ \phi_x\ dx dt \Bigg]
\\
= &\rho_c \Bigg[ \int_{0}^{\infty} \int_{-\infty}^{a + u_a t} \phi \left( x , t \right) dx dt + \left( c-a \right) \int_{0}^{\infty} \phi \left( a + u_a t , t \right) dt
\\
&+ \frac{4}{u_a t} \int_{0}^{\infty} \int_{a + u_a t}^{a} \left( x-a \right) \phi \left( x , t \right) dx dt - 2 \int_{0}^{\infty} \int_{a + u_a t}^{a} \phi \left( x , t \right) dx dt
\\
&+ \left( c-a \right) \int_{0}^{\infty} \phi \left( a , t \right) dt - \left( c-a \right) \int_{0}^{\infty} \phi \left( a + u_a t , t \right) dt
\\
&+ \int_{0}^{\infty} \int_{a}^{c} \phi \left( x , t \right) dx dt - \left( c-a \right) \int_{0}^{\infty} \phi \left( a , t \right) dt \Bigg] + \rho_d \int_{0}^{\infty} \phi \left( \gamma_{ {}_{d} } (t) , t \right) dt
\\
= &\left\langle \rho_c \left( \chi_{ {}_{ \left( -\infty , a + u_a t \right) } } + \frac{4 \left( x-a \right) - 2 u_a t}{u_a t} \chi_{ {}_{ \left( a + u_a t , a \right) } } + \chi_{ {}_{ \left( a , c \right) } } \right) + \rho_d\ \delta_{ x = \gamma_{ {}_{d} } (t) } , \phi \right\rangle,
\end{aligned}
\]
where the curve $x = \gamma_{ {}_{d} } (t)$ is defined on $\left[ 0 , \infty \right)$ by 

\[
\begin{aligned}
\gamma_{ {}_{d} } (t) := \begin{cases}
d,\ &{ 0 \leq t \leq t^* := \frac{(d-b)^2}{2 u_b} },
\\
b + \sqrt{2 u_b t},\ &{ t > t^* }.
\end{cases}
\end{aligned} 
\]
Therefore $\rho = \rho_c \left( \chi_{ {}_{ {}_{ \left( -\infty , a + u_a t \right) } } } + \frac{4 \left( x-a \right) - 2 u_a t}{u_a t} \chi_{ {}_{ {}_{ \left( a + u_a t , a \right) } } } + \chi_{ {}_{ {}_{ \left( a , c \right) } } } \right) + \rho_d\ \delta_{ x = \gamma_{ {}_{d} } (t) }$.

\textbf{Case 2.} $u_a > 0$, $u_b > 0$
	
\begin{enumerate}

	\item $x < a$

In this region, we have $x < a + u_a t$, since $u_a > 0$. Hence $\lim_{\epsilon \rightarrow 0}\ \left( u^\epsilon , R^\epsilon \right)$ equals	

\[
\begin{aligned}
&\lim_{\epsilon \rightarrow 0} 
\begin{pmatrix}
\frac{
\begin{aligned}
u_a \left( \sqrt{\pi} - {\textnormal{erfc}} \left( P_\epsilon \right) \right) + \frac{a-x}{2t} \cdot \frac{1 - e^{- \frac{| u_b |}{\epsilon}}}{A_\epsilon\ e^{P_\epsilon^2 + B_\epsilon^2 - A_\epsilon^2}}
\end{aligned}
}{
\begin{aligned}
\frac{f(A_\epsilon)}{A_\epsilon\ e^{P_\epsilon^2}} + \frac{f(B_\epsilon)}{B_\epsilon\ e^{P_\epsilon^2 + B_\epsilon^2 - A_\epsilon^2}} + \left( \sqrt{\pi} - {\textnormal{erfc}} \left( P_\epsilon \right) \right)
\end{aligned} }
\\
----------------------------
\\
\frac{
\begin{aligned}
\rho_c \left[ \frac{\left( a-x \right)}{A_\epsilon\ e^{P_\epsilon^2}} + \left( x-c-u_a t \right) \left( \sqrt{\pi} - \textnormal{erfc} \left( P_\epsilon \right) \right) - \frac{a-x}{2 A_\epsilon\ e^{P_\epsilon^2 + C_\epsilon^2 - A_\epsilon^2}} \right.
\\
+ \left. \left( x-c \right) \left( \frac{f(A_\epsilon)}{A_\epsilon\ e^{P_\epsilon^2}} - \frac{f(C_\epsilon)}{C_\epsilon\ e^{P_\epsilon^2 + C_\epsilon^2 - A_\epsilon^2}} \right) \right] + \rho_d\ \frac{f(D_\epsilon) \cdot e^{- \frac{| u_b |}{\epsilon}}}{D_\epsilon\ e^{P_\epsilon^2 + D_\epsilon^2 - A_\epsilon^2}}
\end{aligned}
}{
\begin{aligned}
\frac{f(A_\epsilon)}{A_\epsilon\ e^{P_\epsilon^2}} + \frac{f(B_\epsilon)}{B_\epsilon\ e^{P_\epsilon^2 + B_\epsilon^2 - A_\epsilon^2}} + \left( \sqrt{\pi} - {\textnormal{erfc}} \left( P_\epsilon \right) \right)
\end{aligned} }
\end{pmatrix}^T
\\
= &\left( u_a , \rho_c \left( x-c-u_a t \right) \right).
\end{aligned}
\]	
Here we have used the strict inequalities

\[
\begin{aligned}
&(x-a-u_a t)^2 + (x-b)^2 - (x-a)^2 = (x-a-u_a t)^2 + 2 \left| x - \frac{a+b}{2} \right| (b-a) > 0,
\\
&(x-a-u_a t)^2 + (x-d)^2 - (x-a)^2 = (x-a-u_a t)^2 + 2 \left| x - \frac{a+d}{2} \right| (d-a) > 0.
\end{aligned}
\]
In each of the remaining regions under this case, the limit $\lim_{\epsilon \rightarrow 0}\ \left( u^\epsilon , R^\epsilon \right)$ has to be evaluated separately in the subregions $x > a + u_a t$ and $x < a + u_a t$.

	\item $a < x < c$

\subsection*{Subregion 1. $x > a + u_a t$}

\[
\begin{aligned}
&\lim_{\epsilon \rightarrow 0} 
\begin{pmatrix}
\frac{\begin{aligned}
u_a\ \frac{f(P_\epsilon)}{P_\epsilon\ e^{A_\epsilon^2}} + \frac{\left( b-x \right) \left( 1 - e^{- \frac{| u_b |}{\epsilon}} \right)}{2t\ B_\epsilon\ e^{B_\epsilon^2}}
\end{aligned}}{
\begin{aligned}
\left( \sqrt{\pi} - {\textnormal{erfc}} \left( A_\epsilon \right) \right) + {\textnormal{erfc}} \left( B_\epsilon \right) \left( e^{- \frac{| u_b |}{\epsilon}} - 1 \right) + \frac{f(P_\epsilon)}{P_\epsilon\ e^{A_\epsilon^2}}
\end{aligned} }
\\
--------------------------
\\
\frac{
\begin{aligned}
\rho_c \bigg[ \sqrt{2 t \epsilon} \left( e^{- A_\epsilon^2} - \frac{1}{2\ e^{C_\epsilon^2}} \right) + \frac{\left( x-c-u_a t \right) f(P_\epsilon)}{P_\epsilon\ e^{A_\epsilon^2}}
\\
+ \left( x-c \right) \left( \sqrt{\pi} - {\textnormal{erfc}} \left( A_\epsilon \right) - {\textnormal{erfc}} \left( C_\epsilon \right) \right) \bigg] + \rho_d\ {\textnormal{erfc}} \left( D_\epsilon \right) e^{- \frac{| u_b |}{\epsilon}}
\end{aligned} }{
\begin{aligned}
\left( \sqrt{\pi} - {\textnormal{erfc}} \left( A_\epsilon \right) \right) + {\textnormal{erfc}} \left( B_\epsilon \right) \left( e^{- \frac{| u_b |}{\epsilon}} - 1 \right) + \frac{f(P_\epsilon)}{P_\epsilon\ e^{A_\epsilon^2}}
\end{aligned} }
\end{pmatrix}^T
= &\left( 0 , \rho_c \left( x-c \right) \right).
\end{aligned}
\]

\subsection*{Subregion 2. $x < a + u_a t$}	

\[
\begin{aligned}
&\lim_{\epsilon \rightarrow 0} 
\begin{pmatrix}
\frac{
\begin{aligned}
u_a \left( \sqrt{\pi} - {\textnormal{erfc}} \left( P_\epsilon \right) \right) e^{P_\epsilon^2 - A_\epsilon^2} + \frac{\left( x-a \right) \left( 1 - e^{- \frac{| u_b |}{\epsilon}} \right)}{2t\ A_\epsilon\ e^{B_\epsilon^2}}
\end{aligned} }{
\begin{aligned}
\left( \sqrt{\pi} - {\textnormal{erfc}} \left( A_\epsilon \right) \right) + {\textnormal{erfc}} \left( B_\epsilon \right) \left( e^{- \frac{| u_b |}{\epsilon}} - 1 \right) + \left( \sqrt{\pi} - {\textnormal{erfc}} \left( P_\epsilon \right) \right) e^{P_\epsilon^2 - A_\epsilon^2}
\end{aligned} }
\\
------------------------------
\\
\frac{
\begin{aligned}
\rho_c \bigg[ \sqrt{2 t \epsilon} \left( e^{- A_\epsilon^2} - \frac{1}{2\ e^{C_\epsilon^2}} \right) + \left( x-c-u_a t \right) \left( \sqrt{\pi} - {\textnormal{erfc}} \left( P_\epsilon \right) \right) e^{P_\epsilon^2 - A_\epsilon^2}
\\
+ \left( x-c \right) \left( \sqrt{\pi} - {\textnormal{erfc}} \left( A_\epsilon \right) - {\textnormal{erfc}} \left( C_\epsilon \right) \right) \bigg] + \rho_d\ {\textnormal{erfc}} \left( D_\epsilon \right) e^{- \frac{| u_b |}{\epsilon}}
\end{aligned} }{
\begin{aligned}
\left( \sqrt{\pi} - {\textnormal{erfc}} \left( A_\epsilon \right) \right) + {\textnormal{erfc}} \left( B_\epsilon \right) \left( e^{- \frac{| u_b |}{\epsilon}} - 1 \right) + \left( \sqrt{\pi} - {\textnormal{erfc}} \left( P_\epsilon \right) \right) e^{P_\epsilon^2 - A_\epsilon^2}
\end{aligned} }
\end{pmatrix}^T
\\
= &\begin{cases}
\left( u_a , \rho_c \left( x-c-u_a t \right) \right),\ &{ x \in \Big( a , a + \frac{u_a}{2} \cdot t \Big) },
\\
\left( 0 , \rho_c \left( x-c \right) \right),\ &x \in  \Big( a + \frac{u_a}{2} \cdot t , a + u_a t \Big).
\end{cases}
\end{aligned}
\]

	\item $c < x < b$
	
\subsection*{Subregion 1. $x > a + u_a t$}

\[
\begin{aligned}
&\lim_{\epsilon \rightarrow 0}
\begin{pmatrix}
\frac{
\begin{aligned}
u_a\ \frac{f(P_\epsilon)}{P_\epsilon\ e^{A_\epsilon^2}} + \frac{(x-a) \left( 1 - e^{- \frac{| u_b |}{\epsilon}} \right)}{2t\ A_\epsilon\ e^{B_\epsilon^2}}
\end{aligned}
}{
\begin{aligned}
\left( \sqrt{\pi} - {\textnormal{erfc}} \left( A_\epsilon \right) \right) + {\textnormal{erfc}} \left( B_\epsilon \right) \left( e^{- \frac{| u_b |}{\epsilon}} - 1 \right) + \frac{f(P_\epsilon)}{P_\epsilon\ e^{A_\epsilon^2}}
\end{aligned}
}
\\
------------------------
\\
\frac{
\begin{aligned}
\rho_c\ \bigg[ \sqrt{2 t \epsilon} \left( \frac{1}{e^{A_\epsilon^2}} - \frac{1}{2\ e^{C_\epsilon^2}} \right) + \left( x-c-u_a t \right) \frac{f(P_\epsilon)}{P_\epsilon\ e^{A_\epsilon^2}}
\\
+ \left( x-c \right) \left( {\textnormal{erfc}} \left( C_\epsilon \right) - {\textnormal{erfc}} \left( A_\epsilon \right) \right) \bigg] + \rho_d\ {\textnormal{erfc}} \left( D_\epsilon \right) e^{- \frac{| u_b |}{\epsilon}}
\end{aligned}
}{
\begin{aligned}
\left( \sqrt{\pi} - {\textnormal{erfc}} \left( A_\epsilon \right) \right) + {\textnormal{erfc}} \left( B_\epsilon \right) \left( e^{- \frac{| u_b |}{\epsilon}} - 1 \right) + \frac{f(P_\epsilon)}{P_\epsilon\ e^{A_\epsilon^2}}
\end{aligned}
}
\end{pmatrix}^T
= &\left( 0 , 0 \right).
\end{aligned}
\]

\subsection*{Subregion 2. $x < a + u_a t$}
	
\[
\begin{aligned}
&\lim_{\epsilon \rightarrow 0}
\begin{pmatrix}
\frac{
\begin{aligned}
u_a \left( \sqrt{\pi} - {\textnormal{erfc}} \left( P_\epsilon \right) \right) e^{P_\epsilon^2 - A_\epsilon^2} + \frac{(b-x) \left( 1 - e^{- \frac{| u_b |}{\epsilon}} \right)}{2t\ B_\epsilon\ e^{B_\epsilon^2}}
\end{aligned}
}{
\begin{aligned}
\left( \sqrt{\pi} - {\textnormal{erfc}} \left( A_\epsilon \right) \right) + {\textnormal{erfc}} \left( B_\epsilon \right) \left( e^{- \frac{| u_b |}{\epsilon}} - 1 \right) + \left( \sqrt{\pi} - {\textnormal{erfc}} \left( P_\epsilon \right) \right) e^{P_\epsilon^2 - A_\epsilon^2}
\end{aligned}
}
\\
------------------------------
\\
\frac{
\begin{aligned}
\rho_c \bigg[ \sqrt{2 t \epsilon} \left( \frac{1}{e^{A_\epsilon^2}} - \frac{1}{2\ e^{C_\epsilon^2}} \right) + (x-c-u_a t) \left( \sqrt{\pi} - {\textnormal{erfc}} \left( P_\epsilon \right) \right) e^{P_\epsilon^2 - A_\epsilon^2}
\\
+ (x-c) \left( {\textnormal{erfc}} \left( C_\epsilon \right) - {\textnormal{erfc}} \left( A_\epsilon \right) \right) \bigg] + \rho_d\ {\textnormal{erfc}} \left( D_\epsilon \right) e^{- \frac{| u_b |}{\epsilon}}
\end{aligned}
}{
\begin{aligned}
\left( \sqrt{\pi} - {\textnormal{erfc}} \left( A_\epsilon \right) \right) + {\textnormal{erfc}} \left( B_\epsilon \right) \left( e^{- \frac{| u_b |}{\epsilon}} - 1 \right) + \left( \sqrt{\pi} - {\textnormal{erfc}} \left( P_\epsilon \right) \right) e^{P_\epsilon^2 - A_\epsilon^2}
\end{aligned}
}
\end{pmatrix}^T
\\
= &\begin{cases}
\left( u_a , \rho_c \left( x-c-u_a t \right) \right),\ &{ x \in \left( c , a + \frac{u_a}{2} \cdot t \right) },
\\
\left( 0 , 0 \right),\ &{ x \in \left( a + \frac{u_a}{2} \cdot t , a + u_a t \right) }.
\end{cases}
\end{aligned}
\]

	\item $b < x < d$
	
Let us first define
\begin{enumerate}

	\item[$(i)$] $l(s) := a + \frac{u_b}{u_a} + \frac{u_a}{2} \cdot s$,
	
	\item[$(ii)$] $r(s) := a + u_a s$,
	
	\item[$(iii)$] $p(s) := b + \sqrt{2 u_b s}$,
	
	\item[$(iv)$] $q(s) := b+u_a s-\sqrt{2\ u_a \left( b-a \right) s}$

\end{enumerate}
for each $s \geq 0$. Then we can evaluate $\lim_{\epsilon \rightarrow 0} \left( u^\epsilon , R^\epsilon \right)$ separately in the subregions $x > a + u_a t$ and $x < a + u_a t$ as follows:

\subsection*{Subregion 1. $x > a + u_a t$} 

\[
\begin{aligned}
&\lim_{\epsilon \rightarrow 0} 
\begin{pmatrix}
\frac{
\begin{footnotesize}
\begin{aligned}
u_a\ \frac{\left( x-b \right) f(P_\epsilon)\ e^{B_\epsilon^2 - A_\epsilon^2}}{x-a-u_a t} + \frac{(x-b) \left( 1 - e^{- \frac{| u_b |}{\epsilon}} \right)}{2t}
\end{aligned}
\end{footnotesize}
}{
\begin{footnotesize}
\begin{aligned}
\sqrt{\pi}\ B_\epsilon\ e^{B_\epsilon^2 - \frac{u_b}{\epsilon}} - \frac{\left( x-b \right) f(A_\epsilon)\ e^{B_\epsilon^2 - A_\epsilon^2}}{x-a} 
\\
+ f(B_\epsilon) \left( 1 - e^{- \frac{| u_b |}{\epsilon}} \right) + \frac{\left( x-b \right) f(P_\epsilon)\ e^{B_\epsilon^2 - A_\epsilon^2}}{x-a-u_a t}
\end{aligned}
\end{footnotesize}
}
\\
---------------------------------
\\
\frac{
\begin{footnotesize}
\begin{aligned}
&\rho_c \left[ (x-b) \left( \frac{1}{e^{A_\epsilon^2 - B_\epsilon^2}} - \frac{1}{2\ e^{C_\epsilon^2 - B_\epsilon^2}} \right) + \frac{\left( x-c-u_a t \right) \left( x-b \right) f(P_\epsilon)}{x-a-u_a t}\ e^{B_\epsilon^2 - A_\epsilon^2} \right.
\\
+ &\left. (x-c) \left( \frac{\left( x-b \right) f(C_\epsilon)\ e^{B_\epsilon^2 - C_\epsilon^2}}{x-c} - \frac{\left( x-b \right) f(A_\epsilon)\ e^{B_\epsilon^2 - A_\epsilon^2}}{x-a} \right) \right] 
\\
+ &\rho_d\ {\textnormal{erfc}} \left( D_\epsilon \right) B_\epsilon\ e^{B_\epsilon^2 - \frac{u_b}{\epsilon}}
\end{aligned}
\end{footnotesize}
}{
\begin{footnotesize}
\begin{aligned}
&\sqrt{\pi}\ B_\epsilon\ e^{B_\epsilon^2 - \frac{u_b}{\epsilon}} - \frac{\left( x-b \right) f(A_\epsilon)\ e^{B_\epsilon^2 - A_\epsilon^2}}{x-a} 
\\
+ &f(B_\epsilon) \left( 1 - e^{- \frac{| u_b |}{\epsilon}} \right) + \frac{\left( x-b \right) f(P_\epsilon)\ e^{B_\epsilon^2 - A_\epsilon^2}}{x-a-u_a t}
\end{aligned}
\end{footnotesize}
}
\end{pmatrix}^T
\\
= &\begin{cases}
\left( 0 , 0 \right),\ &{ x \in \Big( \max{ \{ p(t) , r(t) \} }\ ,\ d \Big) },
\\
\left( \frac{x-b}{t} , 0 \right),\ &{ x \in \Big( r(t)\ ,\ p(t) \Big) }.
\end{cases}
\end{aligned}
\]	

\subsection*{Subregion 2. $x < a + u_a t$}

\[
\begin{aligned}
&\lim_{\epsilon \rightarrow 0} 
\begin{pmatrix}
\frac{
\begin{aligned}
u_a \left( \sqrt{\pi} - {\textnormal{erfc}} \left( P_\epsilon \right) \right) B_\epsilon\ e^{P_\epsilon^2 + B_\epsilon^2 - A_\epsilon^2} + \frac{(x-b) \left( 1 - e^{- \frac{| u_b |}{\epsilon}} \right)}{2t}
\end{aligned}
}{
\begin{aligned}
\sqrt{\pi}\ B_\epsilon\ e^{B_\epsilon^2 - \frac{u_b}{\epsilon}} - \frac{\left( x-b \right) f(A_\epsilon)\ e^{B_\epsilon^2 - A_\epsilon^2}}{x-a}
\\
+ f(B_\epsilon) \left( 1 - e^{- \frac{| u_b |}{\epsilon}} \right) + \left( \sqrt{\pi} - {\textnormal{erfc}} \left( P_\epsilon \right) \right) B_\epsilon\ e^{P_\epsilon^2 + B_\epsilon^2 - A_\epsilon^2}
\end{aligned}
}
\\
--------------------------------------
\\
\frac{
\begin{aligned}
\rho_c \left[ (x-b) \left( e^{B_\epsilon^2 - A_\epsilon^2} - \frac{1}{2}\ e^{B_\epsilon^2 - C_\epsilon^2} \right) + (x-c-u_a t) \left( \sqrt{\pi} - {\textnormal{erfc}} \left( P_\epsilon \right) \right) B_\epsilon\ e^{P_\epsilon^2 + B_\epsilon^2 - A_\epsilon^2} \right. 
\\
+ \left. (x-c) \left( \frac{\left( x-b \right) f(C_\epsilon)\ e^{B_\epsilon^2 - C_\epsilon^2}}{x-c} - \frac{\left( x-b \right) f(A_\epsilon)\ e^{B_\epsilon^2 - A_\epsilon^2}}{x-a} \right) \right] + \rho_d\ {\textnormal{erfc}} \left( D_\epsilon \right) B_\epsilon\ e^{B_\epsilon^2 - \frac{u_b}{\epsilon}}
\end{aligned}
}{
\begin{aligned}
\sqrt{\pi}\ B_\epsilon\ e^{B_\epsilon^2 - \frac{u_b}{\epsilon}} - \frac{\left( x-b \right) f(A_\epsilon)\ e^{B_\epsilon^2 - A_\epsilon^2}}{x-a}
\\
+ f(B_\epsilon) \left( 1 - e^{- \frac{| u_b |}{\epsilon}} \right) + \left( \sqrt{\pi} - {\textnormal{erfc}} \left( P_\epsilon \right) \right) B_\epsilon\ e^{P_\epsilon^2 + B_\epsilon^2 - A_\epsilon^2}
\end{aligned}
}
\end{pmatrix}^T
\\
= &\begin{cases}
\left( 0 , 0 \right),\ &{ x \in \Big( \max{ \{ p(t) , q(t) \} }\ ,\ r(t) \Big) \cup \Big( \max{ \{ l(t) , p(t) \big\}\ ,\ \min{ \big\{ q(t) , r(t) \} } } \Big) },
\\
\left( u_a , \rho_c \left( x-c-u_a t \right) \right),\ &{ x \in \Big( -\infty\ ,\ \min{ \{ p(t) , q(t) , r(t) \} \Big) } \cup \Big( p(t)\ ,\ \min{ \{ l(t) , q(t) , r(t) \} } \Big) },
\\
\left( \frac{x-b}{t} , 0 \right),\ &{ x \in \Big( q(t)\ ,\ \min{ \{ p(t) , r(t) \} } \Big) }.
\end{cases}
\end{aligned}
\]

	\item $x > d$
	
As in the preceding case, we need to consider the following curves defined in $\left[ 0 , \infty \right)$:

\begin{enumerate}

	\item[$(i)$] $l(s) = a + \frac{u_b}{u_a} + \frac{u_a}{2} \cdot s$,
	
	\item[$(ii)$] $r(s) = a + u_a s$,
	
	\item[$(iii)$] $p(s) = b + \sqrt{2 u_b s}$,
	
	\item[$(iv)$] $q(s) = b+u_a s-\sqrt{2\ u_a \left( b-a \right) s}$.

\end{enumerate}
We can now describe the explicit structure of $\lim_{\epsilon \rightarrow 0} \left( u^\epsilon , R^\epsilon \right)$ in different subregions as follows: 	

\subsection*{Subregion 1. $x > a + u_a t$}

\[
\begin{aligned}
&\lim_{\epsilon \rightarrow 0} \begin{pmatrix}
\frac{
\begin{aligned}
u_a\ \frac{\left( x-b \right) f(P_\epsilon)\ e^{B_\epsilon^2 - A_\epsilon^2}}{x-a-u_a t} + \frac{(x-b) \left( 1 - e^{- \frac{| u_b |}{\epsilon}} \right)}{2t}
\end{aligned}
}{
\begin{aligned}
\sqrt{\pi}\ B_\epsilon\ e^{B_\epsilon^2 - \frac{u_b}{\epsilon}} - \frac{\left( x-b \right) f(A_\epsilon)\ e^{B_\epsilon^2 - A_\epsilon^2}}{x-a} 
\\
+ f(B_\epsilon) \left( 1 - e^{- \frac{| u_b |}{\epsilon}} \right) + \frac{\left( x-b \right) f(P_\epsilon)\ e^{B_\epsilon^2 - A_\epsilon^2}}{x-a-u_a t}
\end{aligned}
}
\\
---------------------------------
\\
\frac{
\begin{aligned}
\rho_c &\left[ (x-b) \left( \frac{1}{e^{A_\epsilon^2 - B_\epsilon^2}} - \frac{1}{2\ e^{C_\epsilon^2 - B_\epsilon^2}} \right) + \frac{\left( x-c-u_a t \right) \left( x-b \right) f(P_\epsilon)}{x-a-u_a t}\ e^{B_\epsilon^2 - A_\epsilon^2} \right.
\\
&+ \left. (x-c) \left( \frac{\left( x-b \right) f(C_\epsilon)\ e^{B_\epsilon^2 - C_\epsilon^2}}{x-c} - \frac{\left( x-b \right) f(A_\epsilon)\ e^{B_\epsilon^2 - A_\epsilon^2}}{x-a} \right) \right] 
\\
+\ \rho_d &\left( \sqrt{\pi} - {\textnormal{erfc}} \left( D_\epsilon \right) \right) B_\epsilon\ e^{B_\epsilon^2 - \frac{u_b}{\epsilon}}
\end{aligned}
}{
\begin{aligned}
\sqrt{\pi}\ B_\epsilon\ e^{B_\epsilon^2 - \frac{u_b}{\epsilon}} - f(A_\epsilon)\ \frac{\left( x-b \right) e^{B_\epsilon^2 - A_\epsilon^2}}{x-a} 
\\
+ f(B_\epsilon) \left( 1 - e^{- \frac{| u_b |}{\epsilon}} \right) + f(P_\epsilon)\ \frac{\left( x-b \right) e^{B_\epsilon^2 - A_\epsilon^2}}{x-a-u_a t}
\end{aligned}
}
\end{pmatrix}^T
\\
= &\begin{cases}
\left( 0 , \rho_d \right),\ &{ x \in \Big( \max{ \{ p(t) , r(t) \} }\ ,\ \infty \Big) },
\\
\left( \frac{x-b}{t} , 0 \right),\ &{ x \in \Big( r(t)\ ,\ p(t) \Big) }.
\end{cases}
\end{aligned}
\]

\subsection*{Subregion 2. $x < a + u_a t$}

\[
\begin{aligned}
&\lim_{\epsilon \rightarrow 0} 
\begin{pmatrix}
\frac{
\begin{aligned}
u_a \left( \sqrt{\pi} - {\textnormal{erfc}} \left( P_\epsilon \right) \right) B_\epsilon\ e^{P_\epsilon^2 + B_\epsilon^2 - A_\epsilon^2} + \frac{(x-b) \left( 1 - e^{- \frac{| u_b |}{\epsilon}} \right)}{2t}
\end{aligned}
}{
\begin{aligned}
\sqrt{\pi}\ B_\epsilon\ e^{B_\epsilon^2 - \frac{u_b}{\epsilon}} - f(A_\epsilon)\ \frac{\left( x-b \right) e^{B_\epsilon^2 - A_\epsilon^2}}{x-a}
\\
+\ f(B_\epsilon) \left( 1 - e^{- \frac{| u_b |}{\epsilon}} \right) + \left( \sqrt{\pi} - {\textnormal{erfc}} \left( P_\epsilon \right) \right) B_\epsilon\ e^{P_\epsilon^2 + B_\epsilon^2 - A_\epsilon^2}
\end{aligned}
}
\\
--------------------------------------
\\
\frac{
\begin{aligned}
\rho_c &\left[ (x-b) \left( \frac{1}{e^{A_\epsilon^2 - B_\epsilon^2}} - \frac{1}{2\ e^{C_\epsilon^2 - B_\epsilon^2}} \right) + (x-c-u_a t) \left( \sqrt{\pi} - {\textnormal{erfc}} \left( P_\epsilon \right) \right) B_\epsilon\ e^{P_\epsilon^2 + B_\epsilon^2 - A_\epsilon^2} \right. 
\\
&+ \left. (x-c) \left( \frac{\left( x-b \right) f(C_\epsilon)\ e^{B_\epsilon^2 - C_\epsilon^2}}{x-c} - \frac{\left( x-b \right) f(A_\epsilon)\ e^{B_\epsilon^2 - A_\epsilon^2}}{x-a} \right) \right] 
\\
+\ \rho_d &\left( \sqrt{\pi} - {\textnormal{erfc}} \left( D_\epsilon \right) \right) B_\epsilon\ e^{B_\epsilon^2 - \frac{u_b}{\epsilon}}
\end{aligned}
}{
\begin{aligned}
\sqrt{\pi}\ B_\epsilon\ e^{B_\epsilon^2 - \frac{u_b}{\epsilon}} - f(A_\epsilon)\ \frac{\left( x-b \right) e^{B_\epsilon^2 - A_\epsilon^2}}{x-a}
\\
+ f(B_\epsilon) \left( 1 - e^{- \frac{| u_b |}{\epsilon}} \right) + \left( \sqrt{\pi} - {\textnormal{erfc}} \left( P_\epsilon \right) \right) B_\epsilon\ e^{P_\epsilon^2 + B_\epsilon^2 - A_\epsilon^2}
\end{aligned}
}
\end{pmatrix}^T
\\
= &\begin{cases}
\left( 0 , \rho_d \right),\ &{ x \in \Big( \max{ \{ p(t) , q(t) \} }\ ,\ r(t) \Big) \cup \Big( \max{ \{ l(t) , p(t) \big\}\ ,\ \min{ \big\{ q(t) , r(t) \} } } \Big) },
\\
\left( u_a , \rho_c \left( x-c-u_a t \right) \right),\ &{ x \in \Big( -\infty\ ,\ \min{ \{ p(t) , q(t) , r(t) \} \Big) } \cup \Big( p(t)\ ,\ \min{ \{ l(t) , q(t) , r(t) \} } \Big) },
\\
\left( \frac{x-b}{t} , 0 \right),\ &{ x \in \Big( q(t)\ ,\ \min{ \{ p(t) , r(t) \} } \Big) }.
\end{cases}
\end{aligned}
\]

\end{enumerate}

Now we have to recover the $\rho$ component. For this purpose, set $R = \lim_{\epsilon \rightarrow 0} R^\epsilon$. There are three cases to consider, namely

\begin{itemize}

	\item $2 u_b < u_a \left( b-a \right)$

	\item $u_b < u_a \left( b-a \right) \leq 2 u_b$

	\item $u_a \left( b-a \right) \leq u_b$

\end{itemize}
Here we provide the details only for the case $2 u_b < u_a \left( b-a \right)$. The other cases can be studied similarly.

For further simplification, let us also assume that $x_{ {}_{p,r} } < d < x_{ {}_{p,l} }$, where $x_{ {}_{p,r} }$ and $x_{ {}_{p,l} }$ denote the respective $x$ co-ordinates of the points of intersection of $x = p(t)$ with the curves $x = r(t)$ and $x = l(t)$. 

In addition to the curves $x = l(t)$, $x = r(t)$, $x = p(t)$ and $x = q(t)$ defined above, let us also introduce the following curves on $\left[ 0 , \infty \right)$:

\[
\begin{aligned}
\tilde{l} (t) &:= a + \frac{u_a}{2} \cdot t, 
\\
\gamma_{ {}_{a} } (t) &:= \begin{cases}
a + \frac{u_a}{2} \cdot t,\ &{ 0 \leq t \leq \frac{2 \left( b-a \right)}{u_a} },
\\
b + u_a t - \sqrt{2 u_a \left( b-a \right) t},\ &{ \frac{2 \left( b-a \right)}{u_a} \leq t \leq t_{ {}_{p,l} } } := \left( \frac{\sqrt{2 u_b} + \sqrt{2 u_a \left( b-a \right)}}{u_a} \right)^2,
\\
a + \frac{u_b}{u_a} + \frac{u_a}{2} \cdot t,\ &{ t \geq t_{ {}_{p,l} } },
\end{cases}
\\
\gamma_{ {}_{b} } (t) &:= \begin{cases}
b + \sqrt{2 u_b t},\ &{ 0 \leq t \leq t_{ {}_{p,l} } },
\\
a + \frac{u_b}{u_a} + \frac{u_a}{2} \cdot t,\ &{ t \geq t_{ {}_{p,l} } },
\end{cases}
\\
\gamma_{ {}_{c} } (t) &:= \begin{cases}
c,\ &{ 0 \leq t \leq \frac{2 \left( c-a \right)}{u_a} },
\\
\gamma_{ {}_{a} } (t),\ &{ t \geq \frac{2 \left( c-a \right)}{u_a} },
\end{cases}
\\
\gamma_{ {}_{d} } (t) &:= \begin{cases}
d,\ &{ 0 \leq t \leq \frac{(d-b)^2}{2 u_b} },
\\
\gamma_{ {}_{b} } (t),\ &{ t \geq \frac{(d-b)^2}{2 u_b} }.
\end{cases}
\end{aligned}
\]
Here we have used the notation $t_{ {}_{p,l} }$ to denote the positive $t$ co-ordinate for the intersection of $x = p(t)$ and $x = l(t)$. These curves can be used to describe the explicit structure of $\left( u , R \right) = \lim_{\epsilon \rightarrow 0} \left( u^\epsilon , R^\epsilon \right)$ as follows:

\[
\begin{aligned}
u(x,t) &= \begin{cases}
u_a,\ &{
\begin{aligned}
&x \in \Big( -\infty , a \Big) \cup \Big( \big( ( a , b ) \setminus \{ c \} \big) \cap \big( -\infty , \tilde{l} (t) \big) \Big) 
\\
&\cup \Big( \big( ( b , \infty ) \setminus \{ d \} \big) \cap \big( ( -\infty , \min{ \{ p(t) , q(t) , r(t) \} } ) \cup ( p(t) , \min{ \{ l(t) , q(t) , r(t) \} } ) \big)  \Big),
\end{aligned} }
\\
\frac{x-b}{t},\ &{ x \in \Big( \big( ( b , \infty ) \setminus \{ d \} \big) \cap \big( \left( r(t) , p(t) \right) \cup \left( q(t) , \min{ \{ p(t) , r(t) \} } \right) \big) \Big) },
\\
0,\ &{ 
\begin{aligned}
x \in &\bigg( \Big( ( a , b ) \setminus \{ c \} \Big) \cap \Big( \big( \tilde{l} (t) , \infty \big) \setminus \{ r(t) \} \Big) \bigg) 
\\
&{
\begin{aligned}
&\cup \Bigg( \Big( ( b , \infty ) \setminus \{ d \} \Big) \cap \bigg( \big( \max{ \{ p(t) , r(t) \} , \infty } \big) 
\\
&\cup \big( \max{ \{ p(t) , q(t) \} } , r(t) \big) \cup \big( \max{ \{ l(t) , p(t) \} } , \min{ \{ q(t) , r(t) \} } \big) \bigg) \Bigg),
\end{aligned} }
\end{aligned} }
\end{cases}
\\
R(x,t) &= \begin{cases}
\rho_c \left( x-c-u_a t \right),\ &{ x \in \Big( -\infty , \gamma_{ {}_{a} } (t) \Big) },
\\
\rho_c \left( x-c \right),\ &{ x \in \Big( \gamma_{ {}_{a} } (t) , \gamma_{ {}_{c} } (t) \Big) },
\\
0,\ &{ x \in { \Big( \gamma_{ {}_{c} } (t) , \gamma_{ {}_{b} } (t) \Big) } \cup { \Big( \gamma_{ {}_{b} } (t) , \gamma_{ {}_{d} } (t) \Big) } },
\\
\rho_d,\ &{ x > \gamma_{ {}_{d} } (t) }.
\end{cases}
\end{aligned}
\]
For any test function $\phi \in C^{\infty}_{c} \left( \textbf{R}^1 \times \left[ 0 , \infty \right) ; \textbf{R}^1 \right)$, we observe that

\[
\begin{aligned}
\left\langle R_x , \phi \right\rangle = &- \left\langle R , \phi_x \right\rangle
\\
= &- \int_{0}^{\infty} \Bigg[ \int_{-\infty}^{\gamma_{ {}_{a} } (t)} \rho_c \left( x-c-u_a t \right) \phi_x\ dx + \int_{\gamma_{ {}_{a} } (t)}^{\gamma_{ {}_{c} } (t)} \rho_c \left( x-c \right) \phi_x\ dx + \int_{\gamma_{ {}_{d} } (t)}^{\infty}\ \rho_d\ \phi_x\ dx \Bigg] dt
\\
= &- \int_{0}^{\infty} \rho_c \Bigg[ - \int_{-\infty}^{\gamma_{ {}_{a} } (t)} \phi\ dx + \left( \gamma_{ {}_{a} } (t) - c - u_a t \right) \phi \left( \gamma_{ {}_{a} } (t) , t \right) 
\\
&- \int_{\gamma_{ {}_{a} } (t)}^{\gamma_{ {}_{c} } (t)} \phi\ dx + \left( \gamma_{ {}_{c} } (t) - c \right) \phi \left( \gamma_{ {}_{c} } (t) , t \right) - \left( \gamma_{ {}_{a} } (t) - c \right) \phi \left( \gamma_{ {}_{a} } (t) , t \right) \Bigg] dt 
\\
&+ \rho_d \int_{0}^{\infty} \phi \left( \gamma_{ {}_{d} } (t) , t \right) dt
\\
= &\left\langle \rho_c \left( \chi_{ {}_{ {}_{ \left( -\infty , \gamma_{ {}_{ c } } (t) \right) } } } - \left( \left( x-c \right) \delta_{ {}_{ x = \gamma_{ {}_{c} } (t) } } - u_a t\ \delta_{ {}_{ x = \gamma_{ {}_{a} } (t) } } \right) \right) + \rho_d\ \delta_{ {}_{ x = \gamma_{ {}_{d} } (t) } } , \phi \right\rangle.
\end{aligned}
\]
Therefore $\rho = \rho_c \left( \chi_{ {}_{ {}_{ \left( -\infty , \gamma_{ {}_{ c } } (t) \right) } } } - \left( \left( x-c \right) \delta_{ {}_{ x = \gamma_{ {}_{c} } (t) } } - u_a t\ \delta_{ {}_{ x = \gamma_{ {}_{a} } (t) } } \right) \right) + \rho_d\ \delta_{ {}_{ x = \gamma_{ {}_{d} } (t) } }$.
	
\textbf{Case 3.} $u_a > 0$, $u_b < 0$

\begin{enumerate}

	\item $x < a$
	
In this region, we have $x < a + u_a t$ and hence, utilizing the inequality

\[
(a+u_a t-x)^2 + (c-x)^2 - (a-x)^2 = (a+u_a t-x)^2 + 2 \left( c-a \right) \left| \frac{a+c}{2} - x \right| > 0,
\]
the required limit $\lim_{\epsilon \rightarrow 0}\ \left( u^\epsilon , R^\epsilon \right)$ equals

\[
\begin{aligned}
&\lim_{\epsilon \rightarrow 0} 
\begin{pmatrix}
\frac{
\begin{aligned}
u_a \left( \sqrt{\pi} - {\textnormal{erfc}} \left( P_\epsilon \right) \right) + \frac{(b-x) \left( e^{- \frac{| u_b |}{\epsilon}} - 1 \right)}{2t\ B_\epsilon\ e^{P_\epsilon^2 + B_\epsilon^2 - A_\epsilon^2 + \frac{u_b}{\epsilon}}}
\end{aligned}
}{
\begin{aligned}
\frac{f(A_\epsilon)}{A_\epsilon\ e^{P_\epsilon^2}} + \frac{f(B_\epsilon) \left( 1 - e^{- \frac{| u_b |}{\epsilon}} \right)}{B_\epsilon\ e^{P_\epsilon^2 + B_\epsilon^2 - A_\epsilon^2 + \frac{u_b}{\epsilon}}} + \left( \sqrt{\pi} - {\textnormal{erfc}} \left( P_\epsilon \right) \right)
\end{aligned}
}
\\
-------------------------------
\\
\frac{
\begin{aligned}
\rho_c \left[ \sqrt{2 t \epsilon} \left( \frac{1}{e^{P_\epsilon^2}} - \frac{1}{2\ e^{P_\epsilon^2 + C_\epsilon^2 - A_\epsilon^2}} \right) + (x-c-u_a t) \left( \sqrt{\pi} - {\textnormal{erfc}} \left( P_\epsilon \right) \right) \right.
\\
\left. + (x-c) \left( \frac{f(A_\epsilon)}{A_\epsilon\ e^{P_\epsilon^2}} - \frac{f(C_\epsilon)}{C_\epsilon\ e^{P_\epsilon^2 + C_\epsilon^2 - A_\epsilon^2}} \right) \right] + \rho_d\ \frac{\left( b-x \right) f(D_\epsilon)\ e^{B_\epsilon^2 - D_\epsilon^2}}{\left( d-x \right) B_\epsilon\ e^{P_\epsilon^2 + B_\epsilon^2 - A_\epsilon^2 + \frac{u_b}{\epsilon}}}
\end{aligned}
}{
\begin{aligned}
\frac{f(A_\epsilon)}{A_\epsilon\ e^{P_\epsilon^2}} + \frac{f(B_\epsilon) \left( 1 - e^{- \frac{| u_b |}{\epsilon}} \right)}{B_\epsilon\ e^{P_\epsilon^2 + B_\epsilon^2 - A_\epsilon^2 + \frac{u_b}{\epsilon}}} + \left( \sqrt{\pi} - {\textnormal{erfc}} \left( P_\epsilon \right) \right)
\end{aligned}
}
\end{pmatrix}^T
\\
= &\begin{cases}
\left( u_a , \rho_c \left( x-c-u_a t \right) \right),\ &{ x < b+u_a t - \sqrt{2 \left( u_a \left( b-a \right) - u_b \right) t} },
\\
\left( \frac{x-b}{t} , 0 \right),\ &{ x > b+u_a t - \sqrt{2 \left( u_a \left( b-a \right) - u_b \right) t} }.
\end{cases}
\end{aligned}
\]
For discussing the passage to the limit in the remaining regions, we introduce the curves
\begin{enumerate}

	\item[$(i)$] $\tilde{l} (s) := a + \frac{u_a}{2} \cdot s$,
	
	\item[$(ii)$] $r(s) := a + u_a s$,
	
	\item[$(iii)$] $p(s) := b - \sqrt{- 2 u_b s}$,
	
	\item[$(iv)$] $q(s) := b+u_a s-\sqrt{2 \left( u_a \left( b-a \right) - u_b \right) s}$

\end{enumerate}
defined for each $s \geq 0$. The limit $\lim_{\epsilon \rightarrow 0}\ \left( u^\epsilon , R^\epsilon \right)$ can then be evaluated as follows:

	\item $a < x < c$

\subsection*{Subregion 1. $x > a + u_a t$}

\[
\begin{aligned}
&\begin{pmatrix}
\frac{
\begin{aligned}
u_a\ \frac{f(P_\epsilon)}{P_\epsilon\ e^{A_\epsilon^2}} + \frac{(b-x) \left( e^{- \frac{| u_b |}{\epsilon}} - 1 \right)}{2t\ B_\epsilon\ e^{B_\epsilon^2 + \frac{u_b}{\epsilon}}}
\end{aligned}
}{
\begin{aligned}
\sqrt{\pi} - {\textnormal{erfc}} \left( A_\epsilon \right) + \frac{f(B_\epsilon) \left( 1 - e^{- \frac{| u_b |}{\epsilon}} \right)}{B_\epsilon\ e^{B_\epsilon^2 + \frac{u_b}{\epsilon}}} + \frac{f(P_\epsilon)}{P_\epsilon\ e^{A_\epsilon^2}}
\end{aligned}
}
\\
----------------------------
\\
\frac{
\begin{aligned}
\rho_c \bigg[ \sqrt{2 t \epsilon} \left( \frac{1}{e^{A_\epsilon^2}} - \frac{1}{2\ e^{C_\epsilon^2}} \right) + \left( x-c-u_a t \right) \frac{f(P_\epsilon)}{P_\epsilon\ e^{A_\epsilon^2}}
\\
+ \left( x-c \right) \left( \sqrt{\pi} - {\textnormal{erfc}} \left( A_\epsilon \right) - {\textnormal{erfc}} \left( C_\epsilon \right) \right) \bigg] + \rho_d\ \frac{\left( b-x \right) e^{B_\epsilon^2 - D_\epsilon^2}}{\left( d-x \right) B_\epsilon\ e^{B_\epsilon^2 + \frac{u_b}{\epsilon}}}
\end{aligned}
}{
\begin{aligned}
\sqrt{\pi} - {\textnormal{erfc}} \left( A_\epsilon \right) + \frac{f(B_\epsilon) \left( 1 - e^{- \frac{| u_b |}{\epsilon}} \right)}{B_\epsilon\ e^{B_\epsilon^2 + \frac{u_b}{\epsilon}}} + \frac{f(P_\epsilon)}{P_\epsilon\ e^{A_\epsilon^2}}
\end{aligned}
}
\end{pmatrix}^T
\\
= &\begin{cases}
\left( 0 , \rho_c \left( x-c \right) \right),\ &{ x \in { \Big( r(t) , p(t) \Big) } },
\\
\left( \frac{x-b}{t} , 0 \right),\ &{ x \in { \Big( \max{ \big\{ r(t) , p(t) \big\} } , \infty \Big) } }.
\end{cases}
\end{aligned}
\]

\subsection*{Subregion 2. $x < a + u_a t$}

\[
\begin{aligned}
&\lim_{\epsilon \rightarrow 0}
\begin{pmatrix}
\frac{
\begin{aligned}
u_a \left( \sqrt{\pi} - {\textnormal{erfc}} \left( P_\epsilon \right) \right) B_\epsilon\ e^{P_\epsilon^2 + B_\epsilon^2 - A_\epsilon^2 + \frac{u_b}{\epsilon}} + \frac{b-x}{2t} \left( e^{- \frac{| u_b |}{\epsilon}} - 1 \right)
\end{aligned}
}{
\begin{aligned}
\left( \sqrt{\pi} - {\textnormal{erfc}} \left( A_\epsilon \right) \right) B_\epsilon\ e^{B_\epsilon^2 + \frac{u_b}{\epsilon}} + f(B_\epsilon) \left( 1 - e^{- \frac{| u_b |}{\epsilon}} \right) 
\\
+ \left( \sqrt{\pi} - {\textnormal{erfc}} \left( P_\epsilon \right) \right) B_\epsilon\ e^{P_\epsilon^2 + B_\epsilon^2 - A_\epsilon^2 + \frac{u_b}{\epsilon}}
\end{aligned}
}
\\
----------------------------------
\\
\frac{
\begin{aligned}
\rho_c \bigg[ \sqrt{2 t \epsilon} \left( \frac{1}{e^{A_\epsilon^2}} - \frac{1}{2\ e^{C_\epsilon^2}} \right) + (x-c-u_a t) \left( \sqrt{\pi} - {\textnormal{erfc}} \left( P_\epsilon \right) \right) B_\epsilon\ e^{P_\epsilon^2 + B_\epsilon^2 - A_\epsilon^2 + \frac{u_b}{\epsilon}}
\\
+ \left( x-c \right) \left( \sqrt{\pi} - {\textnormal{erfc}} \left( A_\epsilon \right) - {\textnormal{erfc}} \left( C_\epsilon \right) \right) B_\epsilon\ e^{B_\epsilon^2 + \frac{u_b}{\epsilon}} \bigg] + \rho_d\ \frac{\left( b-x \right) f(D_\epsilon)}{d-x}\ e^{B_\epsilon^2 - D_\epsilon^2} 
\end{aligned}
}{
\begin{aligned}
\left( \sqrt{\pi} - {\textnormal{erfc}} \left( A_\epsilon \right) \right) B_\epsilon\ e^{B_\epsilon^2 + \frac{u_b}{\epsilon}} + f(B_\epsilon) \left( 1 - e^{- \frac{| u_b |}{\epsilon}} \right) 
\\
+ \left( \sqrt{\pi} - {\textnormal{erfc}} \left( P_\epsilon \right) \right) B_\epsilon\ e^{P_\epsilon^2 + B_\epsilon^2 - A_\epsilon^2 + \frac{u_b}{\epsilon}}
\end{aligned}
}
\end{pmatrix}^T
\\
= &\begin{cases}
\left( 0 , \rho_c \left( x-c \right) \right),\ &{ x \in { { \Big( -\infty , \min{ \big\{ r(t) , p(t) \big\} } \Big) } \cap { \Big( \big( \tilde{l} (t) , q(t) \big) \cup \big( q(t) , \infty \big) \Big) } } },
\\
\left( \frac{x-b}{t} , 0 \right),\ &{ x \in { \Big( \max{ \big\{ p(t) , q(t) \big\} } , r(t) \Big) } },
\\
\left( u_a , \rho_c \left( x-c-u_a t \right) \right),\ &{ x \in { \Big( -\infty , \min{ \big\{ r(t) , q(t) \big\} } \Big) } \cap { \Big( \big( -\infty , \min{ \big\{ \tilde{l} (t) , p(t) \big\} } \big) \cup \big( p(t) , \infty \big) \Big) } }.
\end{cases}
\end{aligned}
\]

	\item $c < x < b$


\subsection*{Subregion 1. $x > a + u_a t$}

\[
\begin{aligned}
&\lim_{\epsilon \rightarrow 0} 
\begin{pmatrix}
\frac{
\begin{aligned}
u_a\ \frac{f(P_\epsilon)}{P_\epsilon\ e^{A_\epsilon^2}} + \frac{(b-x) \left( e^{- \frac{| u_b |}{\epsilon}} - 1 \right)}{2t\ B_\epsilon\ e^{B_\epsilon^2 + \frac{u_b}{\epsilon}}}
\end{aligned}
}{
\begin{aligned}
\sqrt{\pi} - {\textnormal{erfc}} \left( A_\epsilon \right) + \frac{f(B_\epsilon) \left( 1 - e^{- \frac{| u_b |}{\epsilon}} \right)}{B_\epsilon\ e^{B_\epsilon^2 + \frac{u_b}{\epsilon}}} + \frac{f(P_\epsilon)}{P_\epsilon\ e^{A_\epsilon^2}}
\end{aligned}
}
\\
-------------------------
\\
\frac{
\begin{aligned}
\rho_c \Bigg[ \sqrt{2 t \epsilon} \left( \frac{1}{e^{A_\epsilon^2}} - \frac{1}{2\ e^{C_\epsilon^2}} \right) + \left( x-c-u_a t \right) \frac{f(P_\epsilon)}{P_\epsilon\ e^{A_\epsilon^2}}
\\
+ (x-c) \left( {\textnormal{erfc}} \left( C_\epsilon \right) - {\textnormal{erfc}} \left( A_\epsilon \right) \right) \Bigg] + \rho_d\ \frac{\left( b-x \right) e^{B_\epsilon^2 - D_\epsilon^2}}{\left( d-x \right) B_\epsilon\ e^{B_\epsilon^2 + \frac{u_b}{\epsilon}}}
\end{aligned}
}{
\begin{aligned}
\sqrt{\pi} - {\textnormal{erfc}} \left( A_\epsilon \right) + \frac{f(B_\epsilon) \left( 1 - e^{- \frac{| u_b |}{\epsilon}} \right)}{B_\epsilon\ e^{B_\epsilon^2 + \frac{u_b}{\epsilon}}} + \frac{f(P_\epsilon)}{P_\epsilon\ e^{A_\epsilon^2}}
\end{aligned}
}
\end{pmatrix}^T
\\
= &\begin{cases}
\left( 0 , 0 \right),\ &{ x \in { \Big( r(t) , p(t) \Big) } },
\\
\left( \frac{x-b}{t} , 0 \right),\ &{ x \in { \Big( \max{ \big\{ r(t) , p(t) \big\} } , \infty \Big) } }.
\end{cases}
\end{aligned}
\]

\subsection*{Subregion 2. $x < a + u_a t$}

\[
\begin{aligned}
&\lim_{\epsilon \rightarrow 0}
\begin{pmatrix}
\frac{
\begin{aligned}
u_a \left( \sqrt{\pi} - {\textnormal{erfc}} \left( P_\epsilon \right) \right) B_\epsilon\ e^{P_\epsilon^2 + B_\epsilon^2 - A_\epsilon^2 + \frac{u_b}{\epsilon}} + \frac{b-x}{2t} \left( e^{- \frac{| u_b |}{\epsilon}} - 1 \right)
\end{aligned}
}{
\begin{aligned}
\left( \sqrt{\pi} - {\textnormal{erfc}} \left( A_\epsilon \right) \right) B_\epsilon\ e^{B_\epsilon^2 + \frac{u_b}{\epsilon}} + f(B_\epsilon) \left( 1 - e^{- \frac{| u_b |}{\epsilon}} \right) 
\\
+ \left( \sqrt{\pi} - {\textnormal{erfc}} \left( P_\epsilon \right) \right) B_\epsilon\ e^{P_\epsilon^2 + B_\epsilon^2 - A_\epsilon^2 + \frac{u_b}{\epsilon}}
\end{aligned}
}
\\
----------------------------------
\\
\frac{
\begin{aligned}
\rho_c \Bigg[ \sqrt{2 t \epsilon} \left( \frac{1}{e^{A_\epsilon^2}} - \frac{1}{2\ e^{C_\epsilon^2}} \right) + (x-c-u_a t) \left( \sqrt{\pi} - {\textnormal{erfc}} \left( P_\epsilon \right) \right) B_\epsilon\ e^{P_\epsilon^2 + B_\epsilon^2 - A_\epsilon^2 + \frac{u_b}{\epsilon}}
\\
+ (x-c) \left( {\textnormal{erfc}} \left( C_\epsilon \right) - {\textnormal{erfc}} \left( A_\epsilon \right) \right) B_\epsilon\ e^{B_\epsilon^2 + \frac{u_b}{\epsilon}} \Bigg] + \rho_d\ \frac{\left( b-x \right) f(D_\epsilon)}{d-x}\ e^{B_\epsilon^2 - D_\epsilon^2} 
\end{aligned}
}{
\begin{aligned}
\left( \sqrt{\pi} - {\textnormal{erfc}} \left( A_\epsilon \right) \right) B_\epsilon\ e^{B_\epsilon^2 + \frac{u_b}{\epsilon}} + f(B_\epsilon) \left( 1 - e^{- \frac{| u_b |}{\epsilon}} \right) 
\\
+ \left( \sqrt{\pi} - {\textnormal{erfc}} \left( P_\epsilon \right) \right) B_\epsilon\ e^{P_\epsilon^2 + B_\epsilon^2 - A_\epsilon^2 + \frac{u_b}{\epsilon}}
\end{aligned}
}
\end{pmatrix}^T
\\
= &\begin{cases}
\left( 0 , 0 \right),\ &{ x \in { \Big( -\infty , \min{ \big\{ r(t) , p(t) \big\} } \Big) } \cap { \Big( \big( \tilde{l} (t) , q(t) \big) \cup \big( q(t) , \infty \big) \Big) } },
\\
\left( \frac{x-b}{t} , 0 \right),\ &{ x \in { \Big( \max{ \big\{ p(t) , q(t) \big\} } , r(t) \Big) } },
\\
\left( u_a , \rho_c \left( x-c-u_a t \right) \right),\ &{ x \in { \Big( -\infty , \min{ \big\{ r(t) , q(t) \big\} } \Big) } \cap { \Big( \big( -\infty , \min{ \big\{ \tilde{l} (t) , p(t) \big\} } \big) \cup \big( p(t) , \infty \big) \Big) } }.
\end{cases}
\end{aligned}
\]	
Before considering the remaining two regions, we introduce the curve $l : s \longmapsto a + \frac{u_b}{u_a} + \frac{u_a}{2} \cdot s$ defined over $\left[ 0 , \infty \right)$. Then we are able to study the behaviour of $\lim_{\epsilon \rightarrow 0} \left( u^\epsilon , R^\epsilon \right)$ in these regions as follows:

	\item $b < x < d$

\subsection*{Subregion 1.  $x > a + u_a t$}

\[
\begin{aligned}
&\lim_{\epsilon \rightarrow 0} 
\begin{pmatrix}
\frac{
\begin{aligned}
u_a\ \frac{f(P_\epsilon)}{P_\epsilon\ e^{A_\epsilon^2}}\ e^{- \frac{| u_b |}{\epsilon}} + \frac{(x-b) \left( e^{- \frac{| u_b |}{\epsilon}} - 1 \right)}{2t\ B_\epsilon\ e^{B_\epsilon^2}}
\end{aligned}
}{
\begin{aligned}
\sqrt{\pi} - {\textnormal{erfc}} \left( A_\epsilon \right) e^{- \frac{| u_b |}{\epsilon}} + {\textnormal{erfc}} \left( B_\epsilon \right) \left( e^{- \frac{| u_b |}{\epsilon}} - 1 \right) + \frac{f(P_\epsilon)}{P_\epsilon\ e^{A_\epsilon^2}}\ e^{- \frac{| u_b |}{\epsilon}}
\end{aligned}
}
\\
----------------------------
\\
\frac{
\begin{aligned}
\rho_c \bigg[ \sqrt{2 t \epsilon} \left( \frac{1}{e^{A_\epsilon^2}} - \frac{1}{2\ e^{C_\epsilon^2}} \right) + \left( x-c-u_a t \right) \frac{f(P_\epsilon)}{P_\epsilon\ e^{A_\epsilon^2}}\ e^{- \frac{| u_b |}{\epsilon}}
\\
+ \left( x-c \right) \left( {\textnormal{erfc}} \left( C_\epsilon \right) - {\textnormal{erfc}} \left( A_\epsilon \right) \right) e^{- \frac{| u_b |}{\epsilon}} \bigg] + \rho_d\ {\textnormal{erfc}} \left( D_\epsilon \right)
\end{aligned}
}{
\begin{aligned}
\sqrt{\pi} - {\textnormal{erfc}} \left( A_\epsilon \right) e^{- \frac{| u_b |}{\epsilon}} + {\textnormal{erfc}} \left( B_\epsilon \right) \left( e^{- \frac{| u_b |}{\epsilon}} - 1 \right) + \frac{f(P_\epsilon)}{P_\epsilon\ e^{A_\epsilon^2}}\ e^{- \frac{| u_b |}{\epsilon}}
\end{aligned}
}
\end{pmatrix}^T
= &\left( 0 , 0 \right).
\end{aligned}
\]

\subsection*{Subregion 2. $x < a + u_a t$}

\[
\begin{aligned}
&\lim_{\epsilon \rightarrow 0} 
\begin{pmatrix}
\frac{
\begin{aligned}
u_a \left( \sqrt{\pi} - {\textnormal{erfc}} \left( P_\epsilon \right) \right) e^{P_\epsilon^2 - A_\epsilon^2 + \frac{u_b}{\epsilon}} + \frac{(x-b) \left( e^{- \frac{| u_b |}{\epsilon}} - 1 \right)}{2t\ B_\epsilon\ e^{B_\epsilon^2}}
\end{aligned}
}{
\begin{aligned}
\sqrt{\pi} - {\textnormal{erfc}} \left( A_\epsilon \right) e^{- \frac{| u_b |}{\epsilon}} + {\textnormal{erfc}} \left( B_\epsilon \right) \left( e^{- \frac{| u_b |}{\epsilon}} - 1 \right) + \left( \sqrt{\pi} - {\textnormal{erfc}} \left( P_\epsilon \right) \right) e^{P_\epsilon^2 - A_\epsilon^2 + \frac{u_b}{\epsilon}}
\end{aligned}
}
\\
----------------------------------
\\
\frac{
\begin{aligned}
\rho_c \bigg[ \sqrt{2 t \epsilon} \left( \frac{1}{e^{A_\epsilon^2}} - \frac{1}{2\ e^{C_\epsilon^2}} \right) e^{- \frac{| u_b |}{\epsilon}} + (x-c-u_a t) \left( \sqrt{\pi} - {\textnormal{erfc}} \left( P_\epsilon \right) \right) e^{P_\epsilon^2 - A_\epsilon^2 + \frac{u_b}{\epsilon}}
\\
+ \left( x-c \right) \left( {\textnormal{erfc}} \left( C_\epsilon \right) - {\textnormal{erfc}} \left( A_\epsilon \right) \right) e^{- \frac{| u_b |}{\epsilon}} \bigg] + \rho_d\ {\textnormal{erfc}} \left( D_\epsilon \right)
\end{aligned}
}{
\begin{aligned}
\sqrt{\pi} - {\textnormal{erfc}} \left( A_\epsilon \right) e^{- \frac{| u_b |}{\epsilon}} + {\textnormal{erfc}} \left( B_\epsilon \right) \left( e^{- \frac{| u_b |}{\epsilon}} - 1 \right) + \left( \sqrt{\pi} - {\textnormal{erfc}} \left( P_\epsilon \right) \right) e^{P_\epsilon^2 - A_\epsilon^2 + \frac{u_b}{\epsilon}}
\end{aligned}
}
\end{pmatrix}^T
\\
= &\begin{cases}
\left( u_a , \rho_c \left( x-c-u_a t \right) \right),\ &{ x \in \Big( -\infty , \min{ \big\{ r(t) , l(t) \big\} } \Big) },
\\
\left( 0 , 0 \right),\ &{ x \in { { \Big( l(t) , r(t) \Big) } } }.
\end{cases}
\end{aligned}
\]

	\item $x > d$

\subsection*{Subregion 1. $x > a + u_a t$}

\[
\begin{aligned}
&\lim_{\epsilon \rightarrow 0} \begin{pmatrix}
\frac{
\begin{aligned}
u_a\ \frac{f(P_\epsilon)}{P_\epsilon\ e^{A_\epsilon^2}}\ e^{- \frac{| u_b |}{\epsilon}} + \frac{(x-b) \left( e^{- \frac{| u_b |}{\epsilon}} - 1 \right)}{2t\ B_\epsilon\ e^{B_\epsilon^2}}
\end{aligned}
}{
\begin{aligned}
\sqrt{\pi} - {\textnormal{erfc}} \left( A_\epsilon \right) e^{- \frac{| u_b |}{\epsilon}} + {\textnormal{erfc}} \left( B_\epsilon \right) \left( e^{- \frac{| u_b |}{\epsilon}} - 1 \right) + \frac{f(P_\epsilon)}{P_\epsilon\ e^{A_\epsilon^2}}\ e^{- \frac{| u_b |}{\epsilon}}
\end{aligned}
}
\\
----------------------------
\\
\frac{
\begin{aligned}
\rho_c \bigg[ \sqrt{2 t \epsilon} \left( \frac{1}{e^{A_\epsilon^2}} - \frac{1}{2\ e^{C_\epsilon^2}} \right) + \left( x-c-u_a t \right) \frac{f(P_\epsilon)}{P_\epsilon\ e^{A_\epsilon^2}}\ e^{- \frac{| u_b |}{\epsilon}}
\\
+ \left( x-c \right) \left( {\textnormal{erfc}} \left( C_\epsilon \right) - {\textnormal{erfc}} \left( A_\epsilon \right) \right) e^{- \frac{| u_b |}{\epsilon}} \bigg] + \rho_d \left( \sqrt{\pi} - {\textnormal{erfc}} \left( D_\epsilon \right) \right)
\end{aligned}
}{
\begin{aligned}
\sqrt{\pi} - {\textnormal{erfc}} \left( A_\epsilon \right) e^{- \frac{| u_b |}{\epsilon}} + {\textnormal{erfc}} \left( B_\epsilon \right) \left( e^{- \frac{| u_b |}{\epsilon}} - 1 \right) + \frac{f(P_\epsilon)}{P_\epsilon\ e^{A_\epsilon^2}}\ e^{- \frac{| u_b |}{\epsilon}}
\end{aligned}
}
\end{pmatrix}^T
= &\left( 0 , \rho_d \right).
\end{aligned}
\]

\subsection*{Subregion 2. $x < a + u_a t$}

\[
\begin{aligned}
&\lim_{\epsilon \rightarrow 0} 
\begin{pmatrix}
\frac{
\begin{aligned}
u_a \left( \sqrt{\pi} - {\textnormal{erfc}} \left( P_\epsilon \right) \right) e^{P_\epsilon^2 - A_\epsilon^2 + \frac{u_b}{\epsilon}} + \frac{(x-b) \left( e^{- \frac{| u_b |}{\epsilon}} - 1 \right)}{2t\ B_\epsilon\ e^{B_\epsilon^2}}
\end{aligned}
}{
\begin{aligned}
\sqrt{\pi} - {\textnormal{erfc}} \left( A_\epsilon \right) e^{- \frac{| u_b |}{\epsilon}} + {\textnormal{erfc}} \left( B_\epsilon \right) \left( e^{- \frac{| u_b |}{\epsilon}} - 1 \right) + \left( \sqrt{\pi} - {\textnormal{erfc}} \left( P_\epsilon \right) \right) e^{P_\epsilon^2 - A_\epsilon^2 + \frac{u_b}{\epsilon}}
\end{aligned}
}
\\
----------------------------------
\\
\frac{
\begin{aligned}
\rho_c \bigg[ \sqrt{2 t \epsilon} \left( \frac{1}{e^{A_\epsilon^2}} - \frac{1}{2\ e^{C_\epsilon^2}} \right) e^{- \frac{| u_b |}{\epsilon}} + (x-c-u_a t) \left( \sqrt{\pi} - {\textnormal{erfc}} \left( P_\epsilon \right) \right) e^{P_\epsilon^2 - A_\epsilon^2 + \frac{u_b}{\epsilon}}
\\
+ \left( x-c \right) \left( {\textnormal{erfc}} \left( C_\epsilon \right) - {\textnormal{erfc}} \left( A_\epsilon \right) \right) e^{- \frac{| u_b |}{\epsilon}} \bigg] + \rho_d \left( \sqrt{\pi} - {\textnormal{erfc}} \left( D_\epsilon \right) \right)
\end{aligned}
}{
\begin{aligned}
\sqrt{\pi} - {\textnormal{erfc}} \left( A_\epsilon \right) e^{- \frac{| u_b |}{\epsilon}} + {\textnormal{erfc}} \left( B_\epsilon \right) \left( e^{- \frac{| u_b |}{\epsilon}} - 1 \right) + \left( \sqrt{\pi} - {\textnormal{erfc}} \left( P_\epsilon \right) \right) e^{P_\epsilon^2 - A_\epsilon^2 + \frac{u_b}{\epsilon}}
\end{aligned}
}
\end{pmatrix}^T
\\
= &\begin{cases}
\left( u_a , \rho_c \left( x-c-u_a t \right) \right),\ &{ x \in \Big( -\infty , \min{ \big\{ r(t) , l(t) \big\} } \Big) },
\\
\left( 0 , \rho_d \right),\ &{ x \in { { \Big( l(t) , r(t) \Big) } } }.
\end{cases}
\end{aligned}
\]

\end{enumerate}

To recover $\rho$, set $R = \lim_{\epsilon \rightarrow 0} R^\epsilon$. There are two cases to consider, namely

\begin{itemize}

	\item $| u_b | < u_a \left( b-a \right)$

	\item $| u_b | \geq u_a \left( b-a \right)$

\end{itemize}
We provide the computations only for the case $| u_b | < u_a \left( b-a \right)$. For further simplification, let us restrict ourselves to the case $c > x_{ {}_{p,r} }$, where $x_{ {}_{p,r} } = a + u_a \cdot t_{ {}_{p,r} }$ and $\left( x_{ {}_{p,r} } , t_{ {}_{p,r} } \right)$ denotes the point of intersection of $x=p(t)$ with $x=r(t)$ in the upper-half plane. We will also be using other $x$ and $t$ subscripts in reference to the curves defined under this case along with the curves $x=c$ and $x=d$. For example, $\left( x_{ {}_{p,c} } , t_{ {}_{p,c} } \right)$ will denote the point of intersection of $x = p(t)$ and $x=c$ in the upper-half plane.

The restriction imposed above will ensure that $t_{ {}_{p,c} } < t_{ {}_{p,q} }$. In addition to the curves $x = l(t)$, $x = \tilde{l} (t)$, $x = p(t)$ and $x = q(t)$, let us now introduce the curves

\[
\begin{aligned}
\gamma_{ {}_{a} } (t) &:= \begin{cases}
a + \frac{u_a}{2} \cdot t,\ &{ 0 \leq t \leq t_{ {}_{q,\tilde{l}} } := \left( \frac{\sqrt{2 \left( 2 u_a \left( b-a \right) - u_b \right)} - \sqrt{- 2 u_b}}{2 u_a} \right)^2 },
\\
b + u_a t - \sqrt{2 \left( u_a \left( b-a \right) - u_b \right) t},\ &{ t_{ {}_{q,\tilde{l}} } \leq t \leq t_{ {}_{q,l} } := \frac{2 \left( u_a \left( b-a \right) - u_b \right)}{u_a^2} },
\\
a + \frac{u_b}{u_a} + \frac{u_a}{2} \cdot t,\ &{ t \geq t_{ {}_{q,l} } },
\end{cases}
\\
\gamma_{ {}_{b} } (t) &:= \begin{cases}
b - \sqrt{- 2 u_b t},\ &{ 0 \leq t \leq t_{ {}_{p,q} } := \left( \frac{\sqrt{2 \left( u_a \left( b-a \right) - u_b \right)} - \sqrt{- 2 u_b}}{u_a} \right)^2 },
\\
\gamma_{ {}_{a} } (t),\ &{ t \geq t_{ {}_{p,q} } },
\end{cases}
\\
\gamma_{ {}_{c} } (t) &:= \begin{cases}
c,\ &{ 0 \leq t \leq t_{ {}_{p,c} } := \frac{\left( b-c \right)^2}{2 | u_b |} },
\\
\gamma_{ {}_{b} } (t),\ &{ t \geq t_{ {}_{p,c} } },
\end{cases}
\\
\gamma_{ {}_{d} } (t) &:= \begin{cases}
d,\ &{ 0 \leq t \leq t_{ {}_{l,d} } := \frac{2 \left( d-a + \frac{| u_b |}{u_a} \right)}{u_a} },
\\
l(t),\ &{ t \geq t_{ {}_{l,d} } }.
\end{cases}
\end{aligned}
\]
Using these newly introduced curves, we may describe $\left( u , R \right) = \lim_{\epsilon \rightarrow 0} \left( u^\epsilon , R^\epsilon \right)$ as follows:
\[
\begin{aligned}
u(x,t) &= \begin{cases}
u_a,\ &{ 
\begin{aligned}
x \in &\Big( \big( -\infty , a \big) \cap \big( -\infty , q(t) \big) \Big) 
\\
&\cup \bigg( \Big( ( a , b ) \setminus \{ c \} \Big) \cap \Big( -\infty , \min{ \{ q(t) , r(t) \} } \Big) \cap \Big( \big( -\infty , \min{ \{ \tilde{l} (t) , p(t) \} } \big) \cup \big( p(t) , \infty \big) \Big) \bigg)
\\
&\cup \bigg( \Big( ( b , \infty ) \setminus \{ d \} \Big) \cap \Big( -\infty , \min{ \{ l(t) , r(t) \} } \Big) \bigg),
\end{aligned} }
\\
\frac{x-b}{t},\ &{ 
\begin{aligned}
x \in &\Big( \big( -\infty , a \big) \cap \big( q(t) , \infty \big) \Big) 
\\
&\cup \bigg( \Big( ( a , b ) \setminus \{ c \} \Big) \cap \Big( \big( \max{ \{ p(t) , q(t) \} } , r(t) \big) \cup \big( \max{ \{ p(t) , r(t) \} } , \infty \big) \Big) \bigg),
\end{aligned} }
\\
0,\ &{ 
\begin{aligned}
x \in &
\Bigg( 
\Big( ( a , b ) \setminus \{ c \} \Big) 
\cap 
\bigg( 
\Big( r(t) , p(t) \Big) \cup \bigg( \Big( -\infty , \min{ \big\{ p(t) , r(t) \big\} } \Big) \cap \Big( \big( l(t) , q(t) \big) \cup \big( q(t) , \infty \big) \Big) \bigg)
\bigg) 
\Bigg)  
\\
&\cup \bigg( \Big( ( b , \infty ) \setminus \{ d \} \Big) \cap \Big( ( l(t) , r(t) ) \cup ( r(t) , \infty ) \Big) \bigg),
\end{aligned} }
\end{cases}
\\
R(x,t) &= \begin{cases}
\rho_c \left( x-c-u_a t \right),\ &{ x \in \Big( -\infty , \gamma_{ {}_{a} } (t) \Big) },
\\
\rho_c \left( x-c \right),\ &{ x \in \Big( \gamma_{ {}_{a} } (t) , \gamma_{ {}_{c} } (t) \Big) },
\\
0,\ &{ x \in { \Big( \gamma_{ {}_{c} } (t) , \gamma_{ {}_{b} } (t) \Big) } \cup { \Big( \gamma_{ {}_{b} } (t) , \gamma_{ {}_{d} } (t) \Big) } },
\\
\rho_d,\ &{ x > \gamma_{ {}_{d} } (t) }.
\end{cases}
\end{aligned}
\]
We conclude our discussion for this case with the observation that any $\phi \in C^{\infty}_{c} \left( \textbf{R}^1 \times \left[ 0 , \infty \right) ; \textbf{R}^1 \right)$ satisfies

\begin{small}
\[
\begin{aligned}
\left\langle R_x , \phi \right\rangle = &- \left\langle R , \phi_x \right\rangle
\\
= &- \int_{0}^{\infty} \Bigg[ \int_{-\infty}^{\gamma_{ {}_{a} } (t)} \rho_c \left( x-c-u_a t \right) \phi_x\ dx + \int_{\gamma_{ {}_{a} } (t)}^{\gamma_{ {}_{c} } (t)} \rho_c \left( x-c \right) \phi_x\ dx + \int_{\gamma_{ {}_{d} } (t)}^{\infty}\ \rho_d\ \phi_x\ dx \Bigg] dt
\\
= &- \int_{0}^{\infty} \rho_c \Bigg[ - \int_{-\infty}^{\gamma_{ {}_{a} } (t)} \phi\ dx + \left( \gamma_{ {}_{a} } (t) - c - u_a t \right) \phi \left( \gamma_{ {}_{a} } (t) , t \right) 
\\
&- \int_{\gamma_{ {}_{a} } (t)}^{\gamma_{ {}_{c} } (t)} \phi\ dx + \left( \gamma_{ {}_{c} } (t) - c \right) \phi \left( \gamma_{ {}_{c} } (t) , t \right) - \left( \gamma_{ {}_{a} } (t) - c \right) \phi \left( \gamma_{ {}_{a} } (t) , t \right) \Bigg] dt 
\\
&+ \rho_d \int_{0}^{\infty} \phi \left( \gamma_{ {}_{d} } (t) , t \right) dt
\\
= &\left\langle \rho_c \left( \chi_{ {}_{ {}_{ \left( -\infty , \gamma_{ {}_{ c } } (t) \right) } } } - \left( \left( x-c \right) \delta_{ {}_{ x = \gamma_{ {}_{c} } (t) } } - u_a t\ \delta_{ {}_{ x = \gamma_{ {}_{a} } (t) } } \right) \right) + \rho_d\ \delta_{ {}_{ x = \gamma_{ {}_{d} } (t) } } , \phi \right\rangle,
\end{aligned}
\]
\end{small}
which implies that the density component $\rho$ is given by

\[
\rho = \rho_c \left( \chi_{ {}_{ {}_{ \left( -\infty , \gamma_{ {}_{ c } } (t) \right) } } } - \left( \left( x-c \right) \delta_{ {}_{ x = \gamma_{ {}_{c} } (t) } } - u_a t\ \delta_{ {}_{ x = \gamma_{ {}_{a} } (t) } } \right) \right) + \rho_d\ \delta_{ {}_{ x = \gamma_{ {}_{d} } (t) } }.
\]
	
\textbf{Case 4.} $u_a < 0$, $u_b < 0$

\begin{enumerate}

	\item $x < a$
	
In this region, let us first introduce the curves
\begin{enumerate}

	\item[$(i)$] $l(s) := \frac{a+b}{2} + \frac{u_b}{b-a} \cdot s$,
	
	\item[$(ii)$] $r(s) := a + u_a s$,
	
	\item[$(iii)$] $p(s) := b - \sqrt{- 2 u_b s}$,
	
	\item[$(iv)$] $q(s) := b+u_a s-\sqrt{2 \left( u_a \left( b-a \right) - u_b \right) s}$ (defined if $u_a \left( b-a \right) > u_b$)

\end{enumerate}
defined for each $s \geq 0$. The subsequent evaluations of the limit $\lim_{\epsilon \rightarrow 0}\ \left( u^\epsilon , R^\epsilon \right)$ in the subregions $x > a + u_a t$ and $x < a + u_a t$ are shown separately as follows:	
	
\subsection*{Subregion 1. $x > a + u_a t$}

\begin{footnotesize}
\[
\begin{aligned}
&\lim_{\epsilon \rightarrow 0} 
\begin{pmatrix}
\frac{
\begin{aligned}
f(P_\epsilon) \cdot u_a + \frac{(x-a-u_a t) \left( e^{- \frac{| u_b |}{\epsilon}} - 1 \right)}{2t\ e^{B_\epsilon^2 - A_\epsilon^2 + \frac{u_b}{\epsilon}}}
\end{aligned}
}{
\begin{aligned}
\frac{x-a - u_a t}{a-x}\ f(A_\epsilon) + \frac{\left( x-a-u_a t \right) \left( 1 - e^{- \frac{| u_b |}{\epsilon}} \right)}{\left( b-x \right) e^{B_\epsilon^2 - A_\epsilon^2 + \frac{u_b}{\epsilon}}}\ f(B_\epsilon) + f(P_\epsilon)
\end{aligned}
}
\\
----------------------------
\\
\frac{
\begin{aligned}
\rho_c \bigg[ (x-a-u_a t) \left( 1 - \frac{e^{A_\epsilon^2 - C_\epsilon^2}}{2} \right) + \left( x-c-u_a t \right) f(P_\epsilon) 
\\
+ \left( x-c \right) \left( \frac{x-a-u_a t}{a-x}\ f(A_\epsilon) - \frac{x-a-u_a t}{c-x}\ f(C_\epsilon)\ e^{A_\epsilon^2 - C_\epsilon^2} \right) \bigg] 
\\
+\ \rho_d\ \frac{\left( x-a-u_a t \right) e^{B_\epsilon^2 - D_\epsilon^2}}{\left( d-x \right) e^{B_\epsilon^2 - A_\epsilon^2 + \frac{u_b}{\epsilon}}}\ f(D_\epsilon)
\end{aligned}
}{
\begin{aligned}
\frac{x-a - u_a t}{a-x}\ f(A_\epsilon) + \frac{\left( x-a-u_a t \right) \left( 1 - e^{- \frac{| u_b |}{\epsilon}} \right)}{\left( b-x \right) e^{B_\epsilon^2 - A_\epsilon^2 + \frac{u_b}{\epsilon}}}\ f(B_\epsilon) + f(P_\epsilon)
\end{aligned}
}
\end{pmatrix}^T
\\
= &\begin{cases}
\left( \frac{x-a}{t} , \rho_c\ \frac{2 \left( x-a-u_a t \right) \left( x-a \right) + \left( a-c \right) u_a t}{u_a t} \right),\ &{ x \in \Big( r(t) , l(t) \Big) },
\\
\left( \frac{x-b}{t} , 0 \right),\ &{ x \in \Big( \max{ \big\{ r(t) , l(t) \big\} , \infty \Big) } }.
\end{cases}
\end{aligned}
\]
\end{footnotesize}

\subsection*{Subregion  2. $x < a + u_a t$}

\begin{footnotesize}
\[
\begin{aligned}
&\lim_{\epsilon \rightarrow 0} 
\begin{pmatrix}
\frac{
\begin{aligned}
u_a \left( \sqrt{\pi} - {\textnormal{erfc}} \left( P_\epsilon \right) \right) + \frac{b-x}{2t}\ \frac{e^{- \frac{| u_b |}{\epsilon}} - 1}{2t\ B_\epsilon\ e^{P_\epsilon^2 + B_\epsilon^2 - A_\epsilon^2 + \frac{u_b}{\epsilon}}}
\end{aligned}
}{
\begin{aligned}
\frac{f(A_\epsilon)}{A_\epsilon\ e^{P_\epsilon^2}} + \frac{f(B_\epsilon) \left( 1 - e^{- \frac{| u_b |}{\epsilon}} \right)}{B_\epsilon\ e^{P_\epsilon^2 + B_\epsilon^2 - A_\epsilon^2 + \frac{u_b}{\epsilon}}} + \left( \sqrt{\pi} - {\textnormal{erfc}} \left( P_\epsilon \right) \right)
\end{aligned}
}
\\
-------------------------------
\\
\frac{
\begin{aligned}
\rho_c \bigg[ \sqrt{2 t \epsilon} \left( 1 - \frac{1}{2\ e^{C_\epsilon^2 - A_\epsilon^2}} \right) e^{- P_\epsilon^2} + (x-c-u_a t) \left( \sqrt{\pi} - {\textnormal{erfc}} \left( P_\epsilon \right) \right) \bigg] 
\\
+ \left( x-c \right) \left( \frac{f(A_\epsilon)}{A_\epsilon} - \frac{f(C_\epsilon)}{C_\epsilon\ e^{C_\epsilon^2 - A_\epsilon^2}} \right) e^{- P_\epsilon^2} + \rho_d\ \frac{\left( b-x \right) f(D_\epsilon)\ e^{B_\epsilon^2 - D_\epsilon^2}}{\left( d-x \right) B_\epsilon\ e^{P_\epsilon^2 + B_\epsilon^2 - A_\epsilon^2 + \frac{u_b}{\epsilon}}}
\end{aligned}
}{
\begin{aligned}
\frac{f(A_\epsilon)}{A_\epsilon\ e^{P_\epsilon^2}} + \frac{f(B_\epsilon) \left( 1 - e^{- \frac{| u_b |}{\epsilon}} \right)}{B_\epsilon\ e^{P_\epsilon^2 + B_\epsilon^2 - A_\epsilon^2 + \frac{u_b}{\epsilon}}} + \left( \sqrt{\pi} - {\textnormal{erfc}} \left( P_\epsilon \right) \right)
\end{aligned}
}
\end{pmatrix}^T
\\
= &\begin{cases}
\left( u_a , \rho_c \left( x-c-u_a t \right) \right),\ &{ { u_a \left( b-a \right) > u_b },\ { x \in \Big( -\infty , \min{ \big\{ q(t) , r(t) \big\} } \Big) } },
\\
\left( \frac{x-b}{t} , 0 \right),\ &{ { u_a \left( b-a \right) > u_b },\ { x \in \Big( q(t) , r(t) \Big) } },
\\
\left( u_a , \rho_c \left( x-c-u_a t \right) \right),\ &{ u_a \left( b-a \right) \leq u_b,\ x \in \Big( -\infty , r(t) \Big) }.
\end{cases}
\end{aligned}
\]
\end{footnotesize}
	
In the remaining regions, we have $x > a + u_a t$ because of the restriction $u_a < 0$. Therefore the required limit $\lim_{\epsilon \rightarrow 0} \left( u^\epsilon , R^\epsilon \right)$	 will be evaluated as follows:

	\item $a < x < c$
	
\begin{footnotesize}
\[
\begin{aligned}
&\lim_{\epsilon \rightarrow 0} 
\begin{pmatrix}
\frac{
\begin{aligned}
u_a\ \frac{f(P_\epsilon)}{P_\epsilon\ e^{A_\epsilon^2}} + \frac{(b-x) \left( e^{- \frac{| u_b |}{\epsilon}} - 1 \right)}{2t\ B_\epsilon\ e^{B_\epsilon^2 + \frac{u_b}{\epsilon}}}
\end{aligned}
}{
\begin{aligned}
\left( \sqrt{\pi} - {\textnormal{erfc}} \left( A_\epsilon \right) \right) + \frac{f(B_\epsilon) \left( 1 - e^{- \frac{| u_b |}{\epsilon}} \right)}{B_\epsilon\ e^{B_\epsilon^2 + \frac{u_b}{\epsilon}}} + \frac{f(P_\epsilon)}{P_\epsilon\ e^{A_\epsilon^2}}
\end{aligned}
}
\\
-----------------------------
\\
\frac{
\begin{aligned}
\rho_c \bigg[ \sqrt{2 t \epsilon} \left( \frac{1}{e^{A_\epsilon^2}} - \frac{1}{2\ e^{C_\epsilon^2}} \right) + \frac{\left( x-c-u_a t \right) f(P_\epsilon)}{P_\epsilon\ e^{A_\epsilon^2}} \bigg]
\\
+ \left( x-c \right) \left( \sqrt{\pi} - {\textnormal{erfc}} \left( A_\epsilon \right) - {\textnormal{erfc}} \left( C_\epsilon \right) \right) + \rho_d\ \frac{f(D_\epsilon) \left( b-x \right) e^{B_\epsilon^2 - D_\epsilon^2}}{\left( d-x \right) B_\epsilon\ e^{B_\epsilon^2 + \frac{u_b}{\epsilon}}}
\end{aligned}
}{
\begin{aligned}
\left( \sqrt{\pi} - {\textnormal{erfc}} \left( A_\epsilon \right) \right) + \frac{f(B_\epsilon) \left( 1 - e^{- \frac{| u_b |}{\epsilon}} \right)}{B_\epsilon\ e^{B_\epsilon^2 + \frac{u_b}{\epsilon}}} + \frac{f(P_\epsilon)}{P_\epsilon\ e^{A_\epsilon^2}}
\end{aligned}
}
\end{pmatrix}^T
\\
= &\begin{cases}
\left( 0 , \rho_c \left( x-c \right) \right),\ &{ x < b - \sqrt{- 2u_b t} },
\\
\left( \frac{x-b}{t} , 0 \right),\ &{ x > b - \sqrt{- 2 u_b t} }.
\end{cases}
\end{aligned}
\]
\end{footnotesize}

	\item $c < x < b$
	
\begin{footnotesize}
	\[
\begin{aligned}
&\lim_{\epsilon \rightarrow 0} 
\begin{pmatrix}
\frac{
\begin{aligned}
u_a\ \frac{f(P_\epsilon)}{P_\epsilon\ e^{A_\epsilon^2}} + \frac{(b-x) \left( e^{- \frac{| u_b |}{\epsilon}} - 1 \right)}{2t\ B_\epsilon\ e^{B_\epsilon^2 + \frac{u_b}{\epsilon}}}
\end{aligned}
}{
\begin{aligned}
\left( \sqrt{\pi} - {\textnormal{erfc}} \left( A_\epsilon \right) \right) + \frac{f(B_\epsilon) \left( 1 - e^{- \frac{| u_b |}{\epsilon}} \right)}{B_\epsilon\ e^{B_\epsilon^2 + \frac{u_b}{\epsilon}}} + \frac{f(P_\epsilon)}{P_\epsilon\ e^{A_\epsilon^2}}
\end{aligned}
}
\\
--------------------------
\\
\frac{
\begin{aligned}
\rho_c \bigg[ \sqrt{2 t \epsilon} \left( \frac{1}{e^{A_\epsilon^2}} - \frac{1}{2\ e^{C_\epsilon^2}} \right) + \frac{\left( x-c-u_a t \right) f(P_\epsilon)}{P_\epsilon\ e^{A_\epsilon^2}} \bigg]
\\
+ \left( x-c \right) \left( {\textnormal{erfc}} \left( C_\epsilon \right) - {\textnormal{erfc}} \left( A_\epsilon \right) \right) + \rho_d\ \frac{f(D_\epsilon) \left( b-x \right) e^{B_\epsilon^2 - D_\epsilon^2}}{\left( d-x \right) B_\epsilon\ e^{B_\epsilon^2 + \frac{u_b}{\epsilon}}}
\end{aligned}
}{
\begin{aligned}
\left( \sqrt{\pi} - {\textnormal{erfc}} \left( A_\epsilon \right) \right) + \frac{f(B_\epsilon) \left( 1 - e^{- \frac{| u_b |}{\epsilon}} \right)}{B_\epsilon\ e^{B_\epsilon^2 + \frac{u_b}{\epsilon}}} + \frac{f(P_\epsilon)}{P_\epsilon\ e^{A_\epsilon^2}}
\end{aligned}
}
\end{pmatrix}^T
\\
= &\begin{cases}
\left( 0 , 0 \right),\ &{ x < b - \sqrt{- 2u_b t} },
\\
\left( \frac{x-b}{t} , 0 \right),\ &{ x > b - \sqrt{- 2 u_b t} }.
\end{cases}
\end{aligned}
\]
	\end{footnotesize}	

	\item $b < x < d$
	
\begin{footnotesize}
	\[
\begin{aligned}
&\lim_{\epsilon \rightarrow 0} 
\begin{pmatrix}
\frac{
\begin{aligned}
u_a\ \frac{f(P_\epsilon)}{P_\epsilon\ e^{A_\epsilon^2 + \frac{| u_b |}{\epsilon}}} + \frac{(x-b) \left( e^{- \frac{| u_b |}{\epsilon}} - 1 \right)}{2t\ B_\epsilon\ e^{B_\epsilon^2}}
\end{aligned}
}{
\begin{aligned}
\sqrt{\pi} - {\textnormal{erfc}} \left( A_\epsilon \right) e^{- \frac{| u_b |}{\epsilon}} + {\textnormal{erfc}} \left( B_\epsilon \right) \left( e^{- \frac{| u_b |}{\epsilon}} - 1 \right) + \frac{f(P_\epsilon)}{P_\epsilon\ e^{A_\epsilon^2 + \frac{| u_b |}{\epsilon}}}
\end{aligned}
}
\\
----------------------------
\\
\frac{
\begin{aligned}
\rho_c\ \bigg[ \sqrt{2 t \epsilon} \left( \frac{1}{e^{A_\epsilon^2 + \frac{| u_b |}{\epsilon}}} - \frac{1}{2\ e^{C_\epsilon^2 + \frac{| u_b |}{\epsilon}}} \right) + \left( x-c-u_a t \right) \frac{f(P_\epsilon)}{P_\epsilon\ e^{A_\epsilon^2 + \frac{| u_b |}{\epsilon}}} 
\\
+ \left( x-c \right) \left( {\textnormal{erfc}} \left( C_\epsilon \right) - {\textnormal{erfc}} \left( A_\epsilon \right) \right) e^{- \frac{| u_b |}{\epsilon}} \bigg] + \rho_d\ {\textnormal{erfc}} \left( D_\epsilon \right)
\end{aligned}
}{
\begin{aligned}
\sqrt{\pi} - {\textnormal{erfc}} \left( A_\epsilon \right) e^{- \frac{| u_b |}{\epsilon}} + {\textnormal{erfc}} \left( B_\epsilon \right) \left( e^{- \frac{| u_b |}{\epsilon}} - 1 \right) + \frac{f(P_\epsilon)}{P_\epsilon\ e^{A_\epsilon^2 + \frac{| u_b |}{\epsilon}}}
\end{aligned}
}
\end{pmatrix}^T
\\
= &\left( 0 , 0 \right).
\end{aligned}
\]
	\end{footnotesize}	

	\item $x > d$ 
	
\begin{footnotesize}
	\[
\begin{aligned}
&\lim_{\epsilon \rightarrow 0} 
\begin{pmatrix}
\frac{
\begin{aligned}
u_a\ \frac{f(P_\epsilon)}{P_\epsilon\ e^{A_\epsilon^2 + \frac{| u_b |}{\epsilon}}} + \frac{(x-b) \left( e^{- \frac{| u_b |}{\epsilon}} - 1 \right)}{2t\ B_\epsilon\ e^{B_\epsilon^2}}
\end{aligned}
}{
\begin{aligned}
\sqrt{\pi} - {\textnormal{erfc}} \left( A_\epsilon \right) e^{- \frac{| u_b |}{\epsilon}} + {\textnormal{erfc}} \left( B_\epsilon \right) \left( e^{- \frac{| u_b |}{\epsilon}} - 1 \right) + \frac{f(P_\epsilon)}{P_\epsilon\ e^{A_\epsilon^2 + \frac{| u_b |}{\epsilon}}}
\end{aligned}
}
\\
----------------------------
\\
\frac{
\begin{aligned}
\rho_c\ \bigg[ \sqrt{2 t \epsilon} \left( \frac{1}{e^{A_\epsilon^2 + \frac{| u_b |}{\epsilon}}} - \frac{1}{2\ e^{C_\epsilon^2 + \frac{| u_b |}{\epsilon}}} \right) + \left( x-c-u_a t \right) \frac{f(P_\epsilon)}{P_\epsilon\ e^{A_\epsilon^2 + \frac{| u_b |}{\epsilon}}} 
\\
+ \left( x-c \right) \left( {\textnormal{erfc}} \left( C_\epsilon \right) - {\textnormal{erfc}} \left( A_\epsilon \right) \right) e^{- \frac{| u_b |}{\epsilon}} \bigg] + \rho_d \left( \sqrt{\pi} - {\textnormal{erfc}} \left( D_\epsilon \right) \right)
\end{aligned}
}{
\begin{aligned}
\sqrt{\pi} - {\textnormal{erfc}} \left( A_\epsilon \right) e^{- \frac{| u_b |}{\epsilon}} + {\textnormal{erfc}} \left( B_\epsilon \right) \left( e^{- \frac{| u_b |}{\epsilon}} - 1 \right) + \frac{f(P_\epsilon)}{P_\epsilon\ e^{A_\epsilon^2 + \frac{| u_b |}{\epsilon}}}
\end{aligned}
}
\end{pmatrix}^T
\\
= &\left( 0 , \rho_d \right).
\end{aligned}
\]
	\end{footnotesize}	

\end{enumerate}

For recovering $\rho$, set $R = \lim_{\epsilon \rightarrow 0} R^\epsilon$. The explicit structure of $R$ under this case is given by
	
\[
R = \begin{cases}
\rho_c \left( x-c-u_a t \right),\ &{ x \in \Big( -\infty , \gamma_{ {a,1} } (t) \Big) },
\\
\rho_c\ \frac{2 \left( x-a-u_a t \right) \left( x-a \right) + \left( a-c \right) u_a t}{u_a t},\ &{ x \in \Big( \gamma_{ {a,1} } (t) , \gamma_{ {a,2} } (t) \Big) },
\\
\rho_c \left( x-c \right),\ &{ x \in \Big( \gamma_{ {a,2} } (t) , \gamma_{ {c} } (t) \Big) },
\\
0,\ &x \in { \big( \gamma_{ {c} } (t) , \gamma_{ {d} } (t) \big) },
\\
\rho_d,\ &{ x \in \Big( \gamma_{ {d} } (t) , \infty \Big) },
\end{cases}
\]
where $\gamma_{ {d} } (s) := d$ for every $s \geq 0$ and the curves $x = \gamma_{ {a,1} } (t)$, $x = \gamma_{ {a,2} } (t)$ and $x = \gamma_{ {c} } (t)$ are defined as follows:

\begin{itemize}

	\item If $u_a \left( b-a \right) \leq u_b$, then

\[
\begin{aligned}
\gamma_{ {a,1} } (s) &:= a + u_a s,\ { s \geq 0 },
\\
\gamma_{a,2} (s) &:= \begin{cases}
a,\ &{ 0 \leq s \leq t_{ {}_{a,l} } := \frac{(b-a)^2}{2 | u_b |} },
\\
\frac{a+b}{2} + \frac{u_b}{b-a} s,\ &{ s \geq t_{ {}_{a,l} }} ,
\end{cases}
\\
\gamma_{c} (s) &:= \begin{cases}
c,\ &{ 0 \leq s \leq t_{ {}_{c,p} } := \frac{(b-c)^2}{2 | u_b |} },
\\
\frac{a+b}{2} + \frac{u_b}{b-a} s,\ &{ s \geq t_{ {}_{c,p} } }.
\end{cases}
\end{aligned}
\]	
	\item If $u_a \left( b-a \right) > u_b$, then

\[
\begin{aligned}
\gamma_{a,1} (s) &:= \begin{cases}
a + u_a s,\ &{ 0 \leq s \leq t_{ {}_{r,q} } := \frac{(b-a)^2}{2 \left( u_a \left( b-a \right) - u_b \right)} },
\\
b + u_a s - \sqrt{2 \left( u_a \left( b-a \right) - u_b \right) s},\ &{ s \geq t_{ {}_{r,q} } },
\end{cases}
\\
\gamma_{a,2} (s) &:= \begin{cases}
a,\ &{ 0 \leq s \leq t_{ {}_{a,l} } },
\\
\frac{a+b}{2} + \frac{u_b}{b-a} s,\ &{ t_{ {}_{a,l} } \leq s \leq t_{ {}_{l,q} } := \frac{(b-a)^2}{2 \left( u_a \left( b-a \right) - u_b \right)} },
\\
b + u_a s - \sqrt{2 \left( u_a \left( b-a \right) - u_b \right) s},\ &{ s \geq t_{ {}_{l,q} } },
\end{cases}
\\
\gamma_{c} (s) &:= \begin{cases}
c,\ &{ 0 \leq s \leq t_{ {}_{c,p} } := \frac{(b-c)^2}{2 | u_b |} },
\\
\frac{a+b}{2} + \frac{u_b}{b-a} s,\ &{ s \geq t_{ {}_{c,p} } }.
\end{cases}
\end{aligned}
\]

\end{itemize}
In either case, we see that any test function $\phi \in C^{\infty}_{c} \left( \textbf{R}^1 \times \left[ 0 , \infty \right) ; \textbf{R}^1 \right)$ will satisfy

\[
\begin{aligned}
\left\langle R_x , \phi \right\rangle = &- \left\langle R , \phi_x \right\rangle
\\
= &- \rho_c \int_{0}^{\infty} \Bigg[ \int_{-\infty}^{\gamma_{a,1} (t)} \left( x-c-u_a t \right) \phi_x\ dx + \int_{\gamma_{a,1} (t)}^{\gamma_{a,2} (t)} \frac{2 \left( x-a-u_a t \right) \left( x-a \right) + \left( a-c \right) u_a t}{u_a t} \phi_x\ dx
\\
&+ \int_{\gamma_{a,2} (t)}^{\gamma_{c} (t)} \left( x-c \right) \phi_x\ dx \Bigg] dt - \rho_d \Bigg[ \int_{0}^{\infty} \int_{d}^{\infty}\ \phi_x\ dx dt \Bigg]
\\
= & \rho_c \Bigg[ \int_{0}^{\infty} \int_{-\infty}^{\gamma_{a,1} (t)} \phi \left( x , t \right) dx dt + \int_{0}^{\infty} \left( c+u_a t - \gamma_{a,1} (t) \right) \phi \left( \gamma_{a,1} (t) , t \right) dt
\\
&+ \int_{0}^{\infty} \int_{\gamma_{a,1} (t)}^{\gamma_{a,2} (t)} \left( \frac{4 \left( x-a \right)}{u_a t} - 2 \right) \phi \left( x , t \right) dx dt
\\
&+ \int_{0}^{\infty} \left( \gamma_{a,2} (t) - 2a + c - 2 - \frac{2 \left( \gamma_{a,2} (t) - a \right)^2}{u_a t} \right) \phi \left( \gamma_{a,2} (t) , t \right) dt 
\\
&- \int_{0}^{\infty} \left( \gamma_{a,1} (t) - 2a + c - 2 - \frac{2 \left( \gamma_{a,1} (t) - a \right)^2}{u_a t} \right) \phi \left( \gamma_{a,1} (t) , t \right) dt
\\
&+ \int_{0}^{\infty} \int_{\gamma_{a,2} (t)}^{\gamma_{c} (t)} \phi \left( x , t \right) dx dt + \int_{0}^{\infty} \left( \gamma_{a,2} (t) - \gamma_{c} (t) \right) \phi \left( \gamma_{a,2} (t) , t \right) dt 
\\
&- \int_{0}^{\infty} \left( \gamma_{a,1} (t) - \gamma_{c} (t) \right) \phi \left( \gamma_{a,1} (t) , t \right) dt \Bigg] + \rho_d \int_{0}^{\infty} \phi \left( d , t \right) dt
\\
= &\left\langle \rho_c \Bigg[ \chi_{ {}_{ \left( -\infty , \gamma_{ {}_{ a,1 } } (t) \right) } } + \left( \frac{4 \left( x-a \right)}{u_a t} - 2 \right) \chi_{ {}_{ \left( \gamma_{ {}_{ a,1 } } (t) , \gamma_{ {}_{ a,2 } } (t) \right) } } + \chi_{ {}_{ \left( \gamma_{ {}_{ a,2 } } (t) , \gamma_{ {}_{ c } } (t) \right) } } \right. 
\\
&\left. + \left( 2 \left( c-a \right) + u_a t + \gamma_{ {}_{ c } } (t) - x - 2 - \frac{2 \left( x-a \right)^2}{u_a t} \right) \delta_{ {}_{ x = \gamma_{ {}_{ a,1 } } (t) } } \right.
\\
&\left. + \left( 2 \left( x-a \right) + c - \gamma_{ {}_{ c } } (t) - 2 - \frac{2 \left( x-a \right)^2}{u_a t} \right) \delta_{ {}_{ x = \gamma_{ {}_{ a,2 } } (t) } } \Bigg] + \rho_d\ \delta_{ x = \gamma_{ {d} } (t) } , \phi \right\rangle.
\end{aligned}
\]
Therefore $\rho$ is given by

\[
\begin{aligned}
\rho = &\rho_c \Bigg[ \chi_{ {}_{ \left( -\infty , \gamma_{ {}_{ a,1 } } (t) \right) } } + \left( \frac{4 \left( x-a \right)}{u_a t} - 2 \right) \chi_{ {}_{ \left( \gamma_{ {}_{ a,1 } } (t) , \gamma_{ {}_{ a,2 } } (t) \right) } } + \chi_{ {}_{ \left( \gamma_{ {}_{ a,2 } } (t) , \gamma_{ {}_{ c } } (t) \right) } }
\\
&+ \left( 2 \left( c-a \right) + u_a t + \gamma_{ {}_{ c } } (t) - x - 2 - \frac{2 \left( x-a \right)^2}{u_a t} \right) \delta_{ {}_{ x = \gamma_{ {}_{ a,1 } } (t) } } 
\\
&+ \left( 2 \left( x-a \right) + c - \gamma_{ {}_{ c } } (t) - 2 - \frac{2 \left( x-a \right)^2}{u_a t} \right) \delta_{ {}_{ x = \gamma_{ {}_{ a,2 } } (t) } } \Bigg] + \rho_d\ \delta_{ x = \gamma_{ {d} } (t) }.
\end{aligned}
\]

\end{proof}

\section*{Appendix}

\begin{enumerate}

	\item Let us prove the following properties of the function erfc:

\begin{enumerate}

	\item For every $z \in \textbf{R}^1$, $\textnormal{erfc} \left( z \right) + \textnormal{erfc} \left( -z \right) = \sqrt{\pi}$
	
	\item $\lim_{z \rightarrow \infty} \textnormal{erfc} \left( z \right) = 0$
	
	\item $\textnormal{erfc} \left( z \right) = \left( \frac{1}{2 z} - \frac{1}{4 z^3} + o \left( \frac{1}{z^3} \right) \right)\ e^{- z^2} \textnormal{ as } z \rightarrow \infty$
	
	\item $\lim_{z \rightarrow \infty}\ z\ \textnormal{erfc} \left( z \right) e^{z^2} = \frac{1}{2}$

\end{enumerate}
For proving the first property, we fix $z \in \textbf{R}^1$ and proceed as follows:

\[
\begin{aligned}
\textnormal{erfc} \left( -z \right) = \int_{-z}^{\infty} e^{-y^2}\ dy &= \int_{-\infty}^{\infty} e^{-y^2}\ dy - \int_{-\infty}^{-z} e^{-y^2}\ dy
\\
&= \sqrt{\pi} - \int_{z}^{\infty} e^{-\eta^2}\ d\eta
\\
&= \sqrt{\pi} - \textnormal{erfc} \left( z \right).
\end{aligned}
\]
The second property follows from the definition of erfc: 

\[
\lim_{z \rightarrow \infty} \textnormal{erfc} \left( z \right) = \lim_{z \rightarrow \infty} \int_{z}^{\infty}\ e^{-s^2}\ ds = 0.
\]
We now verify the third and fourth properties. For any $z > 1$, we can integrate by parts to get

\[
\begin{aligned}
\textnormal{erfc} \left( z \right) 
		&= \int_{z}^{\infty}\ \left( - \frac{1}{2t} \right)\ \frac{d}{dt} \left( e^{-t^2} \right)\ dt
\\
		&= \frac{1}{2z}\ e^{-z^2} + \int_{z}^{\infty}\ \frac{1}{4 t^3}\ \frac{d}{dt} \left( e^{-t^2} \right)\ dt
\\
		&= \left( \frac{1}{2z} - \frac{1}{4 z^3} \right)\ e^{- z^2} + \int_{z}^{\infty}\ \frac{3}{4 t^4}\ e^{-t^2}\ dt		,
\end{aligned}
\]
so that

\[
\begin{aligned}
\left| z^3\ \left[ e^{z^2}\ \textnormal{erfc} \left( z \right) - \left( \frac{1}{2z} - \frac{1}{4 z^3} \right) \right] \right| 
		&\leq - \frac{3}{8 z^2}\ \int_{z}^{\infty}\ \frac{d}{dt}\ \left( e^{z^2 - t^2} \right)\ dt \leq \frac{3}{8 z^2},		
\end{aligned}
\]
and since $\lim_{z \rightarrow \infty} \frac{3}{8 z^2} = 0$, this proves our claim. This last inequality also implies that

\[
z^2\ \left| z\ \textnormal{erfc} \left( z \right)\ e^{z^2} - \frac{1}{2} \right| \rightarrow \frac{1}{4} \textnormal{ as } z \rightarrow \infty,
\]
and hence we conclude the proof of the fourth property $\lim_{z \rightarrow \infty}\ z\ \textnormal{erfc} \left( z \right) e^{z^2} = \frac{1}{2}$.

\item Let us justify the underlying computations for the case $u_a < 0$, $u_b > 0$ in the region $x<a$:

\begin{itemize}

	\item $\lim_{\epsilon \rightarrow 0} e^{- \frac{| u_a |}{\epsilon}} = \lim_{\epsilon \rightarrow 0} e^{- \frac{| u_b |}{\epsilon}} = 0$
	
	\item \textbf{Simplification of $\lim_{\epsilon \rightarrow 0} A_\epsilon\ e^{{A_\epsilon}^2 + \frac{u_a}{\epsilon}} $:}
	
\[
\begin{aligned}
\lim_{\epsilon \rightarrow 0} A_\epsilon\ e^{{A_\epsilon}^2 + \frac{u_a}{\epsilon}} 
	&= \lim_{\epsilon \rightarrow 0} \frac{a-x}{\sqrt{2 t \epsilon}}\ e^{\frac{(a-x)^2 + 2 u_a t}{2 t \epsilon}} = \begin{cases}
	\infty,\ &{x < a - \sqrt{- 2 u_a t}},
	\\
	0,\ &{x > a - \sqrt{- 2 u_a t}};		
	\end{cases}
\end{aligned}
\]	
	
	\item \textbf{Simplification of $\lim_{\epsilon \rightarrow 0} \textnormal{erfc} \left( A_\epsilon \right) e^{\frac{| u_a |}{\epsilon}}$:}
	
\[
\begin{aligned}
\lim_{\epsilon \rightarrow 0} \textnormal{erfc} \left( A_\epsilon \right) e^{\frac{| u_a |}{\epsilon}} 
	&= \lim_{\epsilon \rightarrow 0} \frac{f(A_\epsilon)}{A_\epsilon\ e^{{A_\epsilon}^2 + \frac{u_a}{\epsilon}}} = \begin{cases}
	0,\ &{x < a - \sqrt{- 2 u_a t}},
	\\
	\infty,\ &{x > a - \sqrt{- 2 u_a t}}; 		
	\end{cases}
\end{aligned}
\]	
	
	\item $\lim_{\epsilon \rightarrow 0} A_\epsilon\ e^{{A_\epsilon}^2 + \frac{u_a}{\epsilon}} \cdot \textnormal{erfc} \left( A_\epsilon \right) e^{\frac{| u_a |}{\epsilon}} 
	= \lim_{\epsilon \rightarrow 0} f \left( A_\epsilon \right) = \frac{1}{2}$
	
	\item $\lim_{\epsilon \rightarrow 0} \frac{A_\epsilon\ e^{{A_\epsilon}^2 + \frac{u_a}{\epsilon}}}{B_\epsilon\ e^{{B_\epsilon}^2 + \frac{u_a}{\epsilon}}}
	= \lim_{\epsilon \rightarrow 0} \frac{a-x}{b-x} \cdot e^{\frac{(a-x)^2 - (b-x)^2}{2 t \epsilon}} = 0$
	
	\item $\lim_{\epsilon \rightarrow 0} A_\epsilon\ e^{{A_\epsilon}^2 + \frac{u_a}{\epsilon}} \cdot \textnormal{erfc} \left( B_\epsilon \right) e^{\frac{| u_a |}{\epsilon}} 
	= \lim_{\epsilon \rightarrow 0} \frac{a-x}{b-x} \cdot f(B_\epsilon) \cdot e^{\frac{(a-x)^2 - (b-x)^2}{2 t \epsilon}} = 0$

\end{itemize}
The last computation combined with the inequalities $0 < a-x < c-x < d-x$ implies that

\[
\begin{aligned}
\lim_{\epsilon \rightarrow 0} A_\epsilon\ e^{{A_\epsilon}^2 + \frac{u_a}{\epsilon}} \cdot \textnormal{erfc} \left( C_\epsilon \right) e^{\frac{| u_a |}{\epsilon}} 
	&= \lim_{\epsilon \rightarrow 0} \frac{a-x}{c-x} \cdot f(B_\epsilon) \cdot e^{\frac{(a-x)^2 - (c-x)^2}{2 t \epsilon}} = 0,
\\
\\
\lim_{\epsilon \rightarrow 0} A_\epsilon\ e^{{A_\epsilon}^2 + \frac{u_a}{\epsilon}} \cdot \textnormal{erfc} \left( D_\epsilon \right) e^{\frac{| u_a |}{\epsilon}} 
	&= \lim_{\epsilon \rightarrow 0} \frac{a-x}{d-x} \cdot f(B_\epsilon) \cdot e^{\frac{(a-x)^2 - (d-x)^2}{2 t \epsilon}} = 0.			
\end{aligned}
\]

	\item Here we provide the details of the computations involved in getting the explicit expressions for $u^\epsilon$ and $R^\epsilon$. Let us recall that $u^\epsilon = \frac{V^\epsilon_x}{V^\epsilon}$ and $R^\epsilon = \frac{S^\epsilon}{V^\epsilon}$, where

\[
\begin{aligned}
V^\epsilon_t &= \frac{\epsilon}{2} V^\epsilon_{xx},\ S^\epsilon_t = \frac{\epsilon}{2} S^\epsilon_{xx},
\\
V^\epsilon (x,0) &= \begin{cases}
e^{- \frac{u_a (x-a)}{\epsilon}},\ &x<a,
\\
1,\ &a<x<b,
\\
e^{- \frac{u_b}{\epsilon}},\ &x>b,
\end{cases}
\\
S^\epsilon (x,0) &= \begin{cases}
\rho_c (x-c)\ e^{- \frac{u_a (x-a)}{\epsilon}},\ &x<a,
\\
\rho_c (x-c),\ &a<x<c,
\\
0,\ &c<x<d,
\\
\rho_d\ e^{- \frac{u_b}{\epsilon}},\ &x>d.
\end{cases}
\end{aligned}
\]	
Therefore

\[
\begin{aligned}
V^\epsilon (x,t) = &\frac{1}{\sqrt{2 \pi t \epsilon}} \left[ \int_{-\infty}^{a}\ e^{- \frac{u_a (y-a)}{\epsilon}} \cdot e^{- \frac{(y-x)^2}{2 t \epsilon}}\ dy + \int_{a}^{b}\ e^{- \frac{(y-x)^2}{2 t \epsilon}}\ dy \right. 
\\
+ &\left. \int_{b}^{\infty}\ e^{- \frac{u_b}{\epsilon}} \cdot e^{- \frac{(y-x)^2}{2 t \epsilon}}\ dy \right],
\\
S^\epsilon (x,t) = &\frac{1}{\sqrt{2 \pi t \epsilon}} \left[ \int_{-\infty}^{a}\ \rho_c (y-c)\ e^{- \frac{u_a (y-a)}{\epsilon}} \cdot e^{- \frac{(y-x)^2}{2 t \epsilon}}\ dy \right. 
\\
+ &\left. \int_{a}^{c}\ \rho_c (y-c)\ e^{- \frac{(y-x)^2}{2 t \epsilon}}\ dy + \int_{d}^{\infty}\ \rho_d\ e^{- \frac{u_b}{\epsilon}} \cdot e^{- \frac{(y-x)^2}{2 t \epsilon}}\ dy \right].
\end{aligned}
\]
These integrals can be further simplified as follows:

\begin{itemize}

	\item \textbf{Simplification of $\int_{-\infty}^{a}\ e^{- \frac{u_a (y-a)}{\epsilon}} \cdot e^{- \frac{(y-x)^2}{2 t \epsilon}}\ dy$:}
	
\[
\begin{aligned}
\int_{-\infty}^{a}\ e^{- \frac{u_a (y-a)}{\epsilon}} \cdot e^{- \frac{(y-x)^2}{2 t \epsilon}}\ dy &= e^{\frac{(u_a t)^2 - 2 u_a t (x-a)}{2 t \epsilon}} \int_{-\infty}^{a}\ e^{- \frac{(y-x+u_at)^2}{2 t \epsilon}}\ dy
\\
&= \sqrt{2 t \epsilon}\ e^{\frac{(u_a t)^2 - 2 u_a t (x-a)}{2 t \epsilon}} \textnormal{erfc} \left( \frac{x-a-u_at}{\sqrt{2 t \epsilon}} \right)
\end{aligned}
\]	
	
	\item $\int_a^b\ e^{- \frac{(y-x)^2}{2 t \epsilon}}\ dy = \sqrt{2 t \epsilon}\ \int_{\frac{a-x}{\sqrt{2 t \epsilon}}}^{\frac{b-x}{\sqrt{2 t \epsilon}}}\ e^{- z^2}\ dz = \sqrt{2 t \epsilon} \left( \textnormal{erfc} \left( \frac{x-b}{\sqrt{2 t \epsilon}} \right) - \textnormal{erfc} \left( \frac{x-a}{\sqrt{2 t \epsilon}} \right) \right)$ 
	
	\item $\int_{b}^{\infty}\ e^{- \frac{(y-x)^2}{2 t \epsilon}}\ dy = \sqrt{2 t \epsilon}\ \int_{\frac{b-x}{\sqrt{2 t \epsilon}}}^{\infty}\ e^{- z^2}\ dz = \sqrt{2 t \epsilon}\ \textnormal{erfc} \left( - \frac{x-b}{\sqrt{2 t \epsilon}} \right)$ 
	
	\item \textbf{Simplification of $\int_{-\infty}^{a}\ (y-c)\ e^{- \frac{u_a (y-a)}{\epsilon}} \cdot e^{- \frac{(y-x)^2}{2 t \epsilon}}\ dy$:}
\[
\begin{aligned}
&\int_{-\infty}^{a}\ (y-c)\ e^{- \frac{u_a (y-a)}{\epsilon}} \cdot e^{- \frac{(y-x)^2}{2 t \epsilon}}\ dy 
\\
&= \int_{-\infty}^{a}\ ( y-x+u_at + x-c-u_at )\ e^{- \frac{(y-x)^2 + 2 u_a t (y-x+x-a)}{2 t \epsilon}}\ dy
\\
&=\ e^{\frac{(x-a-u_at)^2 - (x-a)^2}{2 t \epsilon}} \left[ \int_{-\infty}^{a}\ ( y-x+u_at)\ e^{- \frac{(y-x+u_at)^2}{2 t \epsilon}}\ dy \right.
\\
&+ \left. \int_{-\infty}^{a}\ (x-c-u_at)\ e^{- \frac{(y-x+u_at)^2}{2 t \epsilon}}\ dy \right]
\\
&=\ e^{\frac{(x-a-u_at)^2 - (x-a)^2}{2 t \epsilon}} \left[ \sqrt{2 t \epsilon} \int_{-\infty}^{a}\ \frac{y-x+u_at}{\sqrt{2 t \epsilon}} \cdot e^{- \frac{(y-x+u_at)^2}{2 t \epsilon}}\ dy \right.
\\
&+ (x-c-u_at) \left. \int_{-\infty}^{a}\ e^{- \frac{(y-x+u_at)^2}{2 t \epsilon}}\ dy \right]
\\
&=\ e^{\frac{(x-a-u_at)^2 - (x-a)^2}{2 t \epsilon}} \left[ t \epsilon \int_{\frac{x-a-u_at}{\sqrt{2 t \epsilon}}}^{\infty}\ \frac{d}{dz} \left( - e^{-z^2} \right)\ dz \right.
\\
&+ \sqrt{2 t \epsilon}\ (x-c-u_at) \left. \int_{\frac{x-a-u_at}{\sqrt{2 t \epsilon}}}^{\infty}\ e^{- z^2}\ dz \right]
\\
&=\ e^{\frac{(x-a-u_at)^2 - (x-a)^2}{2 t \epsilon}} \left[ t \epsilon\ e^{- \frac{(x-a-u_at)^2}{\sqrt{2 t \epsilon}}} + \sqrt{2 t \epsilon}\ (x-c-u_at)\ \textnormal{erfc} \left( \frac{x-a-u_at}{\sqrt{2 t \epsilon}} \right) \right]
\end{aligned}
\]
	
	\item \textbf{Simplification of $\int_{a}^{c}\ (y-c)\ e^{- \frac{(y-x)^2}{2 t \epsilon}}\ dy$:}
\[
\begin{aligned}
&\int_{a}^{c}\ (y-c)\ e^{- \frac{(y-x)^2}{2 t \epsilon}}\ dy 
\\
&= \int_a^c\ ( y-x + x-c )\ e^{- \frac{(y-x)^2}{2 t \epsilon}}\ dy 
\\
&= \sqrt{2 t \epsilon} \int_a^c\ \frac{y-x}{\sqrt{2 t \epsilon}} \cdot e^{- \frac{(y-x)^2}{2 t \epsilon}}\ dy + (x-c) \int_{a}^{c}\ e^{- \frac{(y-x)^2}{2 t \epsilon}}\ dy
\\
&= t \epsilon \int_{\frac{a-x}{\sqrt{2 t \epsilon}}}^{\frac{c-x}{\sqrt{2 t \epsilon}}}\ \frac{d}{dz} \left( - e^{- z^2} \right)\ dz + \sqrt{2 t \epsilon}\ (x-c) \int_{\frac{a-x}{\sqrt{2 t \epsilon}}}^{\frac{c-x}{\sqrt{2 t \epsilon}}}\ e^{- z^2}\ dz 
\\
&= t \epsilon \left( e^{- \frac{(x-a)^2}{2 t \epsilon}} - e^{- \frac{(x-c)^2}{2 t \epsilon}} \right) + \sqrt{2 t \epsilon}\ (x-c) \left( \textnormal{erfc} \left( \frac{x-c}{\sqrt{2 t \epsilon}} \right) - \textnormal{erfc} \left( \frac{x-a}{\sqrt{2 t \epsilon}} \right) \right)
\end{aligned}
\]

	\item $\int_{d}^{\infty}\ e^{- \frac{(y-x)^2}{2 t \epsilon}}\ dy = \sqrt{2 t \epsilon}\ \int_{\frac{d-x}{\sqrt{2 t \epsilon}}}^{\infty}\ e^{- z^2}\ dz = \sqrt{2 t \epsilon}\ \textnormal{erfc} \left( - \frac{x-d}{\sqrt{2 t \epsilon}} \right)$ 

\end{itemize}
	
\end{enumerate}

\section*{Acknowledgements}

The research was supported by the project "Basic research in physics and multidisciplinary sciences" (Identification \# RIN4001). The author is thankful to Professor K. T. Joseph for suggesting this problem. He also acknowledges the cooperation and support of Professor Anupam Pal Choudhury, Professor Agnid Banerjee and Professor K. T. Joseph during the preparation of this paper.

\end{document}